\documentclass[12pt,reqno]{amsart}
\usepackage{graphicx}
\usepackage{amssymb}
\usepackage{amsfonts}
\usepackage{amsmath,latexsym,amssymb,amsfonts,amsbsy, amsthm}
\usepackage[usenames]{color}
\usepackage{xspace,colortbl}
\usepackage{mathrsfs}
\usepackage[colorlinks,linkcolor=blue,citecolor=blue]{hyperref}
\usepackage[titletoc]{appendix}
\usepackage{pdfsync}
\usepackage{esint}
\allowdisplaybreaks

\newcommand{\beq}{\begin{equation}}
	\newcommand{\eeq}{\end{equation}}
\newcommand{\ben}{\begin{eqnarray}}
	\newcommand{\een}{\end{eqnarray}}
\newcommand{\beno}{\begin{eqnarray*}}
	\newcommand{\eeno}{\end{eqnarray*}}
\newcommand{\R}{\mathbb{R}}

\newtheorem{theorem}{Theorem}[section]
\newtheorem{lemma}[theorem]{Lemma}
\newtheorem{remark}[theorem]{Remark}
\newtheorem{proposition}[theorem]{Proposition}
\newtheorem{corollary}[theorem]{Corollary}
\newtheorem{definition}[theorem]{Definition}

\numberwithin{equation}{section}

\title[Entropy for supercritical Fujita equation]{F-stability, entropy and energy gap  for supercritical  Fujita equation}

\author[K. Wang]{Kelei Wang$^\dag$}
\address{$^\dag$School of Mathematics and Statistics \\ Wuhan University\\
	Wuhan 430072, China}
\email{wangkelei@whu.edu.cn}

\author[J. Wei]{Juncheng Wei$^\ast$}
\address{$^\dag$Department of Mathematics \\ Chinese University of Hong Kong\\
	Shatin, NT, Hong Kong}
\email{wei@math.cuhk.edu.hk}

\author[K. Wu]{Ke Wu$^\ast$}
\address{$\ast$ School of Mathematics and Statistics \\ Wuhan University\\
	Wuhan 430072, China}
\email{wukemail@whu.edu.cn}

\thanks{K. Wang is supported by  National Key R\&D Program of China (No. 2022YFA1005602) and the National Natural Science Foundation of China (No. 12131017 and No. 12221001).  J. Wei is partially supported  by  National Key R\&D Program of China (No. 2022YFA1005602), and Hong Kong General Research Fund "New frontiers in singular limits of nonlinear partial differential equations. K. Wu is supported by the China Postdoctoral Science Foundation (No. 2023M732712)}


\begin{document}

\begin{abstract}
We  study some problems on self similar solutions to the Fujita
equation when $p>(n+2)/(n-2)$, especially, the characterization of constant solutions by the energy.  Motivated by recent advances in mean curvature flows, we introduce the notion of $F-$functional, $F$-stability and entropy for solutions of supercritical Fujita equation.  Using these tools, we   prove that among bounded positive self similar solutions, the constant
solution has the lowest entropy. Furthermore, there is also a gap between the   entropy of constant and non-constant solutions.   As an application of these results,  we   prove that if $p>(n+2)/(n-2)$, then   the blow up set of type I blow up solutions   is the union of
a $(n-1)-$ rectifiable set  and a set of Hausdorff dimension at most $n-3$.
\end{abstract}

\maketitle

\tableofcontents

\section{Introduction}\label{sec introduction}
Consider the Cauchy problem
\begin{equation}\label{Cauchyproblem}
\left\{\begin{array}{lll}
\partial_{t}u=\Delta u+|u|^{p-1}u,\quad\text{in}\quad\mathbb{R}^{n}\times (0, T),\\
u(\cdot, 0)=u_{0},
\end{array}
\right.
\end{equation}
where $p>1$ and $u_{0}\in L^{\infty}(\mathbb{R}^{n})$. It is well known (see \cite{Fujita1966}) that solutions of \eqref{Cauchyproblem} may blow up in a finite time. Here a solution $u(x, t)$ is said to blow up in a finite time $T$ if $u(x, t)$ satisfies \eqref{Cauchyproblem} and
\[\limsup_{t\to T}\|u(\cdot,t)\|_\infty=\infty.\]
 A point $x_{0}$ is called a blow-up point if there exist sequences
 $\{x_{k}\}$ and $\{t_{k}\}$ such that
  \[\lim_{k\rightarrow\infty}x_{k}=x_{0}, \quad \lim_{k\rightarrow\infty}t_{k}=T,\quad
  \lim_{k\rightarrow\infty}|u(x_{k}, y_{k})|=+\infty.\] The set $\Sigma$ consisting of
  all the blow-up points is termed the blow-up set. For any $x_{0}\in\Sigma$, $\mathcal{T}(u, x_0)$ is the set of blow-up limits of $u$ at $x_0$. In the setting of geometric measure theory, $\mathcal{T}(u, (x_0, T))$  is usually called the set of tangent functions.

The finite time blow up is said to be of type I if
\[\limsup_{t\to T} (T-t)^{\frac{1}{p-1}}\|u(\cdot,t)\|_\infty<+\infty,\]
and of type II if
\[\limsup_{t\to T} (T-t)^{\frac{1}{p-1}}\|u(\cdot,t)\|_\infty=+\infty,\]
where $T$  is the maximal existence time of the $L^{\infty}$ solution $u$.

 In a series of papers, Giga and Kohn \cite{Giga-Kohn1985, Giga-Kohn1987, Giga-Kohn1989} studied the asymptotic behavior of blow-up solutions to \eqref{Cauchyproblem} when $1<p<p_{s}(n)$, where
   \[p_{s}(n)=\left\{\begin{array}{lll}
+\infty,\quad &\text{if}~1\leq n\leq 2,\\
\frac{n+2}{n-2}, &\text{if}~n\geq 3.
\end{array}
\right.\]
In order to analyzing \eqref{Cauchyproblem}, Giga and Kohn consided the self similar transform
\[w(y,\tau)=(T-t)^{\frac{1}{p-1}}u(x,t),\quad x=(T-t)^{\frac{1}{2}}y, \quad T-t=e^{-\tau}.\]
If $u$ satisfies \eqref{Cauchyproblem}, then $w$ satisfies
\begin{equation}\label{selfsimilar}
\partial_{\tau}w-\Delta w + \frac{1}{2} y \cdot \nabla w + \frac{1}{p-1} w - |w|^{p-1}w=0.
\end{equation}
In particular, if $w$ is a stationary solution of \eqref{selfsimilar}, then $w$ satisfies
\begin{equation}\label{SC1}
\Delta w-\frac{1}{2} y\cdot\nabla w-\frac{1}{p-1}w+|w|^{p-1} w=0,\quad\text{in}~\mathbb{R}^{n}.
\end{equation}
It is clear that \eqref{SC1} has three constant solutions $0, \pm\kappa$, where
\[\kappa:=(\frac{1}{p-1})^{\frac{1}{p-1}}.\]
Giga and Kohn \cite{Giga-Kohn1985} proved that if $1< p< p_{s}(n)$ and $u$  is a finite blow up solution  of  \eqref{Cauchyproblem}
  satisfying
  \begin{eqnarray}
    \sup_{\mathbb{R}^n\times(0,T)}(T-t)^{\frac{1}{p-1}}|u(x,t)|< \infty.
  \end{eqnarray}
 Then for any $x_0\in\R^n$,
  \begin{eqnarray}
  \label{limit1}
    \lim_{	t\rightarrow T}(T-t)^{\frac{1}{p-}}u(x_0+(T-t)y,t)=0\ \text{or}\ \pm\kappa
  \end{eqnarray}
uniformly for $y$ bounded. Their proof relies on two ingredients: first there is the {\em generalized Pohozaev} identity for bounded solutions of \eqref{SC1}
\begin{equation} \left(\frac{n}{p+1} +\frac{2-n}{2}\right) \int _{\mathbb{R}^{n}}|\nabla w|^2 \rho dy +\frac{1}{2} \left(\frac{1}{2}-\frac{1}{p+1}\right) \int_{\mathbb{R}^{n}} |y|^2 |\nabla w|^2 \rho dy =0
\end{equation}
where $ \rho = (4\pi)^{-n/2}e^{- |y|^2/4}$ is the Gaussian. Second the following  Giga-Kohn energy functional
\[E[w](\tau)= \frac{1}{2} \int_{\mathbb{R}^{n}} |\nabla w|^2 \rho dy +\frac{1}{2(p-1)} \int _{\mathbb{R}^{n}}|w|^2 \rho dy - \frac{1}{p+1} \int_{\mathbb{R}^{n}} |w|^{p+1} \rho dy\]
is monotonically decreasing for bounded solutions of \eqref{selfsimilar}.

In \cite{Giga-Kohn1987}, Giga and Kohn proved that if $1<p<p_{s}(n), u_{0}\geq 0$, then \eqref{Cauchyproblem} has only type I finite blow up solutions.  In \cite{Giga-Kohn1989}, Giga and Kohn   proved the nondegeneracy of the blow ups: we can tell whether or not a point $a$ is a blow up point by studying the asymptotic behavior of the solution in a backward 
spacetime parabola based at $(a, T)$. In \cite{Merle-Zaag}, Merle and Zaag classified all the bounded global nonnegative solutions to \eqref{selfsimilar} defined on $\mathbb{R}^n\times\mathbb{R}$.

Due to the above mentioned results, the blow up phenomenon of \eqref{Cauchyproblem} when $1<p<p_{s}(n)$ has been well understood. In the remaining part of the paper, we assume that $n\geq 3, p>(n+2)/(n-2)$.

\subsection{Setting and main results}
For any bounded smooth function, let 
\begin{equation}\label{weightedenergy}
E(w)=\frac{1}{2}\int_{\mathbb{R}^{n}}|\nabla w|^{2}\rho dy+\frac{1}{2(p-1)}\int_{\mathbb{R}^{n}}w^{2}\rho dy-\frac{1}{p+1}\int_{\mathbb{R}^{n}}|w|^{p+1}\rho dy
\end{equation}
be the weighted energy of $w$. For any constant $m$  such that $m>\kappa
$, we denote
\[\mathcal{B}_{n, m}=\{w: w~\text{is a positive solution of}~\eqref{SC1}~\text{such that}~\|w\|_{L^{\infty}(\mathbb{R}^{n})}\leq m\}.\]
Since $\kappa\in\mathcal{B}_{n, m}$, then $\mathcal{B}_{n, m}$ is not empty. The first main result is the following.
\begin{theorem}\label{maintheorem1}
Assume $n\geq 3, p>(n+2)/(n-2)$. Then for any $w\in\mathcal{B}_{n, m}$, 
\[E(w)\geq E(\kappa).\]
\end{theorem}
In Theorem \ref{maintheorem1}, the constant $m$ is not important. Therefore, what we essentially proved is that the positive constant solution has the lowest energy among bounded positive self similar solutions.
\begin{remark}
Let
\[\mathcal{E}=\{\text{the set of bounded radially symmetric solutions of}~\eqref{SC1}\}.\]
It has been proved by Matano and Merle \cite[Theorem 1.4]{Merle-Matano2011} that
\[E(w)\geq E(\kappa),\quad\text{for any}~w\in \mathcal{E}\backslash\{0\}.\]
Furthermore, the equality holds if and only if $w=\pm\kappa$. Their proof is ``parabolic'' and is based on the zero-number argument.
\end{remark}
Next, we prove that not only the constant solution has the lowest energy among functions in $\mathcal{B}_{n, m}$, but there is a gap to the second lowest.
\begin{theorem}\label{maintheorem2}
Assume $n\geq 3, p>(n+2)/(n-2)$. Then there exists a positive constant $\epsilon$ depends on $n, p, m$ such that if $w\in\mathcal{B}_{n,m}$ and $w$ is not the positive constant solution of \eqref{SC1}, then
\[E(w)>E(\kappa)+\epsilon.\]
\end{theorem}
After Giga-Kohn's pioneering paper \cite{Giga-Kohn1985}, it is known that the analysis of blow up phenomena of \eqref{Cauchyproblem} is closely related to properties of solutions of \eqref{SC1}. For finite time blow up solutions of \eqref{Cauchyproblem}, we can apply Theorem \ref{maintheorem2} to prove the following result.
\begin{proposition}\label{blowupsetstructure}
Let $n\geq 3, p>(n+2)/(n-2), u_{0}\geq 0$ and let
\[u(x, t)\neq \kappa(T-t)^{-\frac{1}{p-1}}\]
be a solution of the equation \eqref{Cauchyproblem} that blows up at $T$. Assume there is a positive constant $m>\kappa$ such that
 \begin{equation}\label{TypeIassumption}
u(x, t)\leq m(T-t)^{-\frac{1}{p-1}}, \quad\text{in}~\mathbb{R}^{n}\times(0, T).
\end{equation}
For any $R>0$, we set $\Sigma_{R}=\Sigma\cap B_{R}(0)$. Then
\begin{enumerate}
\item $\Sigma_{R}=\Sigma_{n-1} \cup \Sigma_{n-3}$;
\item $\Sigma_{n-1}$ is relatively open in $\Sigma$, and it is countably $(n-1)$-rectifiable;
\item $dim_{\mathcal{H}}(\Sigma_{n-3})\leq n-3$;
\item  $x_{0}\in\Sigma_{n-1}$ if and only if $\mathcal{T}(u, x_{0})=\{\kappa\}$.	
\end{enumerate}
\end{proposition}
If $p\geq p_{s}(n)$, it is known that \eqref{Cauchyproblem} can have type II blow up solutions. We refer to \cite{Collot-Merle-Raphael2020} and \cite{Wei-zhou2020} for some interesting examples. Therefore, it is plausible to extend Theorem \ref{maintheorem1} to unbounded solutions of \eqref{SC1}. In order to clarify this problem, we first give the definition of suitable weak solutions.
\begin{definition}[Suitable weak solutions]
A function $w\in H^{1}_{loc}(\mathbb{R}^{n})\cap L^{p+1}_{loc}(\mathbb{R}^{n})$ is a suitable weak solution of \eqref{SC1}, if
\[\int_{\mathbb{R}^{n}}[|\nabla w|^{2}+w^{2}+|w|^{p+1}]e^{-\frac{|y|^{2}}{4}}dy<\infty\]
and
\begin{itemize}
\item[(1)] $w$ is a weak solution of \eqref{SC1}, that is, for any $\eta\in C_{0}^{\infty}(\mathbb{R}^{n})$,
    \[0=\int_{\mathbb{R}^{n}}\nabla w\cdot\nabla \eta e^{-\frac{|y|^{2}}{4}}dy+\frac{1}{p-1}\int_{\mathbb{R}^{n}} w\eta e^{-\frac{|y|^{2}}{4}}dy- \int_{\mathbb{R}^{n}}|w|^{p-1}w\eta e^{-\frac{|y|^{2}}{4}}dy.\]
    \item[(2)] $w$  satisfies the stationary condition, that is , for any $Y\in C_{0}^{\infty}(\mathbb{R}^{n}, \mathbb{R}^{n})$,
    \[\begin{aligned}0=&-\int_{\mathbb{R}^{n}}DY(\nabla w, \nabla w)\rho dy+\frac{1}{2}\int_{\mathbb{R}^{n}}|\nabla w|^{2}(\textrm{div}Y-\frac{Y\cdot y}{2})\rho dy\\
    &+\frac{1}{2}\int_{\mathbb{R}^{n}}w^{2}(\textrm {div} Y-\frac{Y\cdot y}{2})\rho dy-\frac{1}{p+1}\int_{\mathbb{R}^{n}}|w|^{p+1}(\textrm {div} Y-\frac{Y\cdot y}{2})\rho dy.
     \end{aligned}\]
\end{itemize}
\end{definition}
It is easy to check that bounded smooth solutions are suitable weak solutions. However, a suitable weak solution need not to be smooth everywhere. Indeed,  it is easy to check that if $n\geq 3, p>(n+2)/(n-2)$, then
 \[w(y)=\left[\frac{2}{p-1}\left(n-2-\frac{2}{p-1}\right)\right]^{\frac{1}{p-1}}|y|^{-\frac{2}{p-1}}\]
 is a suitable weak solution of \eqref{SC1} which is not smooth at the origin. For any suitable weak solution $w$, let $Reg(w)$ be the regular part of $w$, then $Reg(w)$ is an open subset of $\mathbb{R}^{n}$. We set
\[\mathcal{B}_{n}=\{w: w~\text{is a suitable weak solution of}~\eqref{SC1}~\text{and}~w>0~\text{on}~Reg(w)\}.\]
\begin{theorem}\label{theoremmain3}
If $n\leq 3$ or $n\geq 4, (n+2)/(n-2)<p<(n+1)/(n-3)$, then
\[E(w)\geq E(\kappa)\]
for any $w\in\mathcal{B}_{n}$.
\end{theorem}

\subsection{Idea of the proof: $F$-functional and entropy} 

For a hypersurface $\Sigma$ of Euclidean space $\mathbb{R}^{n+1}$, the entropy is defined by
\[\lambda(\Sigma)=\sup (4\pi t_{0})^{-\frac{n}{2}}\int_{\Sigma} e^{-\frac{|x-x_{0}|^{2}}{4t_{0}}}dx.\]
Here the supremum is taking over all $t_{0}>0$ and $x_{0}\in\mathbb{R}^{n+1}.$ This quantity was introduced by Colding -Minicozzi \cite{Colding-Minicozzi2012}. As a consequence of Huisken's monotonicity formula, it is non-increasing along the mean
curvature flow, thus giving a Lyapunov functional. In \cite{Colding-I-M-W},  Colding–
Ilmanen–Minicozzi-White proved that within the  closed smooth self-shrinking solutions of the mean curvature flow in $\mathbb{R}^{n+1}$, not only does the round sphere
have the lowest entropy, but also there is a
gap to the second lowest. Based on this result, they conjectured that, for $2\leq n\leq 6$, the
round sphere minimizes the entropy among all closed hypersurfaces. Using a cleverly constructed weak mean curvature flow that ensured
the extinction time singularity was of a special type, this conjecture was verified by Bernstein and Wang \cite{Bernstein-Wang2016}.

As pointed by Vel\'{a}zquez in \cite{Velazquez1994Ann}, the structure of singularities which arise in
mean curvature flow is strikingly similar to those appearing in \eqref{Cauchyproblem}. Because of this reason, we will borrow some ideas from \cite{Colding-Minicozzi2012} and \cite{Colding-I-M-W} to consider bounded solutions of \eqref{SC1}. Inspired by the program developed by Colding and Minicozzi \cite{Colding-Minicozzi2012, Colding-I-M-W}, we will introduce the notation of $F-$ functional and entropy. In the setting of \eqref{Cauchyproblem}, for a bounded $C^{1}$ function $w$, the $F-$ functional is defined by
\begin{equation}\label{Func}
\begin{aligned}
F_{x_{0}, t_{0}}(w)=&\frac{1}{2}(-t_{0})^{\frac{p+1}{p-1}}\int_{\mathbb{R}^{n}}|\nabla w|^{2}G(y-x_{0}, t_{0})dy\\
&-\frac{1}{p+1}(-t_{0})^{\frac{p+1}{p-1}}\int_{\mathbb{R}^{n}}|w|^{p+1}G(y-x_{0}, t_{0})dy\\
&+\frac{1}{2(p-1)}(-t_{0})^{\frac{2}{p-1}}\int_{\mathbb{R}^{n}}w^{2}G(y-x_{0}, t_{0})dy,
\end{aligned}
\end{equation}
where for any $(y, t)\in\mathbb{R}^{n}\times (-\infty, 0)$,
\[G(y, t)=(-4\pi t)^{-\frac{n}{2}}e^{\frac{|y|^{2}}{{4t}}}.\]
In particular,
\[G(y,-1)=\rho(y)=(4\pi)^{-\frac{n}{2}}e^{-\frac{|y|^{2}}{{4}}}.\]
The motivation of defining $F-$ functional in this way comes from Giga-Kohn's monotonicity formula (see \cite[Proposition 3]{Giga-Kohn1985}). The main property of these functionals is that a bounded function is a critical point of $F_{x_{0}, t_{0}}$ if and only if it is the time $t=t_{0}$ slice of a self similar solution of \eqref{Cauchyproblem}.  Using the $F-$functionals, we can also define the entropy $\lambda(w)$ of a bounded smooth function to be the supremum of the $F_{x_{0}, t_{0}}$ functionals
\begin{equation}\label{defineentropyfunctional}
\lambda(w)=\sup\limits_{x_{0}\in\mathbb{R}^{n}, t_{0}\in(-\infty, 0)}F_{x_{0}, t_{0}}(w).
\end{equation}
Similarly we can define $F-$stable and entropy-stable solutions. Under some conditions, we show that these two definitions are equivalent. Using this fact, we can perturb a bounded non-constant positive solution of \eqref{SC1}, while reducing the entropy and making the solution of \eqref{selfsimilar} starting at this perturbed function blows up at a finite time. Finally, we use an induction argument to show that the minimizer of $\lambda(w)$  among $\mathcal{B}_{n, m}$ is attained by the constant.  An entropy gap (or energy gap) is also obtained.

Finally, we point out even though Theorem \ref{maintheorem1} and Theorem \ref{maintheorem2} suggest that the positive  constant  solution of \eqref{SC1} serve the same role as the round sphere in mean curvature flow. There are still some  striking differences.  For instance,  if a mean curvature flow in $\mathbb{R}^{n+1}$ starting
at a closed smooth embedded hypersurface has only generic singularities, then the round sphere is included in the lowest strata $\mathcal{S}_{0}$ \footnote{Here we have adopted the notations in \cite[Section 4]{Colding-Minicozzi2016}} consists of isolated points (see \cite[Lemma 4.2]{Colding-Minicozzi2016}). However, for finite blow up solutions of \eqref{Cauchyproblem}, the constant solution $\kappa$ is included in the top strata $\mathcal{S}_{n}$. This will lead to some essential difficulties in our analysis.
\section{Preliminaries}
In this section, we recall several results  which will be used later.
\begin{lemma}\label{lempre3.1}
Let $w$ be a bounded solution of the equation
\begin{equation}\label{SC1-2}
\Delta w-\frac{y}{2}\cdot\nabla w-\frac{1}{p-1}w+|w|^{p-1}w=0,\quad\text{in}~\mathbb{R}^{n}.
\end{equation}
There exists a positive constant $M'$ depending only on $n, p$ and $\|w\|_{L^{\infty}(\mathbb{R}^{n})}$ such that 
\[|\nabla w|+|\nabla^{2}w|+|\nabla^{3}w|\leq M',\quad\text{in}~\mathbb{R}^{n}\]
\end{lemma}
\begin{proof}
This is proved in \cite[Proposition 1']{Giga-Kohn1985}.
\end{proof}
Next, we prove a regularity result for solutions of \eqref{SC1-2} which decay at infinity.
\begin{lemma}\label{gradientestimate}
 Assume $w$ is a  bounded solution of \eqref{SC1-2} satisfying, for some  positive constant $C$,
\begin{equation}\label{hypothesis}
|w|\leq C(1+|y|)^{-\frac{2}{p-1}},\quad\text{in}~\mathbb{R}^{n}.
\end{equation}
Then there exists a positive constant $C_1$ such that
\begin{equation}\label{gradientrate}
(1+|y|)^{1+\frac{2}{p-1}}|\nabla w|+(1+|y|)^{2+\frac{2}{p-1}}|\nabla^{2}w|\leq C_{1},\quad\text{in}~\mathbb{R}^{n}.
\end{equation}
\end{lemma}
\begin{proof}
Let
\begin{equation}\label{selfsimilarsolution}
u(x, t)=(-t)^{-\frac{1}{p-1}}w\left(\frac{x}{\sqrt{-t}}\right).
\end{equation}
Then $u$ satisfies 
\begin{equation}\label{Fujitaequation}
\partial_{t}u=\Delta u+|u|^{p-1}u,\quad\text{in}~\mathbb{R}^{n}\times(-\infty, 0).
\end{equation}
By \eqref{hypothesis}, $u$ is  bounded on $(B_{2}(0)\backslash B_{1/4}(0))\times [-1, 0)$. Thus we get from parabolic regularity estimates (see the proof of \cite[Proposition 1]{Giga-Kohn1985}) that \[|\nabla u|+|\nabla^{2} u|\leq c_{1},\quad\text{in}~(B_{3/2}(0)\backslash B_{1/2}(0))\times [-\frac{1}{2}, 0)\]
for some positive constant $c_{1}$.  By scaling back, we can get \eqref{gradientrate}.
\end{proof}
For any bounded solution $w$ of \eqref{SC1-2}, let $u$ be the function defined by \eqref{selfsimilarsolution}. Next, we recall the following monotonicity formula.
\begin{lemma}[Monotonicity formula]\label{monotonicityformula}
Fix $(x, t)\in\mathbb{R}^{n}\times (-\infty, 0)$, then for any $T\geq t$, the function
\begin{align}\label{resalemontonicityformula}
E(s; x, T, u)=&\frac{1}{2}(-s)^{\frac{p+1}{p-1}}\int_{\mathbb{R}^{n}}|\nabla u(y, T+s)|^{2}G(y-x, s)dy\notag\\
&-\frac{1}{p+1}(-s)^{\frac{p+1}{p-1}}\int_{\mathbb{R}^{n}}|u|^{p+1}(y, T+s)G(y-x, s)dy\\
&+\frac{1}{2(p-1)}(-s)^{\frac{2}{p-1}}\int_{\mathbb{R}^{n}}u^{2}(y, T+s)G(y-x, s)dy\notag
\end{align}
is nonincreasing with respect to $s$ in $(-\infty, -(T-t))$. If $T>t$, then $E(s; x, T, u)$ is nonincreasing with respect to $s$ in $(-\infty, -(T-t)]$.
\end{lemma}
\begin{proof}
This is a reformulation of \cite[Proposition 3]{Giga-Kohn1985}.
\end{proof}
For any $(x, t)\in\mathbb{R}^{n}\times (-\infty, 0]$, we take $T=t$ into \eqref{resalemontonicityformula}. The monotonicity of $E$ allows us to define the density function
\begin{equation}\label{densityfunction}
	\Theta(x, t ;u):=\lim_{s\to 0}E(s ;x, t,u).
 \end{equation}
 The density function defined in \eqref{densityfunction} satisfies the following two properties, whose proof are standard.
\begin{lemma}\label{densityfunctionproperties1}
The density function $\Theta$ is upper semi-continuous in the sense that if $(x_{i}, t_{i})$ is a sequence of points  converging to $(x_{\infty}, t_{\infty})$, then
  \[\Theta(x_{\infty}, t_{\infty};u)\geq \limsup_{i\to\infty}\Theta(x_{i}, t_{i};u).\]
\end{lemma}
\begin{proof}
For any $\varepsilon>0$, choose an $s<0$ such that
\[E(s; x_{\infty}, t_{\infty}, u)\leq\Theta(x_{\infty}, t_{\infty}; u)+\varepsilon.\]
 Because $(x_{i}, t_{i})\to (x_{\infty}, t_{\infty})$ and $w$ is a bounded solution of \eqref{SC1-2}, we get from Lemma \ref{lempre3.1} that
\[E(s;x_{i}, t_{i},u)\to E(s; x_{\infty}, t_{\infty}, u)\]
as $i\to\infty$. Hence
		\[
			E(s;x_{i}, t_{i},u)\leq \Theta(x_{\infty}, t_{\infty}; u)+2\varepsilon
		\]
	provided that $i$ is large enough.	Let $s\to0$ and $\varepsilon\to 0$, we get the desired claim.
\end{proof}
\begin{lemma}\label{densityfunctionproperties2}
		For any $(x_{0}, t_{0})\in\R^{n}\times (-\infty, 0]$, we have $\Theta(x_{0}, t_{0};u)\leq \Theta(0, 0;u)$.
If $\Theta(x_{0}, t_{0};u)= \Theta(0, 0; u)$, then $u$ is backward self similar with respect to $(x_{0}, t_{0})$. Moreover,  for any $a\in\R$ and $(x,t)\in\R^n\times(-\infty,0)$,
			\[
			u(ax_{0}+x,a^2t_{0}+t)=u(x,t).
			\]
	\end{lemma}
	\begin{proof}
		Define the blow down sequence of $u$ at $(x_{0}, t_{0})$ by
		\[u_{\lambda}(x, t)=\lambda^{\frac{2}{p-1}}u(x_{0}+\lambda x, t_{0}+\lambda^2 t),\]
		where $\lambda\to+\infty$. Since $w$ is a bounded solution of \eqref{SC1-2} and $u$ is backward self similar with respect to $(0,0)$, then
		\[u_{\lambda}(x, t)=u(\lambda^{-1}x_{0}+x, \lambda^{-2}t_{0}+t)\to u(x, t)\]
		locally uniformly in $\R^n\times(-\infty,0)$. This uniform convergence implies that for $s<0$,
		\begin{eqnarray}\label{cone 1}
			\Theta(0, 0;u)&\equiv E(s;0, 0, u)=\lim\limits_{\lambda\to +\infty}E(s;0, 0, u_{\lambda})=\lim\limits_{\lambda\to +\infty}E(\lambda^2 s;x_{0}, t_{0}, u)\\
			&=\lim\limits_{\tau\to -\infty}E(\tau;x_{0}, t_{0},u)\geq \Theta(x_{0}, t_{0};u).  \nonumber
		\end{eqnarray}
  Here in the last inequality we have used the monotonicity formula at $(x_{0}, t_{0})$.

  Assume that $\Theta(x_{0}, t_{0};u)= \Theta(0, 0; u)$. For any $a>0$, we have
		\[\Theta(ax_{0},a^2t_{0}; u)=\Theta(x_{0}, t_{0};u)=\Theta(0, 0;u).\]
		Thus $u$ is backward self similar with respect to $(ax_{0},a^2t_{0})$. Since $u$ is also backward self similar  with respect to $(0,0)$,
		for any $\lambda>0$, we have
		\[u(ax_{0}+x, a^2 t_{0}+t)=\lambda^{\frac{2}{p-1}}u(ax_{0}+\lambda x, a^2 t_{0}+\lambda^2t)=u(\lambda^{-1}ax_{0}+x, \lambda^{-2}a^2t_{0}+t).\]
		Since $u$ is smooth, for any $(x,t)\in\R^n\times(-\infty,0)$,
		\[\lim_{\lambda\to+\infty}u(\lambda^{-1}ax_{0}+x, (a\lambda)^{-2}t_{0}+t)= u(x,t).\]
		Thus
		\[u(ax_{0}+x, a^2 t_{0}+t)=u(x, t).\]
		The case $a<0$ can be obtained similarly, hence we have finished the proof.
	\end{proof}
Finally, for bounded solutions  of \eqref{SC1}, we recall the following Liouville type result.
\begin{theorem}\label{theoremGiga-Kohn1985}
If $n\leq 2$ or $n\geq 3, 1<p\leq (n+2)/(n-2)$, then the only bounded solutions of \eqref{SC1-2} are the trivial ones $w=0$ and $w=\pm\kappa$.
\end{theorem}
\begin{proof}
As mentioned in the introduction, Giga and Kohn \cite[Theorem 1]{Giga-Kohn1985} obtained the Pohozaev type identity
\[0=(\frac{n}{p+1}+\frac{2-n}{2})\int_{\mathbb{R}^{n}}|\nabla w|^{2}e^{-\frac{|y|^{2}}{4}}dy+\frac{1}{2}(\frac{1}{2}-\frac{1}{p+1})\int_{\mathbb{R}^{n}}|y|^{2}|\nabla w|^{2}e^{-\frac{|y|^{2}}{4}}dy.\]
 Then Theorem \ref{theoremGiga-Kohn1985} follows directly from this identity.
\end{proof}
\section{A variational characterization of self similar solutions}
\subsection{The general first variation formula}
Assume  $\phi\in C_{0}^{\infty}(\mathbb{R}^{n})$, 
$x(s)$ and $t(s)$ are variations such that
\[x(0)=x_{0},\quad t(0)=t_{0},\quad x'(0)=y_{0},\quad t'(0)=h.\]
Consider the $F-$ functional defined in \eqref{Func}, then
\begin{align}
F_{x(s), t(s)}(w+s\phi)=&\frac{1}{2}(-t(s))^{\frac{p+1}{p-1}}\int_{\mathbb{R}^{n}}|\nabla (w+s\phi)|^{2}G(y-x(s), t(s))dy\notag\\
&-\frac{1}{p+1}(-t(s))^{\frac{p+1}{p-1}}\int_{\mathbb{R}^{n}}|w+s\phi|^{p+1}G(y-x(s), t(s))dy\notag\\
&+\frac{1}{2(p-1)}(-t(s))^{\frac{2}{p-1}}\int_{\mathbb{R}^{n}}(w+s\phi)^{2}G(y-x(s), t(s))dy\notag.
\end{align}
 Notice that
\begin{align}
&\frac{d}{ds}\left((-4\pi t(s))^{-\frac{n}{2}}e^{\frac{|x(s)-y|^{2}}{{4t(s)}}}\right)\notag\\
=&(-4\pi t(s))^{-\frac{n}{2}}e^{\frac{|x(s)-y|^{2}}{{4t(s)}}}\left[-\frac{nt'(s)}{2t(s)}+\frac{(x(s)-y)\cdot x'(s)}{2t(s)}-\frac{t'(s)|x(s)-y|^{2}}{4t^{2}(s)}\right]\notag,
\end{align}
thus
\begin{align}
&\frac{\partial}{\partial s} (F_{x(s), t(s)}(w+s\phi))\notag\\
=&-\frac{p+1}{2(p-1)}(-t(s))^{\frac{2}{p-1}}t'(s)\int_{\mathbb{R}^{n}}|\nabla w+s\nabla\phi|^{2}G(y-x(s), t(s))dy\notag\\
&+\frac{1}{p-1}(-t(s))^{\frac{2}{p-1}}t'(s)\int_{\mathbb{R}^{n}}|w+s\phi|^{p+1}G(y-x(s), t(s))dy\notag\\
&-\frac{1}{(p-1)^{2}}(-t(s))^{\frac{2}{p-1}-1}t'(s)\int_{\mathbb{R}^{n}}(w+s\phi)^{2}G(y-x(s), t(s))dy\notag\\
&+(-t(s))^{\frac{p+1}{p-1}}\int_{\mathbb{R}^{n}}(\nabla w+s\nabla \phi)\cdot\nabla\phi G(y-x(s), t(s))dy\notag\\
&-(-t(s))^{\frac{p+1}{p-1}}\int_{\mathbb{R}^{n}}|w+s\phi|^{p-1}(w+s\phi)\phi G(y-x(s), t(s))dy\notag\\
&+\frac{1}{p-1}(-t(s))^{\frac{2}{p-1}}\int_{\mathbb{R}^{n}}(w+s\phi)\phi G(y-x(s), t(s))dy\notag\\
&+\frac{1}{2}(-t(s))^{\frac{p+1}{p-1}}\int_{\mathbb{R}^{n}}|\nabla w+s\nabla\phi|^{2}G(y-x(s), t(s))\notag\\
&\times [-\frac{nt'(s)}{2t(s)}+\frac{(x(s)-y)\cdot x'(s)}{2t(s)}-\frac{t'(s)|x(s)-y|^{2}}{4t^{2}(s)}]dy\notag\\
&-\frac{1}{p+1}(-t(s))^{\frac{p+1}{p-1}}\int_{\mathbb{R}^{n}}|w+s\phi|^{p+1}G(y-x(s), t(s))\notag\\
&\times [-\frac{nt'(s)}{2t(s)}+\frac{(x(s)-y)\cdot x'(s)}{2t(s)}-\frac{t'(s)|x(s)-y|^{2}}{4t^{2}(s)}]dy\notag\\
&+\frac{1}{2(p-1)}(-t(s))^{\frac{2}{p-1}}\int_{\mathbb{R}^{n}}(w+s\phi)^{2}G(y-x(s), t(s))\notag\\
&\times[-\frac{nt'(s)}{2t(s)}+\frac{(x(s)-y)\cdot x'(s)}{2t(s)}-\frac{t'(s)|x(s)-y|^{2}}{4t^{2}(s)}]dy\notag.
\end{align}
Take  $s=0$ into the above formula, we have the following result.
\begin{lemma}[The first variation formula]
Let $w$ be a bounded smooth function. Assume  $\phi\in C_{0}^{\infty}(\mathbb{R}^{n})$, 
$x(s)$ and $t(s)$ are variations such that
\[x(0)=x_{0},\quad t(0)=t_{0},\quad x'(0)=y_{0},\quad t'(0)=h.\]
Then
\begin{align}\label{firstvariationformula}
&\frac{\partial}{\partial s}(F_{x(s), t(s)}(w+s\phi))|_{s=0}\notag\\
=&-\frac{p+1}{2(p-1)}(-t_{0})^{\frac{2}{p-1}}h\int_{\mathbb{R}^{n}}|\nabla w|^{2}G(y-x_{0}, t_{0})dy\notag\\
&+\frac{1}{p-1}(-t_{0})^{\frac{2}{p-1}}h\int_{\mathbb{R}^{n}}|w|^{p+1}G(y-x_{0}, t_{0})dy\notag\\
&-\frac{1}{(p-1)^{2}}(-t_{0})^{\frac{2}{p-1}-1}h\int_{\mathbb{R}^{n}}w^{2}G(y-x_{0}, t_{0})dy\notag\\
&+(-t_{0})^{\frac{p+1}{p-1}}\int_{\mathbb{R}^{n}}(\nabla w\cdot\nabla\phi) G(y-x_{0}, t_{0})dy\notag\\
&-(-t_{0})^{\frac{p+1}{p-1}}\int_{\mathbb{R}^{n}}|w|^{p-1}w\phi G(y-x_{0}, t_{0})dy\notag\\
&+\frac{1}{(p-1)}(-t_{0})^{\frac{2}{p-1}}\int_{\mathbb{R}^{n}}w\phi G(y-x_{0}, t_{0}) dy\notag\\
&+\frac{1}{2}(-t_{0})^{\frac{p+1}{p-1}}\int_{\mathbb{R}^{n}}|\nabla w|^{2}G(y-x_{0}, t_{0})\\
&\times[-\frac{nh}{2t_{0}}+\frac{(x_{0}-y)\cdot y_{0}}{2t_{0}}-\frac{h|x_{0}-y|^{2}}{4t_{0}^{2}}]dy\notag\\
&-\frac{1}{p+1}(-t_{0})^{\frac{p+1}{p-1}}\int_{\mathbb{R}^{n}}|w|^{p+1}G(y-x_{0}, t_{0})\notag\\
&\times[-\frac{nh}{2t_{0}}+\frac{(x_{0}-y)\cdot y_{0}}{2t_{0}}-\frac{h|x_{0}-y|^{2}}{4t_{0}^{2}}]dy\notag\\
&+\frac{1}{2(p-1)}(-t_{0})^{\frac{2}{p-1}}\int_{\mathbb{R}^{n}}w^{2}G(y-x_{0}, t_{0})\notag\\
&\times[-\frac{nh}{2t_{0}}+\frac{(x_{0}-y)\cdot y_{0}}{2t_{0}}-\frac{h|x_{0}-y|^{2}}{4t_{0}^{2}}]dy\notag.
\end{align}
\end{lemma}
\begin{definition}
Let $w$ be a bounded smooth function. If for all variations $x(s), t(s)$ and $\phi\in C_{0}^{\infty}(\mathbb{R}^{n})$, we have
\[\frac{\partial}{\partial s}(F_{x(s), t(s)}(w+s\phi))|_{s=0}=0,\]
then we say that $w$ is a critical point for the functional $F_{x_{0}, t_{0}}$.
\end{definition}

\subsection{Critical points of $F-$ functional are self similar solutions}
In this section, we will prove that $w$ is a critical point of $F_{x_{0}, t_{0}}$ if and only if
it is the time $-t_{0}$ slice of a self similar solution of 
\[\partial_{t}u=\Delta u+|u|^{p-1}u,\quad\text{in}~ \mathbb{R}^{n}\times (-\infty, 0)\]
that becomes extinct at the point $x_{0}$ and time $0$.
\begin{proposition}\label{critcalpoints}
A  bounded smooth function  $w$  is a critical point of $F_{x_{0}, t_{0}}$ if and only if $w$ is a  solution to the equation
\begin{equation}\label{1}
\Delta w+\frac{y-x_{0}}{2t_{0}}\cdot \nabla w+\frac{1}{(p-1)t_{0}}w+|w|^{p-1}w=0,\quad\text{in}~\mathbb{R}^{n}.
\end{equation}
\end{proposition}
\begin{proof}
Obviously, it is sufficient to prove this for $x_{0}=0$ and $t_{0}=-1$. 

First assume   $w$ is a critical point of  $F_{0, -1}$. Taking $y_{0}=0$ and $h=0$ into \eqref{firstvariationformula}, we see that for any $\phi\in C_{0}^{\infty}(\mathbb{R}^{n})$,
\begin{align}\label{mlpsi}
0=\int_{\mathbb{R}^{n}}\nabla w\cdot\nabla\phi\rho dy+\frac{1}{p-1}\int_{\mathbb{R}^{n}}w\phi\rho dy-\int_{\mathbb{R}^{n}}|w|^{p-1}w\phi\rho dy.
\end{align}
Thus  
\[0=\int_{\mathbb{R}^{n}}\left[\frac{1}{\rho}\textrm{div}(\rho\nabla w)-\frac{1}{p-1}w+|w|^{p-1}w\right]\phi\rho dy.\]
It follows that $w$ satisfies \eqref{1}.

Next, assume $w$ is a solution of \eqref{1} with $x_{0}=0, t_{0}=-1$. We need to show that for any $y_{0}$, $h$ and $\phi\in C_{0}^{\infty}(\mathbb{R}^{n})$,
\[\frac{\partial}{\partial s}(F_{x(s), t(s)}(w+s\phi))|_{s=0}=0.\]
This is equivalent to the requirement that
\begin{align}\label{firstvariation2024-04-05}
0=&-\frac{p+1}{2(p-1)}h\int_{\mathbb{R}^{n}}|\nabla w|^{2}\rho dy+\frac{h}{p-1}\int_{\mathbb{R}^{n}}|w|^{p+1}\rho dy\notag\\
&-\frac{h}{(p-1)^{2}}\int_{\mathbb{R}^{n}}w^{2}\rho dy+\int_{\mathbb{R}^{n}}(\nabla w\cdot\nabla\phi)\rho dy\notag\\
&+\frac{1}{p-1}\int_{\mathbb{R}^{n}} w\phi \rho dy-\int_{\mathbb{R}^{n}}|w|^{p-1}w\phi dy\\
&+\frac{1}{2}\int_{\mathbb{R}^{n}}|\nabla w|^{2}\left[\frac{nh}{2}+\frac{y\cdot y_{0}}{2}-\frac{h|y|^{2}}{4}\right]\rho dy\notag\\
&-\frac{1}{p+1}\int_{\mathbb{R}^{n}}|w|^{p+1}\left[\frac{nh}{2}+\frac{y\cdot y_{0}}{2}-\frac{h|y|^{2}}{4}\right]\rho dy\notag\\
&+\frac{1}{2(p-1)}\int_{\mathbb{R}^{n}}w^{2}\left[\frac{nh}{2}+\frac{y\cdot y_{0}}{2}-\frac{h|y|^{2}}{4}\right]\rho dy\notag.
\end{align}
Multiplying both sides of \eqref{SC1} by $\rho w$ and integrating by parts, we  get 
\begin{equation}\label{Inte1}
0=\int_{\mathbb{R}^{n}}|\nabla w|^{2}\rho dy-\int_{\mathbb{R}^{n}}|w|^{p+1}\rho dy+\frac{1}{p-1}\int_{\mathbb{R}^{n}}w^{2}\rho dy.
\end{equation}
By \cite[Formula (3,17)]{Giga-Kohn1985}, we have
\begin{align}\label{10}
0=&\frac{2-n}{2}\int_{\mathbb{R}^{n}}|\nabla w|^{2}\rho dy-\frac{n}{2(p-1)}\int_{\mathbb{R}^{n}}w^{2}\rho dy\notag\\
&+\frac{n}{p+1}\int_{\mathbb{R}^{n}} |w|^{p+1}\rho dy+\frac{1}{4}\int_{\mathbb{R}^{n}}|y|^{2}|\nabla w|^{2}\rho dy\\
&+\frac{1}{4(p-1)}\int_{\mathbb{R}^{n}}|y|^{2}w^{2}\rho dy-\frac{1}{2(p+1)}\int_{\mathbb{R}^{n}}|y|^{2} |w|^{p+1}\rho dy.\notag
\end{align}
For $i=1, 2,\cdots n$, multiplying both sides of \eqref{1} by $\rho w_{i}$ and  integrating by parts, we  get
\begin{align}
0=&-\int_{\mathbb{R}^{n}}\nabla w\cdot\nabla w_{i}\rho dy-\frac{1}{p-1}\int_{\mathbb{R}^{n}}w_{i}w\rho dy+\int_{\mathbb{R}^{n}}|w|^{p-1}ww_{i}\rho dy\notag\\
=&-\frac{1}{2}\int_{\mathbb{R}^{n}}(|\nabla w|^{2})_{i}\rho dy-\frac{1}{2(p-1)}\int_{\mathbb{R}^{n}}(w^{2})_{i}\rho dy+\frac{1}{p+1}\int_{\mathbb{R}^{n}}(|w|^{p+1})_{i}\rho dy\notag\\
=&\frac{1}{4}\int_{\mathbb{R}^{n}}y_{i}|\nabla w|^{2}\rho dy+\frac{1}{4(p-1)}\int_{\mathbb{R}^{n}}y_{i}w^{2}\rho dy-\frac{1}{2(p+1)}\int_{\mathbb{R}^{n}}y_{i}|w|^{p+1}\rho dy\notag.
\end{align}
It follows that for any $y_{0}\in\mathbb{R}^{n}$,
 \begin{align}\label{Testfunction}
0=&\frac{1}{2}\int_{\mathbb{R}^{n}}|\nabla w|^{2}(y\cdot y_{0})\rho dy-\frac{1}{p+1}\int_{\mathbb{R}^{n}}|w|^{p+1}(y\cdot y_{0})\rho dy\\
&+\frac{1}{2(p-1)}\int_{\mathbb{R}^{n}}w^{2}(y\cdot y_{0})\rho dy.\notag
\end{align}
Combining \eqref{mlpsi}, \eqref{Inte1}, \eqref{10} and \eqref{Testfunction} in the   way
 \[\eqref{mlpsi}-\frac{h}{p-1}\eqref{Inte1}-\frac{h}{2}\eqref{10}+\frac{1}{2}\eqref{Testfunction},\]
  we see that \eqref{firstvariation2024-04-05} holds. Hence the proof of Proposition \ref{critcalpoints} is accomplished.
\end{proof}

\section{The second variation} 
In this section, we will calculate the second variation formula of the $F-$ functional for simultaneous variations in all three parameters $w$, $x_{0}$ and $t_{0}$. In particular, when $w$ is a bounded solution of the equation
\begin{equation*}
\Delta w-\frac{y}{2}\cdot\nabla w-\frac{1}{p-1}w+|w|^{p-1}w=0,\quad\text{in}~\mathbb{R}^{n},
\end{equation*}
we will  use our calculation to formulate a notion of stability.
\subsection{The general second variation formula}
\begin{lemma}
Let $w$ be a bounded smooth function. Assume $x(s)$ and $t(s)$ are variations such that
\[x(0)=x_{0},\quad t(0)=t_{0},\quad x'(0)=y_{0},\quad t'(0)=h\]
 and $x''(0)=y_{0}', t''(0)=h'$, then
 \begin{align}
 &\frac{\partial^{2}}{\partial s^2} (F_{x(s), t(s)}(w+s\phi))|_{s=0}\notag\\
=&\frac{p+1}{2(p-1)}(-t_{0})^{-1+\frac{2}{p-1}}\left[t_{0}h'+\frac{2}{p-1}h^{2}\right]\int_{\mathbb{R}^{n}}|\nabla w|^{2}G(y-x_{0}, t_{0})dy\notag\\
&-\frac{1}{p-1}(-t_{0})^{-1+\frac{2}{p-1}}\left[t_{0}h'+\frac{2}{p-1}h^{2}\right]\int_{\mathbb{R}^{n}}|w|^{p+1}G(y-x_{0}, t_{0})dy\notag\\
&-\frac{1}{(p-1)^{2}}(-t_{0})^{-2+\frac{2}{p-1}}\left[-t_{0}h'+\frac{p-3}{p-1}h^{2}\right]\int_{\mathbb{R}^{n}}w^{2}G(y-x_{0}, t_{0})dy\notag\\
&-\frac{2(p+1)}{p-1}(-t_{0})^{\frac{2}{p-1}}h\int_{\mathbb{R}^{n}}(\nabla w\cdot\nabla\phi)G(y-x_{0}, t_{0})dy\notag\\
&+\frac{2(p+1)}{p-1}(-t_{0})^{\frac{2}{p-1}}h\int_{\mathbb{R}^{n}}|w|^{p-1}w\phi G(y-x_{0}, t_{0})dy\notag\\
&-\frac{4}{(p-1)^{2}}(-t_{0})^{\frac{2}{p-1}-1}h\int_{\mathbb{R}^{n}}w\phi G(y-x_{0}, t_{0})dy\notag\\
&-\frac{p+1}{p-1}(-t_{0})^{\frac{2}{p-1}}h\int_{\mathbb{R}^{n}}|\nabla w|^{2}G(y-x_{0}, t_{0})\notag\\
&\qquad \quad\times \left[-\frac{nh}{2t_{0}}+\frac{(x_{0}-y)\cdot y_{0}}{2t_{0}}-\frac{h|x_{0}-y|^{2}}{4t_{0}^{2}}\right]dy\notag\\
&+\frac{2}{p-1}(-t_{0})^{\frac{2}{p-1}}h\int_{\mathbb{R}^{n}}|w|^{p+1}G(y-x_{0}, t_{0})\notag\\
&\qquad \quad \times\left [-\frac{nh}{2t_{0}}+\frac{(x_{0}-y)\cdot y_{0}}{2t_{0}}-\frac{h|x_{0}-y|^{2}}{4t_{0}^{2}}\right]dy\notag\\
&-\frac{2}{(p-1)^{2}}(-t_{0})^{\frac{2}{p-1}-1}h\int_{\mathbb{R}^{n}}w^{2}G(y-x_{0}, t_{0})\notag\\
&\qquad \quad \times \left[-\frac{nh}{2t_{0}}+\frac{(x_{0}-y)\cdot y_{0}}{2t_{0}}-\frac{h|x_{0}-y|^{2}}{4t_{0}^{2}}\right]dy\notag\\
&+(-t_{0})^{\frac{p+1}{p-1}}\int_{\mathbb{R}^{n}}|\nabla\phi|^{2}G(y-x_{0}, t_{0})dy\notag\\
&-p(-t_{0})^{\frac{p+1}{p-1}}\int_{\mathbb{R}^{n}}|w|^{p-1}\phi^{2}G(y-x_{0}, t_{0})dy\notag\\
&+\frac{1}{p-1}(-t_{0})^{\frac{2}{p-1}}\int_{\mathbb{R}^{n}}\phi^{2}G(y-x_{0}, t_{0})dy\notag\\
&+2(-t_{0})^{\frac{p+1}{p-1}}\int_{\mathbb{R}^{n}}(\nabla w\cdot\nabla\phi) G(y-x_{0}, t_{0})\notag\\
&\qquad \quad \times \left[-\frac{nh}{2t_{0}}+\frac{(x_{0}-y)\cdot y_{0}}{2t_{0}}-\frac{h|x_{0}-y|^{2}}{4t_{0}^{2}}\right]dy\notag\\
&-2(-t_{0})^{\frac{p+1}{p-1}}\int_{\mathbb{R}^{n}}|w|^{p-1}w\phi G(y-x_{0}, t_{0})\notag\\
&\qquad \quad \times\left[-\frac{nh}{2t_{0}}+\frac{(x_{0}-y)\cdot y_{0}}{2t_{0}}-\frac{h|x_{0}-y|^{2}}{4t_{0}^{2}}\right]dy\notag\\
&+\frac{2}{p-1}(-t_{0})^{\frac{2}{p-1}}\int_{\mathbb{R}^{n}}w\phi G(y-x_{0}, t_{0})\notag\\
&\qquad\quad \times \left[-\frac{nh}{2t_{0}}+\frac{(x_{0}-y)\cdot y_{0}}{2t_{0}}-\frac{h|x_{0}-y|^{2}}{4t_{0}^{2}}\right]dy\notag\\
&+\frac{1}{2}(-t_{0})^{\frac{p+1}{p-1}}\int_{\mathbb{R}^{n}}|\nabla w|^{2} G(y-x_{0}, t_{0})\notag\\
&\qquad \quad \times\Big\{\left[-\frac{nh}{2t_{0}}+\frac{(x_{0}-y)\cdot y_{0}}{2t_{0}}-\frac{h|x_{0}-y|^{2}}{4t_{0}^{2}}\right]^{2}-\frac{nh't_{0}-nh^{2}}{2t_{0}^{2}}\notag\\
&\qquad \quad\quad +\frac{[|y_{0}|^{2}+(x_{0}-y)\cdot y_{0}']t_{0}-h(x_{0}-y)\cdot y_{0}}{2t_{0}^{2}}\notag\\
&\qquad \quad\quad -\frac{[h'|x_{0}-y|^{2}+2h(x_{0}-y)\cdot y_{0}]t_{0}-2h^{2}|x_{0}-y|^{2}}{4t_{0}^{3}}\Big\}dy\notag\\
&-\frac{1}{p+1}(-t_{0})^{\frac{p+1}{p-1}}\int_{\mathbb{R}^{n}}|w|^{p+1}G(y-x_{0}, t_{0})\notag\\
&\qquad \quad\times\Big\{\left[-\frac{nh}{2t_{0}}+\frac{(x_{0}-y)\cdot y_{0}}{2t_{0}}-\frac{h|x_{0}-y|^{2}}{4t_{0}^{2}}\right]^{2}-\frac{nh't_{0}-nh^{2}}{2t_{0}^{2}}\notag\\
&\qquad \quad\quad +\frac{[|y_{0}|^{2}+(x_{0}-y)\cdot y_{0}']t_{0}-h(x_{0}-y)\cdot y_{0}}{2t_{0}^{2}}\notag\\
&\qquad \quad\quad -\frac{[h'|x_{0}-y|^{2}+2h(x_{0}-y)\cdot y_{0}]t_{0}-2h^{2}|x_{0}-y|^{2}}{4t_{0}^{3}}\Big\}dy\notag\\
&+\frac{1}{2(p-1)}(-t_{0})^{\frac{2}{p-1}}\int_{\mathbb{R}^{n}}w^{2}G(y-x_{0}, t_{0})\notag\\
&\qquad \quad\times\Big\{\left[-\frac{nh}{2t_{0}}+\frac{(x_{0}-y)\cdot y_{0}}{2t_{0}}-\frac{h|x_{0}-y|^{2}}{4t_{0}^{2}}\right]^{2}-\frac{nh't_{0}-nh^{2}}{2t_{0}^{2}}\notag\\
&\qquad \quad\quad +\frac{[|y_{0}|^{2}+(x_{0}-y)\cdot y_{0}']t_{0}-h(x_{0}-y)\cdot y_{0}}{2t_{0}^{2}}\notag\\
&\qquad \quad\quad -\frac{[h'|x_{0}-y|^{2}+2h(x_{0}-y)\cdot y_{0}]t_{0}-2h^{2}|x_{0}-y|^{2}}{4t_{0}^{3}}\Big\}dy\notag.
\end{align}
 \end{lemma}
 \begin{proof}
 After some computations, we   get  
\begin{align}
&\frac{\partial^{2}}{\partial s^2} (F_{x(s), t(s)}(w+s\phi))\notag\\
=&\frac{p+1}{2(p-1)}(-t(s))^{-1+\frac{2}{p-1}}\left[t(s)t''(s)+\frac{2}{p-1}(t'(s))^{2}\right]\notag\\
&\qquad
\quad\times\int_{\mathbb{R}^{n}}|\nabla w+s\nabla\phi|^{2}G(y-x(s), t(s))dy\notag\\
&-\frac{1}{p-1}(-t(s))^{-1+\frac{2}{p-1}}\left[t(s)t''(s)+\frac{2}{p-1}(t'(s))^{2}\right]\notag\\
&\qquad \quad\times\int_{\mathbb{R}^{n}}|w+s\phi|^{p+1}G(y-x(s), t(s))dy\notag\\
&-\frac{1}{(p-1)^{2}}(-t(s))^{-2+\frac{2}{p-1}}\left[\frac{p-3}{p-1}(t'(s))^{2}-t(s)t''(s)\right]\notag\\
&\qquad \quad\times\int_{\mathbb{R}^{n}}(w+s\phi)^{2}G(y-x(s), t(s))dy\notag\\
&-\frac{2(p+1)}{p-1}(-t(s))^{\frac{2}{p-1}}t'(s)\int_{\mathbb{R}^{n}}(\nabla w+s\nabla \phi)\cdot\nabla\phi G(y-x(s), t(s))dy\notag\\
&+\frac{2(p+1)}{p-1}(-t(s))^{\frac{2}{p-1}}t'(s)\int_{\mathbb{R}^{n}}|w+s\phi|^{p-1}(w+s\phi)\phi G(y-x(s), t(s))dy\notag\\
&-\frac{4}{(p-1)^{2}}(-t(s))^{\frac{2}{p-1}-1}t'(s)\int_{\mathbb{R}^{n}}(w+s\phi)\phi G(y-x(s), t(s))dy\notag\\
&-\frac{p+1}{p-1}(-t(s))^{\frac{2}{p-1}}t'(s)\int_{\mathbb{R}^{n}}|\nabla w+s\nabla\phi|^{2}G(y-x(s), t(s))\notag\\
&\qquad \quad\times \left[-\frac{nt'(s)}{2t(s)}+\frac{(x(s)-y)\cdot x'(s)}{2t(s)}-\frac{t'(s)|x(s)-y|^{2}}{4t^{2}(s)}\right]dy\notag\\
&+\frac{2}{p-1}(-t(s))^{\frac{2}{p-1}}t'(s)\int_{\mathbb{R}^{n}}|w+s\phi|^{p+1}G(y-x(s), t(s))\notag\\
&\qquad \quad\times \left[-\frac{nt'(s)}{2t(s)}+\frac{(x(s)-y)\cdot x'(s)}{2t(s)}-\frac{t'(s)|x(s)-y|^{2}}{4t^{2}(s)}\right]dy\notag\\
&-\frac{2}{(p-1)^{2}}(-t(s))^{\frac{2}{p-1}-1}t'(s)\int_{\mathbb{R}^{n}}(w+s\phi)^{2}G(y-x(s), t(s))\notag\\
&\qquad \quad\times\left[-\frac{nt'(s)}{2t(s)}+\frac{(x(s)-y)\cdot x'(s)}{2t(s)}-\frac{t'(s)|x(s)-y|^{2}}{4t^{2}(s)}\right]dy\notag\\
&+(-t(s))^{\frac{p+1}{p-1}}\int_{\mathbb{R}^{n}}|\nabla\phi|^{2}G(y-x(s), t(s))dy\notag\\
&-p(-t(s))^{\frac{p+1}{p-1}}\int_{\mathbb{R}^{n}}|w+s\phi|^{p-1}\phi^{2}G(y-x(s), t(s))dy\notag\\
&+\frac{1}{p-1}(-t(s))^{\frac{2}{p-1}}\int_{\mathbb{R}^{n}}\phi^{2}G(y-x(s), t(s))dy\notag\\
&+2(-t(s))^{\frac{p+1}{p-1}}\int_{\mathbb{R}^{n}}(\nabla w+s\nabla\phi)\cdot\nabla\phi G(y-x(s), t(s))\notag\\
&\qquad \quad\times \left[-\frac{nt'(s)}{2t(s)}+\frac{(x(s)-y)\cdot x'(s)}{2t(s)}-\frac{t'(s)|x(s)-y|^{2}}{4t^{2}(s)}\right]dy\notag\\
&-2(-t(s))^{\frac{p+1}{p-1}}\int_{\mathbb{R}^{n}}|w+s\phi|^{p-1}(w+s\phi)\phi G(y-x(s), t(s))\notag\\
&\qquad \quad\times \left[-\frac{nt'(s)}{2t(s)}+\frac{(x(s)-y)\cdot x'(s)}{2t(s)}-\frac{t'(s)|x(s)-y|^{2}}{4t^{2}(s)}\right]dy\notag\\
&+\frac{2}{p-1}(-t(s))^{\frac{2}{p-1}}\int_{\mathbb{R}^{n}}(w+s\phi)\phi G(y-x(s), t(s))\notag\\
&\qquad \quad\times \left[-\frac{nt'(s)}{2t(s)}+\frac{(x(s)-y)\cdot x'(s)}{2t(s)}-\frac{t'(s)|x(s)-y|^{2}}{4t^{2}(s)}\right]dy\notag\\
&+\frac{1}{2}(-t(s))^{\frac{p+1}{p-1}}\int_{\mathbb{R}^{n}}|\nabla w+s\nabla\phi|^{2} G(y-x(s), t(s))\notag\\
&\qquad \quad\times\Big\{\left[-\frac{nt'(s)}{2t(s)}+\frac{(x(s)-y)\cdot x'(s)}{2t(s)}-\frac{t'(s)|x(s)-y|^{2}}{4t^{2}(s)}\right]^{2}-\frac{nt''(s)t(s)-nt'(s)t'(s)}{2t^{2}(s)}\notag\\
&\qquad \quad\quad +\frac{[|x'(s)|^{2}+(x(s)-y)\cdot x''(s)]t(s)-t'(s)(x(s)-y)\cdot x'(s)}{2t^{2}(s)}\notag\\
&\qquad \quad\quad -\frac{[t''(s)|x(s)-y|^{2}+2t'(s)(x(s)-y)\cdot x'(s)]t(s)-2t'(s)t'(s)|x(s)-y|^{2}}{4t^{3}(s)}\Big\}dy\notag\\
&-\frac{1}{p+1}(-t(s))^{\frac{p+1}{p-1}}\int_{\mathbb{R}^{n}}|w+s\phi|^{p+1}G(y-x(s), t(s))\notag\\
&\qquad \quad\times\Big\{\left[-\frac{nt'(s)}{2t(s)}+\frac{(x(s)-y)\cdot x'(s)}{2t(s)}-\frac{t'(s)|x(s)-y|^{2}}{4t^{2}(s)}\right]^{2}-\frac{nt''(s)t(s)-nt'(s)t'(s)}{2t^{2}(s)}\notag\\
&\qquad \quad\quad +\frac{[|x'(s)|^{2}+(x(s)-y)\cdot x''(s)]t(s)-t'(s)(x(s)-y)\cdot x'(s)}{2t^{2}(s)}\notag\\
&\qquad \quad\quad -\frac{[t''(s)|x(s)-y|^{2}+2t'(s)(x(s)-y)\cdot x'(s)]t(s)-2t'(s)t'(s)|x(s)-y|^{2}}{4t^{3}(s)}\Big\}dy\notag\\
&+\frac{1}{2(p-1)}(-t(s))^{\frac{2}{p-1}}\int_{\mathbb{R}^{n}}(w+s\phi)^{2}G(y-x(s), t(s))\notag\\
&\qquad \quad\times\Big\{\left[-\frac{nt'(s)}{2t(s)}+\frac{(x(s)-y)\cdot x'(s)}{2t(s)}-\frac{t'(s)|x(s)-y|^{2}}{4t^{2}(s)}\right]^{2}-\frac{nt''(s)t(s)-nt'(s)t'(s)}{2t^{2}(s)}\notag\\
&\qquad \quad\quad +\frac{[|x'(s)|^{2}+(x(s)-y)\cdot x''(s)]t(s)-t'(s)(x(s)-y)\cdot x'(s)}{2t^{2}(s)}\notag\\
&\qquad \quad\quad -\frac{[t''(s)|x(s)-y|^{2}+2t'(s)(x(s)-y)\cdot x'(s)]t(s)-2t'(s)t'(s)|x(s)-y|^{2}}{4t^{3}(s)}\Big\}dy\notag.
\end{align}
The proof is finished by substituting $s=0$ into the above formula.
\end{proof}
In particular, if $x(0)=0, t(0)=-1$, then the second variation formula is given by the following result.
\begin{lemma}
Let $w$ be a bounded smooth function. Assume $x(s), t(s)$ are variations such that
\[x(0)=0,\quad t(0)=-1,\quad x'(0)=y_{0},\quad t'(0)=h\]
 and $x''(0)=y_{0}'$, $t''(0)=h'$, then
\begin{align}\label{Thegeneralsecondvariationformula}
&\frac{\partial^{2}}{\partial s^2} (F_{x(s), t(s)}(w+s\phi))|_{s=0}\notag\\
=&\frac{p+1}{2(p-1)}(-h'+\frac{2}{p-1}h^{2})\int_{\mathbb{R}^{n}}|\nabla w|^{2}\rho dy\notag\\
&-\frac{1}{p-1}(-h'+\frac{2}{p-1}h^{2})\int_{\mathbb{R}^{n}}|w|^{p+1}\rho dy\notag\\
&-\frac{1}{(p-1)^{2}}(h'+\frac{p-3}{p-1}h^{2})\int_{\mathbb{R}^{n}}w^{2}\rho dy\notag\\
&-\frac{2(p+1)}{p-1}h\int_{\mathbb{R}^{n}}(\nabla w\cdot\nabla\phi)\rho dy\notag\\
&+\frac{2(p+1)}{p-1}h\int_{\mathbb{R}^{n}}|w|^{p-1}w\phi\rho dy-\frac{4}{(p-1)^{2}}h\int_{\mathbb{R}^{n}}w\phi \rho dy\notag\\
&-\frac{p+1}{p-1}h\int_{\mathbb{R}^{n}}|\nabla w|^{2}\rho[\frac{nh}{2}+\frac{y\cdot y_{0}}{2}-\frac{h|y|^{2}}{4}]dy\notag\\
&+\frac{2}{p-1}h\int_{\mathbb{R}^{n}}|w|^{p+1}\rho[\frac{nh}{2}+\frac{y\cdot y_{0}}{2}-\frac{h|y|^{2}}{4}]dy\\
&-\frac{2}{(p-1)^{2}}h\int_{\mathbb{R}^{n}}w^{2}\rho[\frac{nh}{2}+\frac{y\cdot y_{0}}{2}-\frac{h|y|^{2}}{4}]dy\notag\\
&+\int_{\mathbb{R}^{n}}|\nabla\phi|^{2}\rho dy-p\int_{\mathbb{R}^{n}}|w|^{p-1}\phi^{2}\rho dy+\frac{1}{p-1}\int_{\mathbb{R}^{n}}\phi^{2}\rho dy\notag\\
&+2\int_{\mathbb{R}^{n}}(\nabla w\cdot\nabla\phi)\rho[\frac{nh}{2}+\frac{y\cdot y_{0}}{2}-\frac{h|y|^{2}}{4}]dy\notag\\
&-2\int_{\mathbb{R}^{n}}|w|^{p-1}w\phi \rho[\frac{nh}{2}+\frac{y\cdot y_{0}}{2}-\frac{h|y|^{2}}{4}]dy\notag\\
&+\frac{2}{p-1}\int_{\mathbb{R}^{n}}w\phi\rho[\frac{nh}{2}+\frac{y\cdot y_{0}}{2}-\frac{h|y|^{2}}{4}]dy\notag\\
&+\frac{1}{2}\int_{\mathbb{R}^{n}}|\nabla w|^{2} \rho\{[\frac{nh}{2}+\frac{y\cdot y_{0}}{2}-\frac{h|y|^{2}}{4}]^{2}+\frac{nh'+nh^{2}}{2}\notag\\
&+\frac{-|y_{0}|^{2}+y\cdot y_{0}'+hy\cdot y_{0}}{2}-\frac{h'|y|^{2}-2hy\cdot y_{0}+2h^{2}|y|^{2}}{4}\}dy\notag\\
&-\frac{1}{p+1}\int_{\mathbb{R}^{n}}|w|^{p+1}\rho\{[\frac{nh}{2}+\frac{y\cdot y_{0}}{2}-\frac{h|y|^{2}}{4}]^{2}+\frac{nh'+nh^{2}}{2}\notag\\
&+\frac{-|y_{0}|^{2}+y\cdot y_{0}'+hy\cdot y_{0}}{2}-\frac{h'|y|^{2}-2hy\cdot y_{0}+2h^{2}|y|^{2}}{4}\}dy\notag\\
&+\frac{1}{2(p-1)}\int_{\mathbb{R}^{n}}w^{2}\rho\{[\frac{nh}{2}+\frac{y\cdot y_{0}}{2}-\frac{h|y|^{2}}{4}]^{2}+\frac{nh'+nh^{2}}{2}\notag\\
&+\frac{-|y_{0}|^{2}+y\cdot y_{0}'+hy\cdot y_{0}}{2}-\frac{h'|y|^{2}-2hy\cdot y_{0}+2h^{2}|y|^{2}}{4}\}dy\notag.
\end{align}
\end{lemma}
\subsection{The second variation of self similar solutions}
 In this section, we will specialize our calculations from the previous section to the case where  $w$ satisfies \eqref{SC1-2}. In this case, by using \eqref{SC1-2},  the second order variation formula can be simplified.
\begin{theorem}\label{secondvariation}
Let $x(s), t(s)$ are variations of $0, -1$ with $x'(0)=y_{0}, t'(0)=h$ and $x''(0)=y_{0}', t''(0)=h'$. If $w$ is a bounded solution of \eqref{SC1-2}, then for any $\phi\in C_{0}^{\infty}(\mathbb{R}^{n})$,
\begin{align}\label{secondvariation1}
&\frac{\partial^{2}}{\partial s^2} (F_{x(s), t(s)}(w+s\phi))|_{s=0}\notag\\
=&\int_{\mathbb{R}^{n}}|\nabla\phi|^{2}\rho dy+\frac{1}{p-1}\int_{\mathbb{R}^{n}}\phi^{2}\rho dy-p\int_{\mathbb{R}^{n}}|w|^{p-1}\phi^{2}\rho dy\notag\\
&+h\int_{\mathbb{R}^{n}}\left(\frac{2}{p-1}w+y\cdot\nabla w\right)\phi \rho dy-\int_{\mathbb{R}^{n}}\phi(\nabla w\cdot y_{0})\rho dy\\
&-\frac{1}{2}\int_{\mathbb{R}^{n}}(\nabla w\cdot y_{0})(\nabla w\cdot y_{0})\rho dy-h^{2}\int_{\mathbb{R}^{n}}(\frac{1}{p-1}w+\frac{y}{2}\cdot \nabla w)^{2}\rho dy.\notag
\end{align}
In particular, we have
\begin{align}\label{secondvariation0}
&\frac{\partial^{2}}{\partial s^2} (F_{x(s), t(s)}(w+s\phi))|_{s=0}\notag\\
\leq&\int_{\mathbb{R}^{n}}|\nabla\phi|^{2}\rho dy+\frac{1}{p-1}\int_{\mathbb{R}^{n}}\phi^{2}\rho dy-p\int_{\mathbb{R}^{n}}|w|^{p-1}\phi^{2}\rho dy\\
&+h\int_{\mathbb{R}^{n}}\left(\frac{2}{p-1}w+y\cdot\nabla w\right)\phi \rho dy-\int_{\mathbb{R}^{n}}\phi(\nabla w\cdot y_{0})\rho dy\notag.
\end{align}
\end{theorem}
\begin{proof}
This is proved by combining several identies together. We will postpone the proof of this result to Appendix A.
\end{proof}
The second variation formula above allows us to formulate a notion of stability.
\begin{definition}
Let $w$ be a bounded solution of \eqref{SC1-2}.  If for every $\phi\in C_{0}^{\infty}(\mathbb{R}^{n})$, there exists variations $x(s), t(s)$ with 
\[x(0)=0,~~ t(0)=-1,~~ x'(0)=y_{0}, ~~t'(0)=h, ~~x''(0)=y_{0}',~~ t''(0)=h'\]
such that
\[\frac{\partial^{2}}{\partial s^2} (F_{x(s), t(s)}(w+s\phi))|_{s=0}\geq 0,\]
then we say that $w$ is $F-$stable.
\end{definition}
Roughly speaking, a bounded solution of \eqref{SC1-2} is   $F-$stable if modulo translations and dilations, the second derivative of the $F-$ functional is non-negative for all variations at the given  solution.
\begin{remark}
From \cite{Giga-Kohn1985}, we know that the energy functional of \eqref{SC1} is
 \[F(w)=\frac{1}{2}\int_{\mathbb{R}^{n}}|\nabla w|^{2}\rho dy+\frac{1}{2(p-1)}\int_{\mathbb{R}^{n}} w^{2}\rho dy-\frac{1}{p+1}\int_{\mathbb{R}^{n}} |w|^{p+1}\rho dy\]
 Thus we can also define a notion of stability as follows: $w$ is stable if  for any $\phi\in C_{0}^{\infty}(\mathbb{R}^{n})$, the quadratic form
 \[Q(\phi, \phi)=\int_{\mathbb{R}^{n}}|\nabla \phi|^{2}\rho dy+\frac{1}{(p-1)}\int_{\mathbb{R}^{n}} \phi^{2}\rho dy-p\int_{\mathbb{R}^{n}} |w|^{p-1}\phi^{2}\rho dy\]
is nonnegative definite.  However, it is easy to check that if we use this notation, then any nonzero self similar solution is unstable. Therefore, this kind of   stability can not provide any useful information.
\end{remark}

\section{Characterization of the constant solution}
Our next objective is to classify $F-$stable self similar solutions. Before doing this, we first prove a result concerning the constant solutions of \eqref{SC1-2}.
\begin{proposition}\label{ChCS}
Let $w$ be a bounded solution of \eqref{SC1-2}, then $w$ is the constant solution of \eqref{SC1-2} if and only if the function
\begin{equation}\label{definefunctioneta}
\Lambda(w)(y)=\frac{2}{p-1}w(y)+ y\cdot\nabla w(y)
\end{equation}
does not change sign in $\mathbb{R}^{n}$.
\end{proposition}
\begin{remark}
In a recent paper \cite{Choi-Huang2024}, Choi and Huang proved a similar result for smooth linearly stable self-similar solutions \footnote{The precise definition of linearly stable self similar solution can be found in \cite[Definition 2.1]{Choi-Huang2024}.} of the equation \eqref{Cauchyproblem} under an integral condition for all $p>1$.
\end{remark}
In order to give the proof of Proposition \ref{ChCS}, we first recall the weighted spaces
\begin{equation*}\label{18}
L_\omega^q(\mathbb R^n)=\{g\in L_{loc}^q(\mathbb R^n):\int_{\mathbb R^n}|g(y)|^qe^{-\frac{|y|^2}{4}}dy<\infty\}
\end{equation*}
and
\begin{equation*}\label{19}
H_\omega^1(\mathbb R^n)=\{g\in L_{\omega}^2(\mathbb R^n): |\nabla g|+|g|\in L_{\omega}^2(\mathbb R^n)\}.
\end{equation*}
The inner product on $L_\omega^2(\mathbb R^n)$ is defined by
\begin{equation}\label{innerproduct}
\langle\psi_{1}, \psi_{2}\rangle_{w}=\int_{\mathbb{R}^{n}}\psi_{1}\psi_{2}e^{-\frac{|y|^{2}}{4}}dy.
\end{equation}
Then both $L_\omega^2(\mathbb R^n)$ and $H_\omega^1(\mathbb R^n)$ are Hilbert spaces. 

For any bounded solution of \eqref{SC1-2}, we define the linear operator
\begin{equation}\label{linearizew}
L\psi=\Delta \psi-\frac{y}{2}\cdot\nabla\psi-\frac{1}{p-1}\psi+p|w|^{p-1}\psi.
\end{equation}
Recall that $\lambda\in\mathbb{R}$ is an eigenvalue of $L$ if there is a non-zero function $f\in H_\omega^1(\mathbb R^n)$ such that $Lf+\lambda f=0$. The operator $L$ is not self adjoint with respect to the usual $L^{2}$ inner product because of the first order term.
However, using integration by parts, we easily see that $L$ is a self adjoint operator
 in $L_\omega^2(\mathbb R^n)$ with domain $D(L):=H_{\omega}^1(\mathbb R^n)$.  Since
  the natural embedding $\iota:H_{\omega}^1(\mathbb R^n)\hookrightarrow L_{\omega}^2(\mathbb R^n)$ is compact, standard spectral theory gives the
following corollary.
\begin{corollary}\label{cortwoeigenfunctions}
Assume $w$ is a bounded solution of \eqref{SC1-2} and $L$ is defined by \eqref{linearizew}. Then
\begin{itemize}

\item[(1)] $L$ has a sequence of eigenvalues $\lambda_{1}<\lambda_{2}\leq\cdots.$

\item[(2)] There is an orthogonal basis $\{f_{k}\}$ for $L_{w}^{2}(\mathbb{R}^{n})$ with $Lf_{k}+\lambda_{k}f_{k}=0$.

\item[(3)] The smallest eigenvalue $\lambda_{1}$ is simiple and can be characterized by
\begin{equation}\label{thevariationalcharacterization}
\lambda_{1}=\inf_{\psi\in H_{w}^{1}(\mathbb{R}^{n})}\frac{\int_{\mathbb{R}^{n}}|\nabla \psi|^{2}\rho dy+\frac{1}{p-1}\int_{\mathbb{R}^{n}}\psi^{2}\rho dy-p\int_{\mathbb{R}^{n}}|w|^{p-1}\psi^{2}\rho dy}{\int_{\mathbb{R}^{n}}\psi^{2}\rho dy}.
\end{equation}

\item[(4)] Any eigenfunction associated to $\lambda_{1}$ does not change sign.
\end{itemize}
\end{corollary}
The next lemma shows that the operator $L$  has two explicit eigenfunctions which are induced by scaling and translations.
\begin{lemma}\label{lemeig}
If $w$ is a bounded solution of the equation \eqref{SC1-2}, then
\begin{align}
\label{2023.10.18-1}&Lw_{i}=\frac{1}{2}w_{i},\quad i=1,2,\cdots, n,\\
\label{2023.10.18-2}&L\left(\frac{2}{p-1}w+y\cdot\nabla w\right)=\frac{2}{p-1}w+y\cdot\nabla w.
\end{align}
\end{lemma}
\begin{proof}
Since $w$ satisfies \eqref{SC1},   for $i=1,2,\cdots, n$,
\[\Delta w_{i}-\frac{y}{2}\cdot\nabla w_{i}-\frac{1}{2}w_{i}-\frac{1}{p-1}w_{i}+p|w|^{p-1}w_{i}=0.\]
Hence we get \eqref{2023.10.18-1}. 

To show \eqref{2023.10.18-2},  consider the variation $w_{\lambda}(y)=\lambda^{2/(p-1)}w(\lambda y)$. Since $w$ satisfies \eqref{SC1},   $w_{\lambda}$ satisfies 
\begin{equation}\label{eqsca}
\Delta w_{\lambda}-\frac{\lambda^{2}}{2}y\cdot \nabla w_{\lambda}-\frac{\lambda^{2}}{p-1}w_{\lambda}
+|w_{\lambda}|^{p-1}w_{\lambda}=0.
\end{equation}
Taking derivative with respect to $\lambda$ at $\lambda=1$, we obtain
\begin{equation}\label{eqsca1}
\begin{aligned}
0=&\Delta \left(\frac{2}{p-1}w+y\cdot\nabla w\right)-\frac{y}{2}\cdot\nabla \left(\frac{2}{p-1}w+y\cdot\nabla w\right)\\
&-\frac{1}{p-1}\left(\frac{2}{p-1}w+y\cdot\nabla w\right)+p|w|^{p-1}\left(\frac{2}{p-1}w+y\cdot\nabla w\right)\\
&-\left(\frac{2}{p-1}w+y\cdot\nabla w\right),
\end{aligned}
\end{equation}
which is  \eqref{2023.10.18-2}.
\end{proof}
With Corollary \ref{cortwoeigenfunctions} and Lemma \ref{lemeig} in hand, we now have all of the tools that we need to prove Proposition \ref{ChCS}.
\begin{proof}[Proof of Proposition \ref{ChCS}]
If $w$ is the constant solution of \eqref{SC1}, then it is clear that either $\Lambda(w)\equiv 0$ or
 \[\Lambda(w)(y)=\pm2(\frac{1}{p-1})^{\frac{p}{p-1}}.\]
Thus $\Lambda(w)$ does not change sign in $\mathbb{R}^{n}$.

 Next, we need to prove that  if $\Lambda(w)$ does not change sign, then $w$ is the constant solution of  \eqref{SC1-2}.

 Assume that $\Lambda(w)$ does not change sign. By Lemma \ref{lemeig} and the strong maximum principle,   either $\Lambda(w)\equiv 0$ or up to a sign, $\Lambda(w)$ is positive in $\mathbb{R}^{n}$. If the first case holds ($\Lambda(w)\equiv 0$), then $w$ is a $2/(p-1)-$ homogeneous solution of 
 \[\Delta w+|w|^{p-1}w=0,\quad\text{in}~\mathbb{R}^{n}\backslash\{0\}.\]
 Since we have assumed that $w$ is bounded, then $w$ is smooth in $\mathbb{R}^{n}$. The homogeneity then implies that $w\equiv 0$.

 If the second case holds, then we know from  Lemma \ref{lemeig} and the last statement in Corollary \ref{cortwoeigenfunctions} that $-1$ is the smallest eigenvalue of the operator $L$ defined in \eqref{linearizew}. Therefore, for any $\phi\in H^{1}_{w}(\mathbb{R}^{n})$,
\begin{equation}\label{ChCS1}
\begin{aligned}
0\leq&\int_{\mathbb{R}^{n}}|\nabla\phi|^{2}\rho dy+\frac{1}{p-1}\int_{\mathbb{R}^{n}}\phi^{2}\rho dy-p\int_{\mathbb{R}^{n}}|w|^{p-1}\phi^{2}\rho dy+\int_{\mathbb{R}^{n}}\phi^{2}\rho dy.
\end{aligned}
\end{equation}
By taking $\phi=w$ into \eqref{ChCS1}, we get
\begin{equation}\label{ChCS2}
\begin{aligned}
0\leq&\int_{\mathbb{R}^{n}}|\nabla w|^{2}\rho dy+\frac{1}{p-1}\int_{\mathbb{R}^{n}}w^{2}\rho dy-p\int_{\mathbb{R}^{n}}|w|^{p+1}\rho dy+\int_{\mathbb{R}^{n}}w^{2}\rho dy.
\end{aligned}
\end{equation}
Multiplying  both sides of \eqref{SC1} by $w\rho$ and  integrating by parts, we get
\begin{equation}\label{ChCS3}
\begin{aligned}
0=&\int_{\mathbb{R}^{n}}|\nabla w|\rho dy+\frac{1}{p-1}\int_{\mathbb{R}^{n}}w^{2}\rho dy-\int_{\mathbb{R}^{n}}|w|^{p+1}\rho dy.
\end{aligned}
\end{equation}
Combining \eqref{ChCS2} and \eqref{ChCS3}, we have
\begin{equation}\label{ChCS4}
(1-p)\int_{\mathbb{R}^{n}}|w|^{p+1}\rho dy+\int_{\mathbb{R}^{n}}w^{2}\rho dy\geq 0.
\end{equation}
On the other hand, we get from \eqref{ChCS3} that
\begin{equation}\label{ChCS5}
(1-p)\int_{\mathbb{R}^{n}}|w|^{p+1}\rho dy+\int_{\mathbb{R}^{n}}w^{2}\rho dy=-(p-1)\int_{\mathbb{R}^{n}}|\nabla w|^{2}\rho dy\leq 0.
\end{equation}
Hence
\begin{equation}\label{ChCS6}
(1-p)\int_{\mathbb{R}^{n}}|w|^{p+1}\rho dy+\int_{\mathbb{R}^{n}}w^{2}\rho dy=0.
\end{equation}
Plugging \eqref{ChCS6} into \eqref{ChCS3}, we get
\begin{equation}\label{ChCS7}
\int_{\mathbb{R}^{n}}|\nabla w|^{2}\rho dy=0.
\end{equation}
Thus $w$ is a constant function.
\end{proof}
\begin{remark}
The proof of Proposition \ref{ChCS} indicates that a much more suitable choice to define ``Morse index'' of solutions of \eqref{SC1-2} is the number of eigenvalues of $L$ less than $1$.
\end{remark}
\begin{remark}
If we do not assume $w$ is a bounded solution of \eqref{SC1-2}, then
\[w(y)=\left[\frac{2}{p-1}(n-2-\frac{2}{p-1})\right]^{\frac{1}{p-1}}|y|^{-\frac{2}{p-1}}\]
is a solution of \eqref{SC1-2} with $\Lambda(w)(y)\equiv 0$ on $\mathbb{R}^{n}\backslash\{0\}$.
\end{remark}

\section{Classification of $F$-stable self similar solutions}
In this section, we will combine the second variation formula  and the characterization of the constant solution to prove that the constant solution is the only candidate for a bounded solution of \eqref{SC1-2} to be $F-$ stable.
\begin{theorem}\label{thestable}
If $w$ is a bounded solution of \eqref{SC1-2} and $w$ is not the constant solution, then $w$ is not $F$-stable.
\end{theorem}
\begin{proof}
If $w$ is not the constant solution of \eqref{SC1-2}, then the function $\Lambda(w)$ defined in \eqref{definefunctioneta} can not
vanish identically. Moreover, we know from  Lemma \ref{lemeig} that $\Lambda(w)$ is an eigenfunction for \[L=\Delta -\frac{y}{2}\cdot\nabla-\frac{1}{p-1}+p|w|^{p-1}\]
associated with eigenvalue $-1$. By Proposition \ref{ChCS},  $\Lambda(w)$ changes sign. Since $\Lambda(w)$ changes sign, then the last statement in Corollary \ref{cortwoeigenfunctions} yields $-1$ is not the smallest eigenvalue for $L$.  Thus, we can get a positive function $f$
with $Lf +\lambda_{1} f=0$, where $\lambda_{1}<-1$ is the smallest eigenvalue for $L$. Since $L$ is self-adjoint in the weighted space $H^{1}_{w}(\mathbb{R}^{n})$,   $f$ is orthogonal to the eigenfunctions with different eigenvalues. In particular,  we have
\begin{equation}\label{themain1}
\int_{\mathbb{R}^{n}}\left(\frac{2}{p-1}w+y\cdot\nabla w\right)f \rho dy=0
\end{equation}
and for any $y_{0}\in\mathbb{R}^{n}$
\begin{equation}\label{themain1'}
\int_{\mathbb{R}^{n}}(\nabla w\cdot y_{0}) f\rho dy=0.
\end{equation}
Substituting \eqref{themain1} and \eqref{themain1'} into \eqref{secondvariation0} gives
\begin{align}\label{themain2}
&\frac{\partial^{2}}{\partial s^2} (F_{x(s), t(s)}(w+sf))|_{s=0}\notag\\
=&\int_{\mathbb{R}^{n}}|\nabla f|^{2}\rho dy+\frac{1}{p-1}\int_{\mathbb{R}^{n}}f^{2}\rho dy-p\int_{\mathbb{R}^{n}}|w|^{p-1}f^{2}\rho dy\notag\\
&+h\int_{\mathbb{R}^{n}}\left(\frac{2}{p-1}w+y\cdot\nabla w\right)f \rho dy-\int_{\mathbb{R}^{n}}f(\nabla w\cdot y_{0})\rho dy\\
&-\frac{1}{2}\int_{\mathbb{R}^{n}}(\nabla w\cdot y_{0})(\nabla w\cdot y_{0})\rho dy-h^{2}\int_{\mathbb{R}^{n}}(\frac{1}{p-1}w+\frac{y}{2}\cdot \nabla w)^{2}\rho dy\notag\\
\leq &-\lambda_{1}\int_{\mathbb{R}^{n}}f^{2}\rho dy\notag.
\end{align}
It follows that
\[\frac{\partial^{2}}{\partial s^2} (F_{x(s), t(s)}(w+sf))|_{s=0}<0\]
for any choice of $h$ and $y_{0}$. By the definition, $w$ is not $F$-stable, thus the proof is completed.
\end{proof}
\begin{remark}
In the proof of Theorem \ref{thestable}, the variation we can choose  is not unique. Indeed, assume $f$ is the first positive eigenfunction of $L$  and $\xi$ is a small positive constant which will be determined later. Let us take $\phi=f+\xi$ into the second  variation formula.  Without loss of generality, we may normalize $f$ so that $\int_{\mathbb{R}^{n}}f^{2}\rho dy=1$. Then
\begin{align}\label{themainremark}
&\frac{\partial^{2}}{\partial s^2} (F_{x(s), t(s)}(w+s(f+\xi)))|_{s=0}\notag\\
=&\int_{\mathbb{R}^{n}}|\nabla f|^{2}\rho dy-p\int_{\mathbb{R}^{n}}|w|^{p-1}f^{2}\rho dy\notag+\frac{1}{p-1}\int_{\mathbb{R}^{n}}f^{2}\rho dy\notag\\
&+\frac{2\xi}{p-1}\int_{\mathbb{R}^{n}} f\rho dy+\frac{\xi^{2}}{p-1}\int_{\mathbb{R}^{n}}\rho dy\notag\\
&-2p\xi \int_{\mathbb{R}^{n}}|w|^{p-1}f\rho dy -p\xi^{2}\int_{\mathbb{R}^{n}}|w|^{p-1}\rho dy\\
&-\xi\int_{\mathbb{R}^{n}}(\nabla w\cdot y_{0})\rho dy-\frac{1}{2}\int_{\mathbb{R}^{n}}(\nabla w\cdot y_{0})(\nabla w\cdot y_{0})\rho dy\notag\\
&+2h\xi\int_{\mathbb{R}^{n}}(\frac{1}{p-1}w+\frac{1}{2} y\cdot\nabla w)\rho dy-h^{2}\int_{\mathbb{R}^{n}}(\frac{1}{p-1}w+\frac{1}{2} y\cdot\nabla w)^{2}\rho dy\notag\\
\leq &\lambda_{1}\int_{\mathbb{R}^{n}}f^{2}\rho dy+\frac{3(p+1)}{p-1}(\xi^{2}+\xi)\notag.
\end{align}
 If we choose $\xi$ so that
 \[\frac{3(p+1)}{p-1}(\xi^{2}+\xi)<\lambda_{1},\]
 then
\[\frac{\partial^{2}}{\partial s^2} (F_{x(s), t(s)}(w+s(f+\xi)))|_{s=0}<0\]
for any choice of $h$ and $y_{0}$.
\end{remark}
In view of Theorem \ref{thestable}, it will be   natural to ask whether the constant solution is $F-$ stable. By the second variation formula, it is clear that $0$ is $F-$stable.
Next, we will show if $w$ is the positive solution of \eqref{SC1-2},  then the only  way to decrease the $F_{0, 1}$ functional is to translate in space; this will not be used elsewhere.
\begin{proposition}\label{clascons}
If $w$ is the positive constant solution of \eqref{SC1-2}, then for any function $\phi$ such that
\[\int_{\mathbb{R}^{n}}y_{i}\phi\rho dy=0,\quad i=1, 2, \cdots n,\]
there exist $y_{0}, h$ such that
\[\frac{\partial^{2}}{\partial s^2} (F_{x(s), t(s)}(w+s\phi))|_{s=0}\geq 0.\]
\end{proposition}
\begin{proof}
By the assumption, \[w=\kappa=\left(\frac{1}{p-1}\right)^{\frac{1}{p-1}}.\]
Plugging this into \eqref{secondvariation1}, we have
\begin{equation}
\begin{aligned}
&\frac{\partial^{2}}{\partial s^2} (F_{x(s), t(s)}(w+s\phi))|_{s=0}\\
=&\int_{\mathbb{R}^{n}}|\nabla\phi|^{2}\rho dy-\int_{\mathbb{R}^{n}}\phi^{2}\rho dy+\frac{2hw}{p-1}\int_{\mathbb{R}^{n}}\phi\rho dy-\frac{h^{2}w^{2}}{(p-1)^{2}}\int_{\mathbb{R}^{n}}\rho dy
\end{aligned}
\end{equation}
Let $L_{0}$ be the operator defined by
$$L_{0}\psi=\Delta \psi-\frac{1}{2} y\cdot\nabla \psi.$$
For the operator $L_{0}$, we have the following classical result on its eigenvalues.
\begin{lemma}\label{lemsp}
The eigenvalues of the operator $L_{0}$ are given by \[\lambda_{k}=-|\alpha|/2, \quad \alpha=(\alpha_{1}, \alpha_{2}, \cdots, \alpha_{n}),\]
where $\alpha_{i}$ is a nonnegative integer for $1\leq i\leq n$.
\end{lemma}

Let $a$ be a constant such that
\begin{align}
\int_{\mathbb{R}^{n}}(\phi-a)\rho dy=0
\end{align}
and let $\phi_{0}(y)=\phi(y)-a$. Then
\begin{align}
&\frac{\partial^{2}}{\partial s^2} (F_{x(s), t(s)}(w+s\phi))|_{s=0}\notag\\
=&\int_{\mathbb{R}^{n}}|\nabla\phi_{0}|^{2}\rho dy-\int_{\mathbb{R}^{n}}\phi_{0}^{2}\rho dy\notag\\
&+2\frac{h}{p-1}(\frac{1}{p-1})^{\frac{1}{p-1}}a-\frac{h^{2}}{(p-1)^{2}}(\frac{1}{p-1})^{\frac{2}{p-1}}-a^{2}\notag\\
=&-(\frac{h}{p-1}(\frac{1}{p-1})^{\frac{1}{p-1}}-a)^{2}+\int_{\mathbb{R}^{n}}|\nabla\phi_{0}|^{2}\rho dy-\int_{\mathbb{R}^{n}}\phi_{0}^{2}\rho dy\notag.
\end{align}
If we choose $h$ so that
\[\frac{h}{p-1}(\frac{1}{p-1})^{\frac{1}{p-1}}-a=0,\]
then Lemma \ref{lemsp} implies
\[\frac{\partial^{2}}{\partial s^2} (F_{x(s), t(s)}(w+s\phi))|_{s=0}\geq 0.\]
Hence the proof of Proposition \ref{clascons} is complete.
\end{proof}
\begin{remark}
A similar result holds for $-\kappa$.
\end{remark}

\section{Entropy}
The entropy $\lambda(w)$ of a bounded smooth function $w$ is defined to be
\[\lambda(w)=\sup\limits_{x_{0}\in\mathbb{R}^{n}, t_{0}\in(-\infty, 0)}F_{x_{0}, t_{0}}(w).\]
\begin{remark}
The notion of entropy is used by  Colding and Minicozzi (see \cite{Colding-Minicozzi2012}) in mean curvature flow to classify generic singularities. This quantity can be used to measure the complexity of self shrinkers. As discussed in \cite{Colding-Minicozzi2012}, the  advantage of the entropy functional is that it is invariant under dilations, rotations, or translations of $w$. The main disadvantage of the entropy is that for a variation $w_{s}$, it need not depend smoothly on $s$.
\end{remark}
\begin{definition}
We will say that a bounded function $w$ is entropy stable if it is a local minimum for the entropy functional.
\end{definition}
\subsection{The entropy is achieved for bounded self similar solutions}
Although in the definition of entropy, the supremum is over a noncompact space-time domain, the next lemma shows that for a bounded solution of \eqref{SC1-2}, the function $(x_{0}, t_{0})\rightarrow F_{x_{0}, t_{0}}(w)$ has a global maximum at $x_{0}=0, t_{0}=-1$.
\begin{lemma}\label{CSP}
If $w$ is a bounded solution of \eqref{SC1-2}, then $\lambda(w)$ is achieved at $(0, -1)$.
\end{lemma}
\begin{proof}
 For any  bounded solution of \eqref{SC1-2},  set $u(x, t)=(-t)^{-1/(p-1)}w(x/\sqrt{-t})$. It follows from Lemma \ref{monotonicityformula} that for any $(x, t)\in\mathbb{R}^{n}\times(-\infty, 0)$ and $T>t$,
\begin{align}
E(s; x, T, u)=&\frac{1}{2}(-s)^{\frac{p+1}{p-1}}\int_{\mathbb{R}^{n}}|\nabla u(y, T+s)|^{2}G(y-x, s)dy\notag\\
&-\frac{1}{p+1}(-s)^{\frac{p+1}{p-1}}\int_{\mathbb{R}^{n}}|u(y, T+s)|^{p+1}G(y-x, s)dy\notag\\
&+\frac{1}{2(p-1)}(-s)^{\frac{2}{p-1}}\int_{\mathbb{R}^{n}}u(y, T+s)^{2}G(y-x, s)dy\notag.
\end{align}
is nonincreasing with respect to $s$ in $(-\infty, -(T-t)]$. By the definition of $u$, we have
\begin{align}
E(s; x, T, u)=&\frac{1}{2}\left(\frac{s}{T+s}\right)^{\frac{p+1}{p-1}}\int_{\mathbb{R}^{n}}|\nabla w(y)|^{2}G\left(y-\frac{x}{\sqrt{-(T+s)}}, -\frac{s}{T+s}\right)dy\notag\\
&-\frac{1}{p+1}\left(\frac{s}{T+s}\right)^{\frac{p+1}{p-1}}\int_{\mathbb{R}^{n}}|w(y)|^{p+1}G\left(y-\frac{x}{\sqrt{-(T+s)}}, -\frac{s}{T+s}\right)dy\notag\\
&+\frac{1}{2(p-1)}\left(\frac{s}{T+s}\right)^{\frac{2}{p-1}}\int_{\mathbb{R}^{n}}w(y)^{2}G\left(y-\frac{x}{\sqrt{-(T+s)}}, -\frac{s}{T+s}\right)dy\notag\\
=&F_{\frac{x}{\sqrt{-(T+s)}}, -\frac{s}{T+s}}(w).\notag
\end{align}
We take
\[T-t=1,\quad x=x_{0}\sqrt{-\frac{1}{t_{0}}}, \quad T=1+\frac{1}{t_{0}}.\]
Since $E(s; x, t, u)$ is nonincreasing with respect to $s$ in $(-\infty, -1]$, then
\[E(-1; x, T, u)=F_{x_{0}, t_{0}}(w)\leq\lim_{s\rightarrow-\infty}E(s; x, T, u)=F_{0, -1}(w).\]
Since $(x_{0}, t_{0})\in\mathbb{R}^{n}\times(-\infty, 0)$ can be arbitrary, we conclude that $\lambda(w)$ achieves its maximum at $(0, -1)$.
\end{proof}
By the definition of the weighted energy $E(w)$  and the $F-$ functional, it is easy to see that $E(w)=F_{0, -1}(w)$. As a direct consequence of Lemma \ref{CSP}, we have the following.
\begin{corollary}\label{CCSP}
If $w$ is a bounded solution of \eqref{SC1-2}, then $\lambda(w)=E(w).$
\end{corollary}
Lemma \ref{CSP} only tells us the location where the $F-$ functional achieves the maximum. For our application later, we also need  to show that if $w$ is a bounded solution of \eqref{SC1-2} and
\[\nabla w\cdot y_{0}\neq 0,\quad\text{for any}~y_{0}\in\mathbb{R}^{n}\backslash\{0\},\]
then $(x_{0}, t_{0})\to F_{x_{0}, t_{0}}(w)$ has a strict global maximum at $x_{0}=0, t_{0}=-1$.
\begin{lemma}\label{gapestimate}
Suppose $w$ is a bounded solution of \eqref{SC1-2} and for any $y_{0}\in\mathbb{R}^{n}\backslash\{0\}$, $w$ is not translation invariant in the $y_{0}$ direction.
Then for every $\epsilon>0$ sufficiently small, there exists $\delta>0$ such that
\begin{equation}\label{gap}
\sup\{F_{x_{0}, t_{0}}(w):|x_{0}|+|\log t_{0}|>\epsilon\}\leq\lambda(w)-\delta.
\end{equation}
\end{lemma}
\begin{proof}
Since $w$ satisfies \eqref{SC1-2},   $(0, 1)$ is a critical point of the $F$-functional. Moreover, the second variation
formula with $\phi\equiv 0$  gives that the second derivative of $F_{x(s), t(s)}(w)$ at $s=0$ along the paths
\[x(s)=sy_{0},\quad t(s)=-(1+sh)\]
is given by
\begin{align}\label{spacetimevariation}
&\frac{\partial^{2}}{\partial s^2} (F_{x(s), t(s)}(w+s\phi))|_{s=0}\notag\\
=&-\frac{1}{2}\int_{\mathbb{R}^{n}}(\nabla w\cdot y_{0})(\nabla w\cdot y_{0})\rho \notag dy-h^{2}\int_{\mathbb{R}^{n}}\left(\frac{1}{p-1} w+\frac{y}{2}\cdot\nabla w\right)^{2}\rho\notag dy.
\end{align}
This expression is clear non-positive. Moreover, the last line vanishes only when
\[h=0,~~~\nabla w\cdot y_{0}\equiv 0\]
 or
 \[\nabla w\cdot y_{0}\equiv 0,\quad \frac{1}{p-1} w+\frac{y}{2}\cdot\nabla w\equiv 0.\]
 If $y_{0}\neq 0, h=0$, then in both cases $w$ is indeed a function in $\mathbb{R}^{n-1}$. By our assumption, this can not happen. If $y_{0}=0, h\neq 0$, then
 \[\frac{\partial^{2}}{\partial s^2} (F_{x(s), t(s)}(w+s\phi))|_{s=0}=0\]
 implies
 \[\frac{1}{p-1} w+\frac{y}{2}\cdot\nabla w\equiv 0.\]
 This is equivalent to  $w$ is $2/(p-1)-$homogeneous. Since we have assumed $w$ is a bounded solution of \eqref{SC1-2}, then $w\equiv 0$, which contradicts our assumption that $w$ is not translation invariant in any direction again. Therefore, we have shown that the second derivative at $(0, 1)$ is strictly negative. In particular,  the function $(x_{0}, t_{0})\rightarrow F_{x_{0}, t_{0}}(w)$ has a strict local maximum at $(0, 1)$. Hence for every $\epsilon>0$ sufficiently small there exists $\delta >0$ such that
\[\max_{(x, t)\in\{ |x|+|\log t|=\epsilon\}}F_{x, t}(w)<\lambda (w)-\delta.\]
For every $(x_{0}, t_{0})$ such that $|x_{0}|+|\log t_{0}|>\epsilon$, it follows from the proof of Lemma \ref{CSP} that for every $s<-1$,
\begin{equation}\label{Mono}
F_{x_{0}, t_{0}}(w)\leq F_{\frac{x_{0}}{\sqrt{t_{0}+1+t_{0}s}}, -\frac{st_{0}}{t_{0}+1+t_{0}s}}(w).
\end{equation}
Since
\[\lim_{s\rightarrow -\infty}\frac{x_{0}}{\sqrt{t_{0}+1+t_{0}s}}=0\]
and
\[\lim_{s\rightarrow-\infty} -\frac{st_{0}}{t_{0}+1+t_{0}s}=-1.\]
Then  there exists $\tilde{s}$ such that
\[\left(\frac{x_{0}}{\sqrt{t_{0}+1+t_{0}s}}, -\frac{st_{0}}{t_{0}+1+t_{0}s}\right)\in \{ |x|+|\log t|=\epsilon\}.\]
By \eqref{Mono}, we know that $F_{x_{0}, t_{0}}(w)\leq\lambda(w)-\delta$. Hence \eqref{gap} holds.
\end{proof}

\subsection{The equivalence of F-stability and entropy-stability}
Our next objective is to prove that if $w$ is a bounded non-constant solution of \eqref{SC1-2} and there exists a positive constant $c$ such that
\begin{equation}\label{lowerbound}
|w|\leq c(1+|y|)^{-\frac{2}{p-1}},\quad\text{in}~\mathbb{R}^{n},
\end{equation}
then $F-$ stability and entropy stability are equivalent for $w$. In the course of the proof, we need a result concerning the first eigenfunction of the operator
 \[L\psi=\Delta\psi-\frac{1}{2} y\cdot\nabla\psi-\frac{1}{p-1}\psi+p|w|^{p-1}\psi.\]
\begin{lemma}\label{lemeigenfunctiondecay}
Assume $w$ is a bounded non-constant solution of \eqref{SC1-2}. Let $\lambda_{1}$ be the first eigenvalue of the eigenvalue problem
\begin{equation}\label{2024-03-08}
\Delta\psi-\frac{1}{2} y\cdot\nabla\psi-\frac{1}{p-1}\psi+p|w|^{p-1}\psi+\lambda\psi=0,\quad\text{in}~\mathbb{R}^{n}.
\end{equation}
 If $f$ is a positive eigenfunction associated to $\lambda_{1}$, then there exists a positive constant $C$ such that
\begin{equation}\label{eigenfunctiondecay}
(1+|y|)^{\frac{2p}{p-1}}f+(1+|y|)^{\frac{2p}{p-1}+1}|\nabla f|\leq C,\quad\text{in}~\mathbb{R}^{n}.
\end{equation}
\end{lemma}
\begin{proof}
Let $\lambda_{1}$ be the smallest eigenvalue of \eqref{2024-03-08} and let $f$ be a positive eigenfunction associated to $\lambda_{1}$. Then $f$ satisfies
\begin{equation}\label{eigenfunction}
\Delta f-\frac{y}{2}\cdot\nabla f-\frac{1}{p-1}f+p|w|^{p-1}f+\lambda_{1}f=0,\quad\text{in}~\mathbb{R}^{n}.
\end{equation}
Since $w$ is a bounded non-constant solution of \eqref{SC1-2},   $\lambda_{1}<1$. Hence 
  \eqref{lowerbound} implies there exists a positive constant $R$ such that that
\begin{equation}\label{condition1}
-\frac{2p}{p-1}(n-2-\frac{2p}{p-1})|y|^{-2}+1+p|w|^{p-1}+\lambda_{1}<0,\quad\text{in}~\mathbb{R}^{n}\backslash B_{R}(0).
\end{equation}
By taking $R$ large enough, we may  also assume that
\begin{equation}\label{supersolution}
-\frac{1}{p-1}+p|w|^{p-1}+\lambda_{1}<0,\quad\text{in}~|y|\geq B_{R}.
\end{equation}
By standard elliptic regularity theory,  there exists a positive constant $M$ such that
\begin{equation}\label{innerestimate}
|f|< MR^{-\frac{2p}{p-1}},\quad\text{in}~B_{R}.
\end{equation}

For any $k\geq 1$, consider the Dirichlet problem
\begin{equation}\label{Dirichletproblem}
\left\{\begin{array}{lll}
\Delta g_{k}-\frac{y}{2}\cdot\nabla g_{k}-\frac{1}{p-1}g_{k}+p|w|^{p-1}g_{k}+\lambda_{1}g_{k}=0,&\quad\text{in}~B_{R+k}\backslash B_{R},\\
g_{k}=f, &\quad\text{on}~\partial B_{R},\\
g_{k}=M|y|^{-\frac{2p}{p-1}},&\quad\text{on}~\partial B_{R+k},
\end{array}
\right.
\end{equation}
Since we have assumed that \eqref{supersolution} holds,  the zeroth order term of the second order elliptic equation in  \eqref{Dirichletproblem} has a negative sign. Notice that
\[(L+\lambda_{1})|y|^{-\frac{2p}{p-1}}=[-\frac{2p}{p-1}(n-2-\frac{2p}{p-1})|y|^{-2}+1+p|w|^{p-1}+\lambda_{1}]|y|^{-\frac{2p}{p-1}}.\]
Thus, it follows from \eqref{condition1}, \cite[Theorem 8.3]{Gi-Tr} and the maximum principle that for any $k\geq 1$, \eqref{Dirichletproblem} has a unique smooth solution $g_{k}$, which is bounded above by $M|y|^{-(2p)/(p+1)}$. By using the maximum principle again, $g_{k}$ is bounded below by $0$. Letting $k\rightarrow\infty$, we get from the Arzel\'{a}-Ascoli theorem that $\{g_{k}\}$ converges to some function $g_{\infty}$ in $C^{2}_{loc}(\mathbb{R}^{n}\backslash B_{R})$, where
\begin{equation}\label{limitDirichletproblem}
\left\{\begin{array}{lll}
\Delta g_{\infty}-\frac{y}{2} \cdot\nabla g_{\infty}-\frac{1}{p-1}g_{\infty}+p|w|^{p-1}g_{\infty}+\lambda_{1}g_{\infty}=0,&\quad\text{in}~\mathbb{R}^{n}\backslash B_{R},\\
g_{\infty}=f, &\quad\text{on}~\partial B_{R}.
\end{array}
\right.
\end{equation}
 Moreover,
\begin{equation}\label{limitdecay}
|g_{\infty}|\leq M|y|^{-\frac{2p}{p-1}},\quad\text{in}~\mathbb{R}^{n}\backslash B_{R}.
\end{equation}
Multiplying  both sides of \eqref{limitDirichletproblem} and integrating over $\mathbb{R}^{n}\backslash B_{R}$, we  get 
\begin{equation}
\int_{\mathbb{R}^{n}\backslash B_{R}}|\nabla g_{\infty}|^{2}\rho dy<\infty.
\end{equation}

We claim that $g_{\infty}=f$.

Indeed, by denoting $h:=f-g_{\infty}$, then $h$ satisfies
\begin{equation}\label{newlinearequation}
\left\{\begin{array}{lll}
\Delta h-\frac{y}{2}\cdot\nabla h-\frac{1}{p-1}h+p|w|^{p-1}h+\lambda_{1}h=0,&\quad\text{in}~\mathbb{R}^{n}\backslash B_{R},\\
h=0,&\quad\text{on}~\partial B_{R}.
\end{array}
\right.
\end{equation}
Let $\phi_{k}$ be a smooth cutoff function such that $\phi_{k}=1$ in $B_{k}$, $\phi_{k}=0$ outsider $B_{2k}$ and $|\nabla\phi_{k}|\leq C/k$. Multiplying both sides of \eqref{newlinearequation} by $h\phi_{k}^{2}\rho$ and integrating by parts, we get
\begin{equation}
\begin{aligned}
0=&-\int_{\mathbb{R}^{n}\backslash B_{R}}|\nabla h|^{2}\phi_{k}^{2}\rho dy-\frac{1}{p-1}\int_{\mathbb{R}^{n}\backslash B_{R}}h^{2}\phi_{k}^{2}\rho dy\\
&+p\int_{\mathbb{R}^{n}\backslash B_{R}}|w|^{p-1} h^{2}\phi_{k}^{2}\rho dy+\lambda_{1}\int_{\mathbb{R}^{n}\backslash B_{R}}h^{2}\phi_{k}^{2}\rho dy\\
&-2\int_{\mathbb{R}^{n}\backslash B_{R}}h\phi_{k}\nabla h\cdot\nabla\phi_{k}\rho dy.
\end{aligned}
\end{equation}
Recall that we have assumes that \eqref{supersolution} holds, so
\[\int_{\mathbb{R}^{n}\backslash B_{R}}|\nabla h|^{2}\phi_{k}^{2}\rho dy\leq -2\int_{\mathbb{R}^{n}\backslash B_{R}}h\phi_{k}\nabla h\cdot\nabla\phi_{k}\rho dy.\]
Letting $k\rightarrow\infty$, we conclude that
\[\int_{\mathbb{R}^{n}\backslash B_{R}}|\nabla h|^{2}\rho dy=0.\]
Hence $h=0$ and the claim is proved. 

By \eqref{innerestimate} and \eqref{limitdecay},   there exists a positive constant $C$ such that
\begin{equation}\label{eigenfunctionestimate}
|f|\leq C(1+|y|)^{-\frac{2p}{p-1}},\quad\text{in}~\mathbb{R}^{n}.
\end{equation}

Let
\[u(x, t)=(-t)^{-\frac{1}{p-1}}w(\frac{x}{\sqrt{-t}}),~~~~\tilde{f}(x, t)=(-t)^{\lambda_{1}-\frac{1}{p-1}}f(\frac{x}{\sqrt{-t}}).\]
Then $\tilde{f}$ satisfies
\[\partial_{t}\tilde{f}=\Delta \tilde{f}+p|u|^{p-1}\tilde{f},\quad\text{in}~\mathbb{R}^{n}\times (-\infty, 0).\]
Combining \eqref{lowerbound} and \eqref{eigenfunctionestimate}, we see there exists a positive constant $C$ depending only on $n, p$, $w$ and $f$ such that
\begin{equation}\label{higherordrestimate}
|\nabla\tilde{f}(x, t)|+|\nabla^{2}\tilde{f}(x, t)|\leq C(-t)^{\lambda_{1}},\quad\text{in}~B_{1}\times (-2, 0).
\end{equation}
This is equivalent to
\begin{equation}\label{higherordrestimate'}
(1+|y|)^{\frac{2p}{p-1}+1}|\nabla f|+(1+|y|)^{\frac{2p}{p-1}+2}|\nabla^{2}f|\leq C,\quad\text{in}~\mathbb{R}^{n}\backslash B_{1}.
\end{equation}
By \eqref{eigenfunctionestimate} and \eqref{higherordrestimate'}, we see \eqref{eigenfunctiondecay} holds.
\end{proof}
Now we can prove the main result in this section.
\begin{theorem}\label{FSEES}
Assume $w$ is a bounded non-constant  solution of \eqref{SC1-2}. If there exists a positive constant $C$ such that
\begin{equation*}
|w|\leq C(1+|y|)^{-\frac{2}{p-1}},\quad\text{in}~\mathbb{R}^{n},
\end{equation*}
then there is a variation $w_{s}$ with $w_{0}=w$ such that
\[\lambda(w_{s})<\lambda(w)\]
for all $s\neq 0$. In particular, $w$ is not entropy stable.
\end{theorem}
\begin{proof}
Let us take a one-parameter variation $w_{s}=w+sf$ for $s\in[-2\epsilon, 2\epsilon]$, where $f$ is the first eigenfunction of the operator
\[L\psi=\Delta\psi-\frac{y}{2}\cdot\nabla\psi-\frac{1}{p-1}\psi+p|w|^{p-1}\psi.\]
Without loss of generality, we may assume $f>0$  in $\mathbb{R}^{n}$. By  the proof of Theorem \ref{thestable}, we know that for any $x(s)$ and $t(s)$ with $x(0)=0$ and $t(0)=-1$,
\begin{equation}\label{FSEES1}
\frac{\partial^{2}}{\partial s^2} (F_{x(s), t(s)}(w_{s}))|_{s=0}<0.
\end{equation}
We will use this to prove that $w$ is also not entropy stable. In order to get this, we define a function $G:\mathbb{R}^{n}\times\mathbb{R}^{-}\times[-2\epsilon, 2\epsilon]\rightarrow\mathbb{R}$ by
\begin{equation}\label{FEEES2}
G(x_{0}, t_{0}, s)=F_{x_{0}, t_{0}}(w_{s}).
\end{equation}
We will show that there exists some $\epsilon_{1}>0$ such that if $s\neq 0$ and $|s|\leq \epsilon_{1}$, then
\begin{equation}\label{FEEES3}
\lambda(w_{s})\equiv \sup\limits_{x_{0}\in\mathbb{R}^{n}, t_{0}\in (-\infty, 0)}G(x_{0}, t_{0}, s)<G(0, -1, 0)=\lambda(w).
\end{equation}
This will give the theorem with $w_{s}$ for any $s\neq 0$ in $(-\epsilon_{1}, \epsilon_{1})$. 

The proof of \eqref{FEEES3} will be divided into the following six steps.\\
\textbf{Step 1:} $G$ has a strict local maximum at $(0, -1, 0)$.

Since $w$ is a solution of the equation \eqref{SC1}, it follows from Proposition \ref{critcalpoints} that $\nabla G$ vanishes at $(0, -1, 0)$.
Given $y_{0}\in\mathbb{R}^{n}, a\in\mathbb{R}$ and $b\in\mathbb{R}\backslash\{0\}$, the second derivative of
$G(sy_{0}, -(1+as), bs)$ at $s=0$ is just
\[\frac{\partial}{\partial_{ss}}G(sy_{0}, -(1+as), bs)|_{s=0}=b^{2}\frac{\partial}{\partial s^2} (F_{x(s), t(s)}(w_{s}))|_{s=0}\]
with
\[x(s)=s\frac{y_{0}}{b},\quad t(s)=-\left(1+\frac{a}{b}s\right).\]
Thus the second derivative of $G(sy, 1+as, 0)$ at $s=0$ is given by
\begin{align}
&\frac{\partial}{\partial_{ss}}G(sy_{0}, -(1+as), bs)|_{s=0}\notag\\
=&-\frac{1}{2b^{2}}\int_{\mathbb{R}^{n}}(\nabla w\cdot y_{0})(\nabla w\cdot y_{0})\rho dy-\frac{a^{2}}{b^{2}}\int_{\mathbb{R}^{n}}\left(\frac{1}{p-1}w+\frac{y}{2}\cdot \nabla w\right)^{2}\rho\notag dy.
\end{align}
Similar to the reasoning used in the proof of Lemma \ref{gapestimate}, it is negative.  Thus we conclude that the Hessian of $G$ at $(0, 1, 0)$ is negative definite. It follows that $G$ has a strict local maximum at $(0, 1, 0)$. In particular, there exists $\epsilon_{2}\in (0, \epsilon)$ such  that
\begin{equation}\label{localestimate}
G(x_{0}, t_{0}, s)<G(0, -1, 0)
\end{equation}
provided that $0<x_{0}^{2}+(\log t_{0})^{2}+s^{2}<\epsilon_{2}^{2}.$

\textbf{Step 2:} $G(x_{0}, t_{0}, 0)$ has a strict global maximum at $(0, -1)$. Moreover, there exists a positive constant $\delta>0$  such  that
\begin{equation}\label{step2.1}
G(x_{0}, t_{0}, 0)<G(0, -1, 0)-\delta
\end{equation}
for all $x_{0}, t_{0}$ with
$\epsilon_{2}^{2}/4<x_{0}^{2}+(\log t_{0})^{2}.$

Indeed, we get from the definition of $G(x_{0}, t_{0}, s)$ that $G(x_{0}, t_{0}, 0)=F_{x_{0}, t_{0}}(w)$ and $G(0, -1, 0)=F_{0, -1}(w)=\lambda(w)$. Thus \eqref{step2.1} is a direct consequence of Lemma \ref{gapestimate}.

\textbf{Step 3:} $|\partial_{s}G|$ is  bounded on compact sets of $\mathbb{R}^{n}\times\mathbb{R}^{-}\times[-2\epsilon, 2\epsilon]$.

By the definition of $G$ and the first variation formula, we have
\begin{equation}
\begin{aligned}
\partial_{s}G(x_{0}, t_{0}, s)=&(-t_{0})^{\frac{p+1}{p-1}}\int_{\mathbb{R}^{n}}(\nabla w\cdot\nabla f)G(y-x_{0}, t_{0})dy\\
&-(-t_{0})^{\frac{p+1}{p-1}}\int_{\mathbb{R}^{n}}|w+sf|^{p-1}(w+sf)fG(y-x_{0}, t_{0})dy\\
&+\frac{1}{(p-1)}(-t_{0})^{\frac{2}{p-1}}\int_{\mathbb{R}^{n}}fwG(y-x_{0}, t_{0})dy.
\end{aligned}
\end{equation}
It is clear that $\partial_{s}G$ is continuous in all three variables $x_{0}, t_{0}$ and $s$. We conclude that $|\partial_{s}G|$ is  bounded on compact subsets.

\textbf{Step 4:} Given $M\geq 1$, there exists  a positive constant $C$ depending only on $n, p, w, M$ such that if $|x_{0}|\gg M$, then
\[G(x_{0}, t_{0}, s)\leq G(x_{0}, t_{0}, 0)+Cs.\]

 By the definition, we have
\begin{align}\label{defineGx0t0s}
G(x_{0}, t_{0}, s)&=\frac{1}{2}(-t_{0})^{\frac{p+1}{p-1}}\int_{\mathbb{R}^{n}}|\nabla w+sf|^{2}G(y-x_{0}, t_{0})dy\notag\\
&-\frac{1}{p+1}(-t_{0})^{\frac{p+1}{p-1}}\int_{\mathbb{R}^{n}}|w+sf|^{p+1}G(y-x_{0}, t_{0})dy\notag\\
&+\frac{1}{2(p-1)}(-t_{0})^{\frac{2}{p-1}}\int_{\mathbb{R}^{n}}(w+sf)^{2}G(y-x_{0}, t_{0})dy\notag\\
=&\frac{1}{2}(-t_{0})^{\frac{p+1}{p-1}}\int_{\mathbb{R}^{n}}|\nabla w|^{2}G(y-x_{0}, t_{0})dy\notag\\
&-\frac{1}{p+1}(-t_{0})^{\frac{p+1}{p-1}}\int_{\mathbb{R}^{n}}|w|^{p+1}G(y-x_{0}, t_{0})dy\notag\\
&+\frac{1}{2(p-1)}(-t_{0})^{\frac{2}{p-1}}\int_{\mathbb{R}^{n}}w^{2}G(y-x_{0}, t_{0})dy\\
&+s(-t_{0})^{\frac{p+1}{p-1}}\int_{\mathbb{R}^{n}}\nabla w\cdot\nabla f G(y-x_{0}, t_{0})dy\notag\\
&-s(-t_{0})^{\frac{p+1}{p-1}}\int_{\mathbb{R}^{n}}|w|^{p-1}wfG(y-x_{0}, t_{0})dy\notag\\
&+\frac{s}{p-1}(-t_{0})^{\frac{2}{p-1}}\int_{\mathbb{R}^{n}}w fG(y-x_{0}, t_{0})dy\notag\\
&+\frac{s^{2}}{2}(-t_{0})^{\frac{p+1}{p-1}}\int_{\mathbb{R}^{n}}|\nabla f|^{2}G(y-x_{0}, t_{0})dy\notag\\
&-\frac{ps^{2}}{2}(-t_{0})^{\frac{p+1}{p-1}}\int_{\mathbb{R}^{n}}|w|^{p-1}f^{2}G(y-x_{0}, t_{0})dy\notag\\
&+\frac{s^{2}}{2(p-1)}(-t_{0})^{\frac{2}{p-1}}\int_{\mathbb{R}^{n}}f^{2}G(y-x_{0}, t_{0})dy\notag\\
&-(-t_{0})^{\frac{p+1}{p-1}}\int_{\mathbb{R}^{n}}F(s, w, f)G(y-x_{0}, t_{0})dy,\notag
\end{align}
where
\[F(s, w, f)=\frac{1}{p+1}[|w+sf|^{p+1}-|w|^{p+1}]-s|w|^{p-1}wf-\frac{ps^{2}}{2}|w|^{p-1}f^{2}.\]
We fix a positive constant $M$ such that $M\gg 1$. For any $x_{0}$ with $|x_{0}|\gg M$, consider the changing of variable $y=x_{0}+\sqrt{-t_{0}}z$. Then
\begin{align}
&(-t_{0})^{\frac{p+1}{p-1}}\int_{\mathbb{R}^{n}}\nabla w\cdot\nabla fG(y-x_{0}, t_{0})dy\notag\\
=&(-t_{0})^{\frac{p+1}{p-1}}\int_{\mathbb{R}^{n}}\nabla w\cdot\nabla f(x_{0}+\sqrt{-t_{0}}z)\rho(z)dz\\
\leq & I_{1}+I_{2}+I_{3}\notag
\end{align}
with
\begin{equation}
\begin{aligned}
I_{1}=&(-t_{0})^{\frac{p+1}{p-1}}\int_{\left\{|z|\leq \frac{|x_{0}|}{2\sqrt{-t_{0}}}\right\}}|\nabla w||\nabla f|(x_{0}+\sqrt{-t_{0}}z)\rho(z)dz,\\
I_{2}=&(-t_{0})^{\frac{p+1}{p-1}}\int_{\left\{\frac{|x_{0}|}{2\sqrt{-t_{0}}}\leq|z|\leq \frac{3|x_{0}|}{2\sqrt{-t_{0}}}\right\}}|\nabla w||\nabla f|(x_{0}+\sqrt{-t_{0}}z)\rho(z)dz,\\
I_{3}=&(-t_{0})^{\frac{p+1}{p-1}}\int_{\left\{|z|\geq \frac{3|x_{0}|}{2\sqrt{-t_{0}}}\right\}}|\nabla w||\nabla f|(x_{0}+\sqrt{-t_{0}}z)\rho(z)dz.
\end{aligned}
\end{equation}

First, we estimate the term $I_{3}$. If $|z|\geq (3|x_{0}|)/(2\sqrt{-t_{0}})$, then
\begin{equation}\label{domainI3}
|x_{0}+\sqrt{-t_{0}z}|\geq \sqrt{-t_{0}}|z|-|x_{0}|\geq \sqrt{-t_{0}}|z|-\frac{2}{3}\sqrt{-t_{0}}|z|\geq\frac{1}{3}\sqrt{-t_{0}}|z|.
\end{equation}
Therefore,
\[|x_{0}+\sqrt{-t_{0}z}|\geq\frac{|x_{0}|}{2}\geq \frac{M}{2}.\]
By \eqref{lowerbound}, Lemma \ref{gradientestimate} and Lemma \ref{lemeigenfunctiondecay}, there exists a positive constant $C$ such that
\[\begin{aligned}
I_{3}=&(-t_{0})^{\frac{p+1}{p-1}}\int_{\left\{|z|\geq \frac{3|x_{0}|}{2\sqrt{-t_{0}}}\right\}}|\nabla w||\nabla f|(x_{0}+\sqrt{-t_{0}}z)\rho(z)dz\\
\leq& C(-t_{0})^{\frac{p+1}{p-1}}\int_{\left\{|z|\geq \frac{3|x_{0}|}{2\sqrt{-t_{0}}}\right\}}|x_{0}+\sqrt{-t_{0}}z|^{-\frac{4}{p-1}-2}\rho(z)dz.
\end{aligned}\]
Taking \eqref{domainI3} into account, we see there exists a positive constant $C$ depending only on $n, p$ and $w$ such that
\begin{equation}\label{estimatei3}
\begin{aligned}
I_{3}\leq& C\int_{\left\{|z|\geq \frac{3|x_{0}|}{2\sqrt{-t_{0}}}\right\}}|z|^{-\frac{4}{p-1}-2}\rho(z)dz\\
\leq&C\int_{\mathbb{R}^{n}}|z|^{-\frac{4}{p-1}-2}\rho(z)dz\leq C.
\end{aligned}
\end{equation}
Here we have applied the fact that $p>(n+2)/(n-2)$.

For the term $I_{1}$, we divide the estimate  into two cases. If $ |x_{0}|/\sqrt{-t_{0}}\geq M$, then
\begin{equation}\label{estimatei1}
I_{1}\leq C(-t_{0})^{\frac{p+1}{p-1}}|x_{0}|^{-\frac{4}{p-1}-2}\leq C.
\end{equation}
On the other hand, if $|x_{0}|/\sqrt{-t_{0}}\leq M$, then
\begin{equation}\label{estimatei1'}
I_{1}\leq C(-t_{0})^{\frac{p+1}{p-1}}|x_{0}|^{-\frac{4}{p-1}-2}\left(\frac{|x_{0}|}{\sqrt{-t_{0}}}\right)^{n}\leq C.
\end{equation}

 Finally, we  estimate $I_{2}$ as follows,
\begin{align}
I_{2}=&(-t_{0})^{\frac{p+1}{p-1}}\int_{\left\{\frac{|x_{0}|}{2\sqrt{-t_{0}}}\leq|z|\leq \frac{3|x_{0}|}{2\sqrt{-t_{0}}}\right\}}|\nabla w||\nabla f|(x_{0}+\sqrt{-t_{0}}z)\rho(z)dz\notag\\
\leq &C(-t_{0})^{\frac{p+1}{p-1}}e^{\frac{|x_{0}|^{2}}{16t_{0}}}\int_{\left\{\frac{|x_{0}|}{2\sqrt{-t_{0}}}\leq|z|\leq \frac{3|x_{0}|}{2\sqrt{-t_{0}}}\right\}}|\nabla w||\nabla f|(x_{0}+\sqrt{-t_{0}}z)dz\notag\\
\leq & C(-t_{0})^{\frac{p+1}{p-1}}e^{\frac{|x_{0}|^{2}}{16t_{0}}}(-t_{0})^{-\frac{n}{2}}\int_{\left\{|y|\leq 4|x_{0}|\right\}}|\nabla w_{0}||\nabla f|dy\notag\\
\leq & C(-t_{0})^{\frac{p+1}{p-1}}e^{\frac{|x_{0}|^{2}}{16t_{0}
}}(-t_{0})^{-\frac{n}{2}}|x_{0}|^{n-2-\frac{4}{p-1}}\notag\\
\leq &C\left(\frac{|x_{0}|}{\sqrt{-t_{0}}}\right)^{n-2-\frac{4}{p-1}}e^{\frac{|x_{0}|^{2}}{16t_{0}}}\leq C.\notag
\end{align}
In conclusion, there exist $M\geq 1$  and $C$ depending only on $n, p, w_{0}$ such that if $|x_{0}|\gg M$ or $|t_{0}|\gg M$, then
\[(-t_{0})^{\frac{p+1}{p-1}}\int_{\mathbb{R}^{n}}|\nabla w||\nabla f|G(y-x_{0}, t_{0})dy\leq C.\]
Other terms can be estimated similarly. Therefore, we conclude that there exists $M\geq 1$  and $C$ depending only on $n, p, w$ such that if $|x_{0}|\gg M$ or $|t_{0}|\gg M$, then
\[F_{x_{0}, t_{0}}(w+sf)\leq F_{x_{0}, t_{0}}(w)+Cs.\]

\textbf{Step 5:} Given $M\geq 1$, there exists  a positive constant $C$ depending only on $n, p, w, M$ such that if $|x_{0}|\leq M$ and $|t_{0}|\gg M$, then
\[G(x_{0}, t_{0}, s)\leq G(x_{0}, t_{0}, 0)+Cs.\]

Similar to the proof of Step 4, we only show that there exists a positive constant $C$ depending only on $n, p, w, M$ such that
\begin{equation}\label{step5.1e}
   (-t_{0})^{\frac{p+1}{p-1}}\int_{\mathbb{R}^{n}}|\nabla w||\nabla f|G(y-x_{0}, t_{0})dy\leq C.
\end{equation}
There are three cases:

\textbf{Case 1:} $n-5-4/(p-1)>-1$.

Then \eqref{lowerbound} and Lemma \ref{lemeigenfunctiondecay} imply
\[
\begin{aligned}
&(-t_{0})^{\frac{p+1}{p-1}}\int_{\mathbb{R}^{n}}|\nabla w||\nabla f|G(y-x_{0}, t_{0})dy\\
\leq& C(-t_{0})^{\frac{p+1}{p-1}-\frac{n}{2}}\int_{\mathbb{R}^{n}}(1+|y|)^{-4-\frac{4}{p-1}}dy\\
\leq&C(-t_{0})^{\frac{p+1}{p-1}-\frac{n}{2}}.
\end{aligned}
\]
Since we have assumed that $p>(n+2)/(n-2)$, it is clear that \eqref{step5.1e} holds.

\textbf{Case 2:} $n-5-4/(p-1)<-1$.

In this case, we have
\[
\begin{aligned}
&(-t_{0})^{\frac{p+1}{p-1}}\int_{\mathbb{R}^{n}}|\nabla w||\nabla f|G(y-x_{0}, t_{0})dy\\
\leq & C(-t_{0})^{\frac{p+1}{p-1}-\frac{n}{2}}\int_{\{|y|\leq(-t_{0})^{\alpha(p)}\}}(1+|y|)^{-4-\frac{4}{p-1}}dy\\
&+C(-t_{0})^{\frac{p+1}{p-1}}\int_{\{|y|\geq(-t_{0})^{\alpha(p)}\}}(-4\pi t_{0})^{-\frac{n}{2}}e^{\frac{|y-x_{0}|^{2}}{4 t_{0}}}dy\leq C,
\end{aligned}
\]
where
\[\alpha(p):=\frac{n-2-\frac{4}{p-1}}{2(n-4-\frac{4}{p-1})}.\]

\textbf{Case 3:} $n-5-4/(p-1)=-1$.

This case   can be proved in the same way as in  Case2.  

In conclusion, we have shown that  both \eqref{step5.1e} and Step 5 hold.

\textbf{Step 6:} There exists $t_{1}$ such that if $t_{1}<t_{0}<0$, then $G(x_{0}, t_{0}, s)$ is strictly less than $G(0, -1, 0)$.

Indeed, since $|\nabla w|,|\nabla f|, f$ are bounded in $\mathbb{R}^{n}$ and $w$ satisfies \eqref{lowerbound},  there exists two constant $C_{1}, C_{2}$ independent of $x_{0}$ such that
\[G(x_{0}, t_{0}, s)=F_{x_{0}, t_{0}}(w_{s})\leq C_{1}(-t_{0})^{\frac{p+1}{p-1}}+C_{2}(-t_{0})^{\frac{2}{p-1}}.\]
By Corollary \ref{CCSP}, we have
\[\lambda(w)=E(w)=\left(\frac{1}{2}-\frac{1}{p+1}\right)\int_{\mathbb{R}^{n}}|w|^{p+1}\rho dy>0.\]
Therefore, if $-t_{0}$ is small enough, then $G(x_{0}, t_{0}, s)<\lambda(w)=G(0, -1, 0)$.

\textbf{Step 7:}  There exists some $\epsilon_{1}>0$ such that if $s\neq 0$ and $|s|\leq \epsilon_{1}$, then
\begin{equation}\label{Conclusion}
\lambda(w_{s})\equiv \sup\limits_{x_{0}, t_{0}}G(x_{0}, t_{0}, s)<G(0, -1, 0)=\lambda(w).
\end{equation}
Let us consider  three separate regions depending on the size of $|x_{0}|^{2}+(\log t_{0})^{2}$. First, it follows from Step 4, Step 5 and Step 6 that there is some $R>0$ such that \eqref{Conclusion} holds for every $s$ whenever
\[|x_{0}|^{2}+(\log t_{0})^{2}>R^{2}.\]
Secondly, as long as $s$ is small, Step 1 implies that \eqref{Conclusion} holds whenever $|x_{0}|^{2}+(\log t_{0})^{2}$ is sufficiently small.

Finally, in the intermediate region where $|x_{0}|^{2}+(\log t_{0})^{2}$ is bounded  from above and bounded uniformly away from zero, Step 2 implies that $G$ is strictly  less than $\lambda(w)$ at $s=0$. Moreover, it follows from Step 3 that the $s$ derivative of $G$ is uniformly bounded. Hence there exists some $\epsilon_{3}>0$ such that $G(x_{0}, t_{0}, s)$ is strictly less than $\lambda(w)$ whenever $(x_{0}, s_{0})$ is in the intermediate region as long as $|s|\leq \epsilon_{3}$.
\end{proof}

\section{The energy of self similar solutions}
Results obtained in the previous sections hold for all bounded solutions of \eqref{SC1-2}. In this section, we will turn our attention to functions in $\mathcal{B}_{n, m}$. We will prove that if $w\in\mathcal{B}_{n, m}$ and there exists a positive constant $C$  such that \eqref{lowerbound} holds, then $E(w)$ is strictly larger than the weighted energy of the positive constant solution of \eqref{SC1}.
\begin{lemma}[Blow up criteria]\label{blowupcriteria}
Let $w$ be a positive solution of
\begin{equation}\label{selfsimilarequation}
\partial_{\tau}w=\Delta w-\frac{y}{2}\cdot\nabla w-\frac{1}{p-1}w+w^{p},\quad\text{in}~\mathbb{R}^{n}\times(0, +\infty).
\end{equation}
Assume for any $M>0$, there exists a positive constant $C(M)$ such that
\[w\leq C(M)\quad\text{in}~\mathbb{R}^{n}\times[0, M].\]
Then for any $\tau\in [0, +\infty)$,
\begin{equation}\label{integralaveragecondition}
(4\pi)^{-\frac{n}{2}}\int_{\mathbb{R}^{n}}w(y, \tau)e^{-\frac{|y|^{2}}{4}}dy\leq \kappa.
\end{equation}
\end{lemma}
\begin{proof}
This is exactly \cite[Proposition 3.8]{Merle-Zaag}.
\end{proof}
\begin{lemma}
If $w$ is a bounded positive solution of the equation \eqref{SC1} such that \eqref{lowerbound} holds. Assume $f$ is the first eigenfunction of the operator
\begin{equation}\label{definelinearoperatorL}
L\psi=\Delta\psi-\frac{y}{2}\cdot\nabla\psi-\frac{1}{p-1}\psi+pw^{p-1}\psi,
\end{equation}
Without loss of generality, we may assume $f$ is positive in $\mathbb{R}^{n}$. Consider the Cauchy problem
\begin{equation}\label{IVP}
\left\{\begin{array}{lll}
\partial_{\tau}\tilde{w}=\Delta \tilde{w}-\frac{y}{2}\cdot\nabla\tilde{w}-\frac{1}{p-1}\tilde{w}+\tilde{w}^{p},\\
\tilde{w}(\cdot, 0)=w+sf,
\end{array}
\right.
\end{equation}
where $s$ is a small positive constant, then $\tilde{w}$ blows up in finite time.
\end{lemma}
\begin{proof}
Assume it is false, then $\tilde{w}$ exists in $\mathbb{R}^{n}\times (0, +\infty)$. Consider the function
\[\bar{w}(y, \tau)=w(y)+se^{-\lambda_{1}\tau}f(y),\]
where $\lambda_{1}$ is the first eigenvalue of the operator defined in \eqref{definelinearoperatorL}.  
Since $w$ satisfies \eqref{SC1},  $\bar{w}$ satisfies
\begin{equation*}
\left\{\begin{array}{lll}
\partial_{\tau}\bar{w}=\Delta \bar{w}-\frac{y}{2}\cdot\nabla\bar{w}-\frac{1}{p-1}\bar{w}+w^{p}+spw^{p-1}e^{-\lambda_{1}\tau}f,\\
\tilde{w}(\cdot, 0)=w+sf.
\end{array}
\right.
\end{equation*}
Notice that for any $a\geq 0, b\geq 0$, we have $(a+b)^{p}\geq a^{p}+pa^{p-1}b$. Thus $\bar{w}$ satisfies
\begin{equation*}
\left\{\begin{array}{lll}
\partial_{\tau}\bar{w}\leq\Delta \bar{w}-\frac{y}{2}\cdot\nabla\bar{w}-\frac{1}{p-1}\bar{w}+\bar{w}^{p},\\
\tilde{w}(\cdot, 0)=w+sf.
\end{array}
\right.
\end{equation*}
By the maximum principle, we have
\[\tilde{w}\geq \bar{w},\quad\text{in}~\mathbb{R}^{n}\times(0,+\infty).\]
In particular, for any $\tau\geq 0$,
\begin{equation}\label{integralaveragecompare}
\begin{aligned}
&(4\pi)^{-\frac{n}{2}}\int_{\mathbb{R}^{n}}\tilde{w}(y, \tau)e^{-\frac{|y|^{2}}{4}}dy\\
\geq & (4\pi)^{-\frac{n}{2}}\int_{\mathbb{R}^{n}}\bar{w}(y, \tau)e^{-\frac{|y|^{2}}{4}}dy\\
=&(4\pi)^{-\frac{n}{2}}\int_{\mathbb{R}^{n}}w(y)e^{-\frac{|y|^{2}}{4}}dy+(4\pi)^{-\frac{n}{2}}e^{-\lambda_{1}\tau}\int_{\mathbb{R}^{n}}f(y)e^{-\frac{|y|^{2}}{4}}dy.
\end{aligned}
\end{equation}
Since $\lambda_{1}<-1$, we get from \eqref{integralaveragecompare} that
\[(4\pi)^{-\frac{n}{2}}\int_{\mathbb{R}^{n}}\tilde{w}(y, \tau)e^{-\frac{|y|^{2}}{4}}dy>\kappa\]
provided that $\tau$ is large enough. By Lemma \ref{blowupcriteria}, this is a contradiction.
\end{proof}
\begin{lemma}\label{lemblowupset}
Let $\tau_{1}$ be the first blow up time of $\tilde{w}$, then there exist $R>0$ and $c>0$ such that
\[\tilde{w}\leq c,\quad\text{in}~(\mathbb{R}^{n}\backslash B_{R}(0))\times (0, \tau_{1}).\]
In particular, if we use  $\Sigma$  to denote the blow up set of $\tilde{w}$, then $\Sigma$ is a compact subset of $\mathbb{R}^{n}$.
\end{lemma}
\begin{proof}
 Since $w$ satisfies \eqref{lowerbound}, it follows  from  Lemma \ref{gradientestimate} and Lemma \ref{lemeigenfunctiondecay} that there exists a positive constant $C$ such that
 \[(1+|y|)^{\frac{2}{p-1}}w+(1+|y|)^{1+\frac{2}{p-1}}|\nabla w|\leq C,\quad\text{in}~\mathbb{R}^{n}\]
 and
 \[(1+|y|)^{\frac{2p}{p-1}}|f|+(1+|y|)^{1+\frac{2p}{p-1}}|\nabla f|\leq C,\quad\text{in}~\mathbb{R}^{n}.\]
 Thus given $\delta_{0}>0, \epsilon>0$, there exists $R$ such that if $|y_{0}|>R$ and $(z_{0}, \tau_{0})\in B_{\delta_{0}}(y_{0})\times(\tau_{1}-\delta_{0}^{2}, \tau_{1}+\delta_{0}^{2})$, then
 \[\mathcal{E}_{(z_{0}, \tau_{0})}(w+sf)<\epsilon,\]
 where
 \[
 \begin{aligned}
 \mathcal{E}_{(z_{0}, \tau_{0})}(w+sf)=&\frac{1}{2}(\tau_{0})^{\frac{2}{p-1}-\frac{n}{2}+1}\int_{\mathbb{R}^{n}}|\nabla(w+sf)|^{2}e^{-\frac{|y-z_{0}|^{2}}{4\tau_{0}}}dy\\
 &-\frac{1}{p+1}(\tau_{0})^{\frac{2}{p-1}-\frac{n}{2}+1}\int_{\mathbb{R}^{n}}|w+sf|^{p+1}e^{-\frac{|y-z_{0}|^{2}}{4\tau_{0}}}dy\\
 &+\frac{1}{2(p-1)}(\tau_{0})^{\frac{2}{p-1}-\frac{n}{2}}\int_{\mathbb{R}^{n}}(w+sf)^{2}e^{-\frac{|y-z_{0}|^{2}}{4\tau_{0}}}dy\\
 \end{aligned}\]
Following the arguments in the proof of \cite[Theorem 3.1]{Dushizhong2019} (starting from formula (3.12)), we  obtain 
 \[w\leq C_{*}\delta_{0}^{-\frac{2}{p-1}}, \quad\text{in}~B_{\delta_{0}/2}(y_{0})\times(\tau_{1}-\delta_{0}^{2}, \tau_{1}),\]
 where $C_{*}$ is a universal positive constant. Then Lemma \ref{lemblowupset} follows from this and the definition of $\tau_{1}$.
\end{proof}
\begin{proposition}\label{decreaseentropy}
If $w\in \mathcal{B}_{n, m}$ and \eqref{lowerbound} holds, then
\[E(w)\geq E(\kappa).\]
\end{proposition}
\begin{proof}
The proof of Proposition \ref{decreaseentropy} will be divided into several steps.

 \textbf{Step 1}: Let $\tilde{w}$ be the solution of \eqref{IVP} and let $\tau_{1}$ be the first blow up time. Then 
\begin{equation}\label{step1estimate}
\partial_{\tau}\tilde{w}>0,\quad\text{in}~\mathbb{R}^{n}\times(0, \tau_{1}).
\end{equation}

Since $\tilde{w}$ satisfies \eqref{IVP}, 
\[
\begin{aligned}
\partial_{\tau}\tilde{w}(\cdot, 0)=&\Delta (w+sf)-\frac{y}{2}\cdot\nabla(w+sf)-\frac{1}{p-1}(w+sf)+(w+sf)^{p}\\
=&(w+sf)^{p}-w^{p}-spw^{p-1}f-s\lambda_{1}f>0.
\end{aligned}
\]
We also have the equation for $\partial_{\tau}\tilde{w}$,
\begin{equation*}
\left\{\begin{array}{lll}
\partial_{\tau}\partial_{\tau}\tilde{w}=\Delta\partial_{\tau}\tilde{w}-\frac{y}{2}\cdot\nabla\partial_{\tau}\tilde{w}-\frac{1}{p-1}\partial_{\tau}\tilde{w}+p\bar{w}^{p-1}\partial_{\tau}\tilde{w},\quad\text{in}~~ \mathbb{R}^{n}\times(0, \tau_{1}),\\
\partial_{\tau}\tilde{w}(\cdot, 0)>0.
\end{array}
\right.
\end{equation*}
By the  maximum principle, we  get  \eqref{step1estimate}.

\textbf{Step 2}: Let $\tilde{w}$ be the solution of \eqref{IVP} and let $\tau_{1}$ be the first blow up time. Then there exists a positive constant $C$ such that
\begin{equation}\label{step3.1}
\tilde{w}\leq C(\tau_{1}-\tau)^{-\frac{1}{p-1}},\quad\text{in}~\mathbb{R}^{n}\times(0, \tau_{1}).
\end{equation}

By Lemma \ref{lemblowupset}, the singular set $\Sigma$ is a compact subset of $\mathbb{R}^{n}$. Fix $R>0$ such that $\Sigma\subset B_{R/2}(0)$.
By Step 1, there exist $0<\tau_{0}<\tau_{1}$ and a positive constant $c_{0}$ such that
\begin{equation}\label{step3.2}
\partial_{\tau}\tilde{w}\geq c_{0},\quad\text{on}~B_{R}(0)\times \{\tau=\tau_{0}\}\cup\partial B_{R}(0)\times [\tau_{0}, \tau_{1}].
\end{equation}
Moreover, by the choice of $R$, there exists a positive constant $C_{1}$ such that
\begin{equation}\label{step3.3}
\tilde{w}^{p}\leq C_{1},\quad\text{on}~B_{R}(0)\times \{\tau=\tau_{0}\}\cup\partial B_{R}(0)\times [\tau_{0}, \tau_{1}].
\end{equation}
Combining \eqref{step3.2} and \eqref{step3.3}, we know that there exists a small  positive constant $\epsilon$ such that
\begin{equation}\label{step3.4}
\partial_{\tau}\tilde{w}\geq\epsilon\tilde{w}^{p},\quad\text{on}~B_{R}(0)\times \{\tau=\tau_{0}\}\cup\partial B_{R}(0)\times [\tau_{0}, \tau_{1}].
\end{equation}
Because $\tilde{w}$ satisfies \eqref{IVP}, we have
\begin{equation}\label{step3.5}
\partial_{\tau}\partial_{\tau}\tilde{w}=\Delta \partial_{\tau}\tilde{w}-\frac{y}{2}\cdot\nabla\partial_{\tau}\tilde{w}-\frac{1}{p-1}\partial_{\tau}\tilde{w}+p\tilde{w}^{p-1}\partial_{\tau}\tilde{w}
\end{equation}
and
\[\partial_{\tau}\tilde{w}^{p}=\Delta \tilde{w}^{p}-\frac{y}{2}\cdot\nabla \tilde{w}^{p}+p\tilde{w}^{p-1}\tilde{w}^{p}-[p(p-1)\tilde{w}^{p-2}|\nabla \tilde{w}|^{2}+\tilde{w}^{p}].\]
Let $\omega(y, \tau)=\partial_{\tau}\tilde{w}(y, \tau)-\epsilon\tilde{w}^{p}(y, \tau)$, where $\epsilon$ is a small positive constant satisfying  \eqref{step3.4}. Then $\omega$ satisfies
\begin{equation}\label{step3.6}
\left\{\begin{array}{lll}
\partial_{\tau}\omega-\Delta \omega+\frac{y}{2}\cdot \nabla\omega+\frac{1}{p-1}\omega-p\tilde{w}^{p-1}\omega\geq 0,\quad\text{in}~B_{R}\times(\tau_{0}, \tau_{1}),\\
\omega>0 \qquad\qquad\qquad\quad\text{on}~B_{R}(0)\times \{\tau=\tau_{0}\}\cup\partial B_{R}(0)\times [\tau_{0}, \tau_{1}].
\end{array}
\right.
\end{equation}
By the maximum principle,
\begin{equation}\label{step3.7}
\partial_{\tau}w\geq \epsilon w^{p},\quad\text{in}~B_{R}\times(\tau_{0}, \tau_{1}).
\end{equation}
For any $\tau_{0}<\tau<\tau_{1}-\delta$, integrating \eqref{step3.7} from $\tau$ to $\tau_{1}-\delta$ gives
\[w(\tau_{1}-\delta)^{1-p}-w(\tau)^{1-p}\leq -(p-1)\epsilon(\tau_{1}-\delta-\tau).\]
Hence
\[w\leq M(\tau_{1}-\delta-\tau)^{-\frac{1}{p-1}},\quad\text{in}~B_{R}\times(\tau, \tau_{1}-\delta),\]
where  $M=((p-1)\epsilon)^{-1/(p-1)}$.
Taking $\delta\to 0$, we conclude that
\[w\leq M(\tau_{1}-\tau)^{-\frac{1}{p-1}},\quad\text{in}~B_{R}\times(\tau_{0}, \tau_{1}),\]
 Combining this with Lemma \ref{lemblowupset}, we  get \eqref{step3.1}.

\textbf{Step 3}: There exists a smooth solution $\tilde{w}_{*}$ of \eqref{SC1} such that $E(\tilde{w}_{*})<E(w)$.

We set
\[\hat{w}(z, \varsigma)=(1-e^{\tau-\tau_{0}})^{\frac{1}{p-1}}w(y+e^{\frac{\tau}{2}}y_{0}, \tau)\]
where
\[z=(1-e^{\tau-\tau_{0}})^{-\frac{1}{2}}(y+e^{\frac{\tau}{2}}y_{0}),~~\varsigma=-\ln(e^{-\tau}-e^{-\tau_{0}}).\]
Then $\hat{w}$ satisfies the equation
\[
\left\{\begin{array}{lll}
\partial_{\varsigma}\hat{w}=\Delta\hat{w}-\frac{y}{2}\cdot \nabla\hat{w}-\frac{1}{p-1}\hat{w}+\hat{w}^{p},\\
\hat{w}(z, -\ln(1-e^{-\tau_{0}}))=(1-e^{\tau-\tau_{0}})^{\frac{1}{p-1}}[w(\tilde{y}, -1)+sf(\tilde{y}, -1)],
\end{array}
\right.\]
where $\tilde{y}=y+y_{0}$.

Let $\{\varsigma_{k}\}$ be a sequence such that $\lim_{k\rightarrow\infty} \varsigma_{k}=\infty$ and let $w_{k}(z, \varsigma)=w(z, \varsigma+\varsigma_{k})$.
Similar to the proof of (see \cite[Proposiiton 4]{Giga-Kohn1985}), $\lim_{k\rightarrow\infty}\tilde{w}_{k}=\tilde{w}_{*}$
for some solution of \eqref{SC1} uniformly on compact subsets. By applying the monotonicity formula (see \cite[Proposiiton 3]{Giga-Kohn1985}), we have
\begin{equation}\label{compareenergy}
E(\hat{w}(z, -\ln(1-e^{-\tau_{0}})))>E(\tilde{w}_{*}).
\end{equation}
We compute
\begin{equation*}\label{limitationenergy}
\begin{aligned}
&E(\hat{w}(z, -\ln(1-e^{-\tau_{0}})))\\
=&\frac{1}{2}\int_{\mathbb{R}^{n}}|\nabla \hat{w}(\cdot, -\ln(1-e^{-\tau_{0}}))|^{2}\rho dz-\frac{1}{p+1} \int_{\mathbb{R}^{n}} \hat{w}^{p+1}(\cdot,-\ln(1-e^{-\tau_{0}}))\rho dz\\
&+\frac{1}{2(p-1)}\int_{\mathbb{R}^{n}} \hat{w}^{2}(\cdot, -\ln(1-e^{-s_{0}}))\rho dz\\
=&\frac{1}{2}(1-e^{-\tau_{0}})^{\frac{p+1}{p-1}}\int_{\mathbb{R}^{n}}|\nabla w-s\nabla f|^{2}G(y, e^{-\tau_{0}}-1)d\tilde{y}\\
&-\frac{1}{p+1}(1-e^{-\tau_{0}})^{\frac{p+1}{p-1}}\int_{\mathbb{R}^{n}}(w+sf)^{p+1}]G(y, e^{-\tau_{0}}-1)d\tilde{y}\\
&+\frac{1}{2(p-1)}(1-e^{-\tau_{0}})^{\frac{2}{p-1}}\int_{\mathbb{R}^{n}} (w+sf)^{2}G(y, e^{-\tau_{0}}-1)d\tilde{y}\\
=&F_{y_{0}, e^{-\tau_{0}}-1}(w+sf).
\end{aligned}
\end{equation*}
Since $w$ satisfies \eqref{lowerbound}, we get from Theorem \ref{thestable} that $w$ is not $F-$stable. Then Theorem \ref{FSEES} implies that $w$ is not entropy stable.
Moreover, we know from the proof of Theorem \ref{FSEES} that
 \[F_{y_{0}, e^{-\tau_{0}}-1}(w+sf)<E(w)\]
 provided that $s$ is small enough.  Combining this with \eqref{compareenergy}, we conclude the proof of this step.

\textbf{Step 4}: The function $\tilde{w}_{*}$ in Step 3 is the positive constant solution of \eqref{SC1}.

By the definition of $\hat{w}$, it is easy to see that
\[(1-e^{\tau-\tau_{0}})^{-\frac{1}{p-1}}\hat{w}(z, \varsigma)=w(y+e^{\frac{\tau}{2}}y_{0}, \tau).\]
Taking derivative with respect to $\tau$ on both sides, we have
\[\begin{aligned}
&\frac{y}{2}_{0}e^{\frac{\tau}{2}}\nabla w(y+e^{\frac{\tau}{2}}y_{0}, \tau)+\partial_{\tau}w(y+e^{\frac{\tau}{2}}y_{0}, \tau)\\
=&-\frac{1}{p+1}(1-e^{\tau-\tau_{0}})^{-\frac{1}{p-1}-1}e^{\tau-\tau_{0}}\hat{w}+\frac{1}{2}(1-e^{\tau-\tau_{0}})^{-\frac{1}{2}-1-\frac{1}{p-1}}e^{\tau-\tau_{0}}(ye^{\frac{\tau}{2}}+y_{0})\nabla\hat{w}\\
&+\frac{y}{2}_{0}e^{\frac{\tau}{2}}(1-e^{\tau-\tau_{0}})^{-\frac{1}{2}}(1-e^{\tau-\tau_{0}})^{-\frac{1}{p-1}}\nabla\hat{w}+\partial_{\varsigma}\hat{w}(1-e^{\tau-\tau_{0}})^{-\frac{1}{p-1}}\left(\frac{e^{-\tau}}{e^{-\tau}-e^{-\tau_{0}}}\right).
\end{aligned}\]
Taking derivative with respect to $y$ on both sides, we have
\[\nabla w(y+e^{\frac{\tau}{2}}y_{0},\tau)=(1-e^{\tau-\tau_{0}})^{-\frac{1}{p-1}}(1-e^{\tau-\tau_{0}})^{-\frac{1}{2}}\nabla \hat{w}.\]
Since $\tilde{w}$ satisfies \eqref{step1estimate}, we have
\[\frac{1}{p-1}\hat{w}+\frac{1}{2}z\cdot \nabla\hat{w}+\partial_{\varsigma}\hat{w}\geq \epsilon \hat{w}^{p}>0,\quad\text{in}~\mathbb{R}^{n}\times(-\ln(1-e^{-\tau_{0}}), \infty).\]
By Step 2, we know that there exists a positive constant $C_{3}$ such that
\[\hat{w}\leq c_{3},\quad\text{in}~\mathbb{R}^{n}\times(-\ln(1-e^{-\tau_{0}}),\infty).\]
Let $\{w_{k}\}$ be the functions defined in Step 3. Then $\partial_{\varsigma}\hat{w}$ converges to $0$ uniformly on compact subsets.
 By the above analysis, we conclude that $\tilde{w}_{*}$ is a nonnegative solution of \eqref{SC1} satisfying
\[\tilde{w}_{*}\leq C_{3},\quad\text{in}~\mathbb{R}^{n}\]
and
\[\frac{1}{p-1}\tilde{w}_{*}+\frac{1}{2}z\cdot\nabla\tilde{w}_{*}>0,\quad\text{in}~\mathbb{R}^{n}.\]
By Proposition \ref{ChCS}, $\tilde{w}_{*}$ is the constant solution of \eqref{SC1}.

 Combining Step 2, Step 3 and Step 4, we finish the proof of Proposition \ref{decreaseentropy}.
\end{proof}

\section{Constant solutions have the lowest energy}
In this section, we will combine the Federer type dimension reduction arguments with the results obtained in the  previous sections  to prove that the positive constant solution of \eqref{SC1} has the lowest energy among functions in $\mathcal{B}_{n, m}$.
\begin{lemma}\label{ESL}
There exists a positive constant $C$ depends only on $n, p, m$ such that if $w\in\mathcal{B}_{n, m}$ , then
\[C^{-1}\leq E(w)\leq C.\]
\end{lemma}
\begin{proof}
It follows from \eqref{Inte1} and \eqref{weightedenergy} that
\begin{equation}\label{Energy'}
E(w)=\left(\frac{1}{2}-\frac{1}{p+1}\right)\int_{\mathbb{R}^{n}}w^{p+1}\rho dy.
\end{equation}
For any $w\in\mathcal{B}_{n, m}$, it is clear that there exists a point $y_{0}\in\mathbb{R}^{n}$ such that $w(y_{0})=\kappa$. Since $w\in\mathcal{B}_{n, m}$, it follows from Lemma \ref{lempre3.1} that there exists a constant $r_{1}>0$ depending only on $n, p$ and $m$ such that
\[w\geq \left(\frac{1}{2(p-1)}\right)^{\frac{1}{p-1}},\quad\text{in}~B_{r_{1}}(y_{0}).\]
Therefore, there exists a positive constant $C$ depending only on $n, p$ and $m$ such that
\begin{equation}\label{ESL1}
\int_{\mathbb{R}^{n}}w^{p+1}\rho dy\geq \int_{B_{r_{1}}(y_{0})}w^{p+1}\rho dy\geq C.
\end{equation}
By \eqref{Energy'} and \eqref{ESL1}, we have
\begin{equation}\label{ESL2}
E(w)\leq \left(\frac{1}{2}-\frac{1}{p+1}\right)C.
\end{equation}
Next, since $w\in\mathcal{B}_{n, m}$, then
\begin{equation}\label{ESL3}
E(w)=\left(\frac{1}{2}-\frac{1}{p+1}\right)\int_{\mathbb{R}^{n}}w^{p+1}\rho dy\leq \left(\frac{1}{2}-\frac{1}{p+1}\right) m^{p+1}.
\end{equation}
Combining \eqref{ESL2} and \eqref{ESL3}, we  finish the proof.
\end{proof}
\begin{lemma}\label{minimizer}
There exists $w_{0}\in\mathcal{B}_{n, m}$ such that
\[E(w_{0})=\inf_{w\in \mathcal{B}_{n, m}}E(w).\]
\end{lemma}
\begin{proof}
Let $\{w_{i}\}$ be a sequence such that
 \[\lim_{i\rightarrow\infty}E(w_{i})=\inf_{w\in \mathcal{B}_{n, m}}E(w).\]
 It follows from Lemma \ref{lempre3.1} that $|w_{i}|+|\nabla w_{i}|+|\nabla^{2}w_{i}|+|\nabla^{3}w_{i}|$ are uniformly bounded. By the  Arzel\'{a}-Ascoli theorem, we know that there exists a function $w_{0}$ such that $\lim_{i\rightarrow\infty}w_{i}=w_{0}$.  Since the functions $w_{i}$ converge to $w_{0}$ uniformly on compact subsets of $\mathbb{R}^{n}$, then $w_{0}\in\mathcal{B}_{n, m}$. The convergence also implies that
 \[E(w_0)=\lim_{i\rightarrow\infty}E(w_{i}),\]
 so $w_0$ attains the minima in $\mathcal{B}_{n, m}$.
\end{proof}
\begin{proof}[Proof of Theorem \ref{maintheorem1}]
By Lemma \ref{minimizer}, there exists $w_{0}\in\mathcal{B}_{n, m}$ such that
\begin{equation}\label{themaino3}
E(w_{0})=\inf_{w\in\mathcal{B}_{n, m}}E(w).
\end{equation}
Since $\kappa\in \mathcal{B}_{n, m}$, it is clear that
\begin{equation}\label{themaino3'}
E(w_{0})\leq \left(\frac{1}{2}-\frac{1}{p+1}\right)(\frac{1}{p-1})^{\frac{p+1}{p-1}}.
\end{equation}
We will prove by induction  that
\begin{equation}\label{lowestenergysolution}
E(w_{0})=\left(\frac{1}{2}-\frac{1}{p+1}\right)(\frac{1}{p-1})^{\frac{p+1}{p-1}}.
\end{equation}

If $n=1, 2$, we get from Theorem \ref{theoremGiga-Kohn1985} that  $w_{0}$ is the constant solution. In particular, \eqref{lowestenergysolution} holds.

Let us assume that Theorem \ref{maintheorem1} holds for $n-1$ with $n\geq 3$. That is to say,
\begin{equation}\label{themaino2}
E(w)\geq \left(\frac{1}{2}-\frac{1}{p+1}\right)(\frac{1}{p-1})^{\frac{p+1}{p-1}}
\end{equation}
for any $w\in\mathcal{B}_{n-1, m}$. We want to show that Theorem \ref{maintheorem1}
holds for dimension $n$. If $w_{0}$ is the constant solution of \eqref{SC1}, then we are done. Therefore, we assume $w_{0}$ is not the constant solution of \eqref{SC1}. Set $u_{0}(x, t)=(-t)^{-1/(p-1)}w_{0}(x/\sqrt{-t})$. Since $w_{0}\in\mathcal{B}_{n, m}$,  $u_{0}$ is a solution of
\[\partial_{t}u=\Delta u+u^{p},\quad\text{in}~\mathbb{R}^{n}\times(-\infty, 0)\]
satisfying
\[u(x, t)\leq m(-t)^{-\frac{1}{p-1}},\quad\text{in}~\mathbb{R}^{n}\times(-\infty, 0).\]
Assume $u_{0}$ blows up at some point $x_{0}\neq 0$ and $\Theta(x_{0}, 0; u_{0})=\Theta(0, 0; u_{0})$.  Lemma \ref{densityfunctionproperties2} implies that $w_{0}$ is translation invariant in the $x_{0}$ direction. Therefore, $w_{0}\in \mathcal{B}_{n-1, m}$. By the induction assumption,
\[E(w_{0})\geq \left(\frac{1}{2}-\frac{1}{p-1}\right)(\frac{1}{p-1})^{\frac{p+1}{p-1}},\]
and the proof is complete.

If $\Theta(x_{0}, 0; u_{0})<\Theta(0, 0; u_{0})$, we can also apply the Federer dimension reduction to get a function $w_{1}\in \mathcal{B}_{n, m}$ which is translation invariant in the $x_{0}$ direction. In particular,
 $w_{1}\in \mathcal{B}_{n-1, m}$. By the induction assumption and the properties of the density function, we have
\[E(w_{0})=\Theta(0, 0; u_{0})>E(w_{1})\geq \left(\frac{1}{2}-\frac{1}{p+1}\right)(\frac{1}{p-1})^{\frac{p+1}{p-1}}.\]
In view of \eqref{themaino3'}, this is a contradiction.

Next, assume that $u$ does not blow up at any $x_{0}\neq 0$. Then there exist $C_{1}, \delta_{0}$ such that
\[u_{0}\leq c_{1},\quad\text{in}~\{1/4<|x|<1/2\}\times(-\delta_{0}, 0).\]
This is equivalent to
\[w_{0}\leq C_{1}|y|^{-\frac{2}{p-1}},\quad\text{in}~|y|\geq \frac{1}{4\sqrt{\delta_{0}}}.\]
 Thus $w_{0}\in \mathcal{B}_{n, m}$ satisfies \eqref{lowerbound}. By Proposition \ref{decreaseentropy},
\[E(w_{0})> E(\kappa)=\left(\frac{1}{2}-\frac{1}{p-1}\right)(\frac{1}{p-1})^{\frac{p+1}{p-1}},\]
which contradicts \eqref{themaino3'} again. 
\end{proof}

\section{Proof of Theorem \ref{maintheorem2} and Proposition \ref{blowupsetstructure}}
In this section, we will prove Theorem \ref{maintheorem2} and Proposition \ref{blowupsetstructure}. 
\subsection{Energy gap}
First, we prove that not only does the positive constant  solution achieve the lowest weighted energy among functions in $\mathcal{B}_{n, m}$, but there is a gap to the second lowest.
\begin{proof}[Proof of Theorem \ref{maintheorem2}]
Suppose instead that there is a sequence of self-similar solutions $\{w_{i}\}\subset\mathcal{B}_{n, m}$, which is not equal  to the positive constant solution with
\begin{equation}\label{gap1}
E(w_{i})< E(\kappa)+2^{-i}.
\end{equation}
Since $\{w_{i}\}\subset \mathcal{B}_{n, m}$,  regularity theories imply there exists a positive constant $C(n, p, m)$  such that for any $i$
\[|w_{i}|+|\nabla w_{i}|+|\nabla^{2} w_{i}|+|\nabla^{3}w_{i}|\leq C(n, p, m).\]
By the Ascoli-Arzel\`{a} theorem,   there exists a function  $w_{\infty}\in C^{2}(\mathbb{R}^{n})$ such that $w_{i}\rightarrow w_{\infty}$ uniformly on compact subset of $\mathbb{R}^{n}$. Moreover, $w_{\infty}$ is a solution of \eqref{SC1} such that
\[E(w_{\infty})\leq E(\kappa).\]
Because Theorem \ref{maintheorem1} says that $\kappa$ is the unique least energy solution,  up to a subsequence, $w_{i}\rightarrow \kappa$ as $i\rightarrow\infty$ uniformly on compact subset of $\mathbb{R}^{n}$ and $E(w_{\infty})=E(\kappa)$.

Let $\lambda_{1, i}$ be the first eigenfunction of \[L\psi=\Delta\psi-\frac{y}{2}\cdot\nabla\psi-\frac{1}{p-1}\psi+pw_{i}^{p-1}\psi\]
and let $f_{i}$ be a positive solution of the equation
\begin{equation}\label{fiequation}
\Delta f_{i}-\frac{y}{2}\cdot\nabla f_{i}-\frac{1}{p-1}f_{i}+pw_{i}^{p-1}f_{i}+\lambda_{1, i}f_{i}=0,\quad\text{in}~\mathbb{R}^{n}.
\end{equation}
 Without loss of generality, we may assume that $\int_{\mathbb{R}^{n}}f_{i}^{2}\rho dy=1$. Since $\{w_{i}\}\in\mathcal{B}_{n, m}$, it follows easily from \eqref{thevariationalcharacterization} that the eigenvalues $\lambda_{1, i}$ are uniformly bounded. By choosing a subsequence if necessary, we may assume that
 \[\lambda_{1, i}\to\lambda_{1, \infty},\quad\text{as}~i\to\infty.\]
 Multiplying  both sides of  \eqref{fiequation} by $f_{i}\rho$ and integrating by parts, we   see
 $\{f_{i}\}$ are uniformly bounded in $H^{1}_{w}(\mathbb{R}^{n})$. Then standard elliptic regularity theory implies that
  $\{f_{i}\}$ are uniformly bounded in $C^{2,\alpha}_{loc}(\mathbb{R}^{n})$. Using the Ascoli-Arzel\`{a}
   theorem again, we know that there exists a function $f_{\infty}$ such that
    $f_{i}\to f_{\infty}$ uniformly on compact subsets of $\mathbb{R}^{n}$. Taking limit in \eqref{fiequation}, we deduce that $f_{\infty}$ is a $C^{2}$ solution of the equation
 \begin{equation}\label{finftyequation}
\Delta f_{\infty}-\frac{y}{2}\cdot\nabla f_{\infty}+f_{\infty}+\lambda_{1, \infty}f_{\infty}=0,\quad\text{in}~\mathbb{R}^{n}.
\end{equation}
Since we have assumed that $f_{i}$ is positive and $\int_{\mathbb{R}^{n}}f_{i}^{2}\rho dy=1$,  $f_{\infty}$ is positive and $\int_{\mathbb{R}^{n}}f_{\infty}^{2}\rho dy=1$. Hence $\lambda_{\infty}$ is the first eigenvalue of the linear operator
\[L\psi=\Delta\psi-\frac{y}{2}\cdot\nabla\psi+\psi+\lambda\psi\]
and $f_{\infty}$ is the associated eigenfunction.  By Lemma \ref{lemsp}, $\lambda_{1, \infty}=-1$ and $f_{\infty}=1$. By Lemma \ref{lemeig} and Proposition \ref{ChCS}, we have
\[\int_{\mathbb{R}^{n}}f_{i}\left(\frac{2}{p-1}w_{i}+y\cdot\nabla w_{i}\right)\rho dy=0.\]
Letting $i$ tend to $\infty$ gives
\[\int_{\mathbb{R}^{n}}f_{\infty}\left(\frac{2}{p-1}w_{\infty}+y\cdot\nabla w_{\infty}\right)\rho dy=0.\]
Since $f_{\infty}=1$ and $w_{\infty}=\kappa$, this is a contradiction. 
 \end{proof}
 
\subsection{Proof of Proposition \ref{blowupsetstructure}}
In this section, we will combine Theorem \ref{maintheorem1} and Theorem \ref{maintheorem2} to  prove Proposition \ref{blowupsetstructure}.
\begin{proof}[Proof of Proposition \ref{blowupsetstructure}]
We know that  the blow up set can be stratified into subsets
\[\mathcal{S}_{0}\subset\mathcal{S}_{1}\cdots\mathcal{S}_{n}=\Sigma.\]
The set $\mathcal{S}_{k}$ consists of  all blow up points where the tangent flows are at most translation invariant in $k$ directions. Moreover, if for any $R>0$, we set $\mathcal{S}_{k}(R)=\mathcal{S}_{k}\cap B_{R}(0)$, then $\dim_{\mathcal{H}}(\mathcal{S}_{k}(R))\leq k$. We set
\[\Sigma_{n-1}=(\Sigma\cap B_{R}(0))\backslash \mathcal{S}_{n-3}(R), \quad \Sigma_{n-2}= (\Sigma\cap B_{R}(0))\backslash \Sigma_{n-1},\]
then $\Sigma_{R}=\Sigma\cap B_{R}(0)=\Sigma_{n-1}\cup\Sigma_{n-2}$ and $\dim_{\mathcal{H}}(\Sigma_{n-2})\leq n-3$. Thus we have proved both (1) and (3).

If $(x_{0}, T)\in \Sigma_{n-1}$, then the tangent flows are at least translation invariant in $n-2$ directions. Hence any tangent flow can be regarded  as a bounded solution of \eqref{SC1} in $\mathbb{R}^{2}$. By Theorem \ref{theoremGiga-Kohn1985} and \cite[Theorem 3.1]{Dushizhong2019}, any tangent flow is the positive constant solution of \eqref{SC1}. Hence we have proved (4).

In order to finish the proof of Proposition \ref{blowupsetstructure}, it still remains to prove (2). First, we prove that $\Sigma_{n-1}$ is relative open in $\Sigma_{R}$. Assume it is false, then there exists a point $(x_{0}, T)\in \Sigma_{n-1}$ and  a sequence $\{(x_{i}, T)\}\subset\Sigma_{n-2}$ such that $\lim_{i\to\infty} x_{i}=x_{0}$.  Since $\{(x_{i}, T)\}\subset\Sigma_{n-2}$ , we get from Theorem \ref{maintheorem2} that
\[\Theta(x_{i}, T; u)\geq E(\kappa)+\epsilon.\]
Since $(x_{0}, T)\in \Sigma_{n-1}$, then 
\[\Theta(x_{0}, T; u)= E(\kappa).\]
Applying Lemma \ref{densityfunctionproperties1}, we have
\[\Theta(x_{0}, T; u)= E(\kappa)\geq\limsup_{i\to\infty}\Theta(x_{i}, T; u)\geq E(\kappa)+\epsilon,\]
which is a contradiction. Thus we have shown that  $\Sigma_{n-1}$ is relative open in $\Sigma_{R}$. Once this is established, the proof of main theorem in \cite{Velazquez1993-2} yields  $\Sigma_{n-1}$ is $n-1 $ rectifiable.

\end{proof}
\section{Proof of Theorem \ref{maintheorem2}}

In this  section, we prove Theorem \ref{maintheorem2}. First, we study the case when $1\leq n\leq 3, 1<p<\infty$ or $n\geq 3$ and $1<p<(n+1)/(n-3)$.
If we view any homogeneous positive solution $w$ of
\eqref{ellipticsolution} as a suitable weak solution of \eqref{SC1}, then $w$ is not the lowest energy solution among self-similar solutions. The proof of this fact is based on  the following classification result.
\begin{proposition}\label{classifyode}
 Assume $1\leq n\leq 3, 1<p<\infty$ or $n\geq 3$ and $1<p<(n+1)/(n-3)$. If $\Phi$ is a bounded positive solution of the equation \eqref{sphereequation}, then $\Phi=\beta^{1/(p-1)}$ with
 \[\beta=\frac{2}{p-1}\left(n-2-\frac{2}{p-1}\right).\]
\end{proposition}
\begin{proof}
The proof of this rigidity result can be found in \cite[Theorem 6.1]{Veron1991}, \cite[Theorem B.2]{Gidas-Spruck1981} and \cite[Theorem 1]{Dolbeault2014}.
\end{proof}
\begin{lemma}\label{twodiemsionenergy}
Assume $n\geq 3$ and $1<p<(n+1)/(n-3)$. If $w$ is a positive homogeneous of the equation
\begin{equation}\label{twodimensionequation}
\Delta w+w^{p}=0,\quad\text{in}~\mathbb{R}^{n}\backslash\{0\},
\end{equation}
then
\begin{equation}\label{radialsingularsolutionenergy}
E(w)>\left(\frac{1}{2}-\frac{1}{p+1}\right)\left(\frac{1}{p-1}\right)^{\frac{p+1}{p-1}}.
\end{equation}
\end{lemma}
\begin{proof}
Let $\Phi$ be a positive function  defined on $\mathbb{S}^{n-1}$ such that $w(r, \theta)=r^{-2/(p-1)}\Phi(\theta)$. Then $\Phi$ satisfies \begin{equation}\label{sphereequation}
\Delta_{\mathbb{S}^{n-1}}\Phi-\beta\Phi+\Phi^{p}=0,
\end{equation}
where $\mathbb{S}^{n-1}$ is the unit sphere in $\mathbb{R}^{n-1}$ and
\[\beta=\frac{2}{p-1}\left(n-2-\frac{2}{p-1}\right).\]
By Proposition \ref{classifyode},  $\Phi=\beta^{1/(p-1)}$. It follows that $w=\beta^{1/(p-1)}r^{-\frac{2}{p-1}}$.  In particular, we have
\[
\begin{aligned}
E(w)=&\left(\frac{1}{2}-\frac{1}{p+1}\right)\int_{\mathbb{R}^{n}}w^{p+1}\rho dy\\
=&\left(\frac{1}{2}-\frac{1}{p+1}\right)\left[\frac{2}{p-1}\left(n-2-\frac{2}{p-1}\right)\right]^{\frac{p+1}{p-1}}\int_{\mathbb{R}^{n}}|y|^{\frac{-2(p+1)}{p-1}}\rho dy\\
\end{aligned}
\]
Thus in order to  get \eqref{radialsingularsolutionenergy}, it sufficies to show that
\begin{equation}\label{comparewithsingularsolution}
\left[\frac{2}{p-1}\left(n-2-\frac{2}{p-1}\right)\right]^{\frac{p+1}{p-1}}\int_{\mathbb{R}^{n}}|y|^{\frac{-2(p+1)}{p-1}}\rho dy>(\frac{1}{p-1})^{\frac{p+1}{p-1}}.
\end{equation}
This has been  essentially proved by Matano and Merle  in \cite{Merle-Matano2011}.   In \cite{Merle-Matano2011}, Matano and Merle  gave a ``parabolic proof'' without using the explicit formula. To be self contained, we will give a more direct proof in Appendix B.
\end{proof}
\begin{lemma}\label{subcriticalsuitableweaksolutions}
 Assume $n\leq 2$ or $n\geq 3, 1<p\leq (n+2)/(n-2)$, $w$ is a suitable weak solution of \eqref{SC1}, then $w=0$ or $w=\pm \kappa$.
\end{lemma}
\begin{proof}
In \cite{Giga-Kohn1985}, the proof of Theorem 1 is obtained by choosing three test functions and then combining the three identities in a suitable way.
 By approximation, we are still able to do the same computation for suitable weak solutions to derive the Pohozaev identity.
\end{proof}
\begin{proof}[Proof of Theorem \ref{maintheorem2}]
Similar to the proof of Theorem \ref{maintheorem1}, we will prove by induction.  If $n=1, 2$, we get from Lemma \ref{subcriticalsuitableweaksolutions} that $\mathcal{B}_{n}$ consists only of the constant solutions, so Theorem \ref{maintheorem2} holds.

Assume that Theorem \ref{maintheorem2} holds for $n-1$ with $n\geq 3$.  We want to show that Theorem \ref{maintheorem2}
holds for dimension $n$.

Let $u(x, t):=(-t)^{-1/(p-1)}w(x/\sqrt{-t})$, which is a suitable weak solution of \eqref{Fujitaequation}.

Assume $u$ blows up at some point $x_{0}\neq 0$, then we can apply the Federer dimension reduction to get    a solution $w_{1}$ of \eqref{SC1} such that $w_{1}$ is translation invariant in the $x_{0}$ direction. Moreover,
\[E(w_{1})=\Theta(x_{0}, 0; u)\leq \Theta(0, 0; u)\leq E(w).\]
By the inductive assumption,
\[E(w)\geq \left(\frac{1}{2}-\frac{1}{p+1}\right)(\frac{1}{p-1})^{\frac{p+1}{p-1}},\]
and the proof is complete.

Next,   assume  $u(x_{0}, 0)\neq 0$ for any $x_{0}\neq 0$. By
standard parabolic regularity theory, there exist $C_{1}, \delta_{0}$ such that
\[u\leq C_{1},\quad\text{in}~\{1/4<|x|<1/2\}\times [-\delta_{0}, 0].\]
This is equivalent to
\[w\leq C_{1}|y|^{-\frac{2}{p-1}},\quad\text{in}~|y|\geq \frac{1}{4\sqrt{\delta_{0}}}.\]
There are two cases.
\begin{itemize}
\item Case 1: There exists a constant $cC_{0}$ such that
\[w\leq C_{0},\quad\text{in}~\mathbb{R}^{n};\]
\item Case 2: There exists a point $y_{0}\in\mathbb{R}^{n}$ such that $w_{0}$ blows up at $y_{0}$.
\end{itemize}

In Case 1,  $w_{0}$ satisfies \eqref{lowerbound}. Then we can apply Lemma \ref{decreaseentropy} to obtain 
\[E(w)>\left(\frac{1}{2}-\frac{1}{p+1}\right)(\frac{1}{p-1})^{\frac{p+1}{p-1}}.\]

In Case 2, we choose a sequence $\{\lambda_{k}\}$ such that $\lim_{k\rightarrow\infty}\lambda_{k}=+\infty.$ For any $k$,  define
$w_{k}(y)=\lambda_{k}^{-2/(p-1)}w(\lambda_{k}^{-1}(y-y_{0}))$. Then $w_{k}$ satisfies the equation
\begin{equation}\label{eqscale}
\Delta w_{k}-\frac{\lambda_{k}^{-2}}{2}(y+y_{0})\cdot\nabla w_{k}-\frac{\lambda_{k}^{-2}}{p-1}w_{k}
+w_{k}^{p}=0.
\end{equation}
By the convergence theories obtained in \cite{Wang-Wei2021}, there exists a function $w_{\infty}$ such that $\lim_{k\rightarrow\infty}w_{k}=w_{\infty}$. Moreover, $w_{\infty}$ is a homogeneous suitable weak solution of the equation
\[\Delta w_{\infty}+w_{\infty}^{p}=0,\quad\text{in}~\mathbb{R}^{n}.\]
Therefore, there exists a function $\Phi$ defined on $\mathbb{S}^{n-1}$ such that $w_{\infty}(r, \theta)=r^{-\frac{2}{p-1}}\Phi(\theta)$,
where $(r, \theta)$ is the polar coordinate. If $\Phi$ is not smooth on $\mathbb{S}^{n-1}$, we can apply
the Federer dimension procedure (see \cite[Section 4]{Wang-Wei2015}) once again
to reduce to the case that $\Phi$ is smooth on $\mathbb{S}^{n-2}$. By
Lemma \ref{twodiemsionenergy},
\[E(w)\geq E(w_{\infty})>\left(\frac{1}{2}-\frac{1}{p+1}\right)(\frac{1}{p-1})^{\frac{p+1}{p-1}}.\]
Hence the proof of Theorem \ref{maintheorem2} is completed.
\end{proof}

\section{Some open questions}
The results above also suggest several closely related questions. The first one is  about the  energy of functions in $\mathcal{B}_{n}$.

\textbf{Open question 1:} Does Theorem \ref{theoremmain3} hold for all $n\geq 4, p\geq (n+1)/(n-3)$ and $w\in\mathcal{B}_{n}$?

As mentioned before, if $n\geq 3$, $p>(n+2)/(n-2)$, then $\mathcal{B}_{n}$ contains ``elliptic solutions''. That is, $2/(p-1)$ homogeneous solutions of the elliptic equation
\begin{equation}\label{ellipticsolution}
 \Delta w+w^{p}=0,\quad\text{in}~\mathbb{R}^{n}\backslash\{0\}.
 \end{equation}
 Let
 \[\mathcal{B}_{n, e}=\{w\in\mathcal{B}_{n}: w~\text{is a $2/(p-1)$ homogeneous solution of}~ \eqref{ellipticsolution}\}.\]
 It turns out that functions in $\mathcal{B}_{n, e}$ are the main obstructions to prove the first open problem.  If $w\in\mathcal{B}_{n, e}$, then there exists a positive function $\Phi$ defined on $\mathbb{S}^{n-1}$ such that
$w(r, \theta)=r^{-2/(p-1)}\Phi(\theta)$, where $(r, \theta)$ is the polar coordinate. Since $w$ satisfies \eqref{ellipticsolution}, then $\Phi$ is  a positive solution of the equation \eqref{sphereequation}.

Let $\Phi$ be a bounded positive solution of \eqref{sphereequation}, we define
\[\mathcal{E}(\Phi)=\frac{1}{2}\int_{\mathbb{S}^{n-1}}|\nabla\Phi|^{2}d\theta+\frac{1}{2\beta}\int_{\mathbb{S}^{n-1}}\Phi^{2}d\theta-\frac{1}{p+1} \int_{\mathbb{S}^{n-1}}\Phi^{p+1}d\theta.\]

\textbf{Open question 2:} If $n\geq 4, p\geq (n+1)/(n-3)$, does
\[\mathcal{E}(\Phi)\geq \mathcal{E}(\beta^{\frac{1}{p-1}})\]
for any bounded positive solution of \eqref{sphereequation}?

Finally, for finite time blow up solutions of \eqref{Cauchyproblem}, we propose the following.

\textbf{Open question 3:} Without condition \ref{TypeIassumption}, does Proposition \ref{blowupsetstructure} hold for all $n\geq 3$ and $p=(m+2)/(m-2)$ with $3\leq m\leq n$?
\begin{remark}
If $n\geq 7$ and $p=(n+2)/(n-2)$, it is proved by Wang-Wei \cite{Wang-Wei2021} that any finite time blow up solution of \eqref{Cauchyproblem} must be of type I. Combing this fact with \cite[Theorem 1]{Giga-Kohn1985} and main theorem in \cite{Velazquez1993-2}, we know that the blow up set $\Sigma$ is $n-1$ rectifiable. Moreover, any tangent flow is the constant solution of \eqref{SC1}.
\end{remark}

\section*{Appendix A}
In this appendix, we prove Theorem \ref{secondvariation}.
\begin{proof}[Proof of Theorem \ref{secondvariation}]
Because $w$ is a bounded solution of \eqref{SC1-2},  for any $\phi\in C_{0}^{\infty}(\mathbb{R}^{n})$,
\[\frac{\partial}{\partial s} (F_{x(s), t(s)}(w+s\phi))|_{s=0}=0.\]
Substituting \eqref{mlpsi}, \eqref{firstvariation2024-04-05} and \eqref{Testfunction} into \eqref{Thegeneralsecondvariationformula}, we have
\begin{align}\label{Thegeneralsecondvariationformula'}
&\frac{\partial^{2}}{\partial s^2} (F_{x(s), t(s)}(w+s\phi))|_{s=0}\notag\\
=&\frac{p+1}{2(p-1)}\frac{2}{p-1}h^{2}\int_{\mathbb{R}^{n}}|\nabla w|^{2}\rho dy\notag\\
&-\frac{1}{p-1}\frac{2}{p-1}h^{2}\int_{\mathbb{R}^{n}}|w|^{p+1}\rho dy\notag\\
&-\frac{1}{(p-1)^{2}}(-\frac{2}{p-1}+1)h^{2}\int_{\mathbb{R}^{n}}w^{2}\rho dy\notag\\
&+\frac{2}{p-1}h\int_{\mathbb{R}^{n}}w\phi\rho dy\notag\\
&-\frac{p+1}{p-1}h\int_{\mathbb{R}^{n}}|\nabla w|^{2}\rho[\frac{nh}{2}+\frac{y\cdot y_{0}}{2}-\frac{h|y|^{2}}{4}]dy\notag\\
&+\frac{2}{p-1}h\int_{\mathbb{R}^{n}}|w|^{p+1}\rho[\frac{nh}{2}+\frac{y\cdot y_{0}}{2}-\frac{h|y|^{2}}{4}]dy\\
&-\frac{2}{(p-1)^{2}}h\int_{\mathbb{R}^{n}}w^{2}\rho[\frac{nh}{2}+\frac{y\cdot y_{0}}{2}-\frac{h|y|^{2}}{4}]dy\notag\\
&+\int_{\mathbb{R}^{n}}|\nabla\phi|^{2}\rho dy-p\int_{\mathbb{R}^{n}}|w|^{p-1}\phi^{2}\rho dy+\frac{1}{p-1}\int_{\mathbb{R}^{n}}\phi^{2}\rho dy\notag\\
&+2\int_{\mathbb{R}^{n}}(\nabla w\cdot\nabla\phi)\rho[\frac{y\cdot y_{0}}{2}-\frac{h|y|^{2}}{4}]dy\notag\\
&-2\int_{\mathbb{R}^{n}}|w|^{p-1}w\phi \rho[\frac{y\cdot y_{0}}{2}-\frac{h|y|^{2}}{4}]dy\notag\\
&+\frac{2}{p-1}\int_{\mathbb{R}^{n}}w\phi\rho[\frac{y\cdot y_{0}}{2}-\frac{h|y|^{2}}{4}]dy\notag\\
&+\frac{1}{2}\int_{\mathbb{R}^{n}}|\nabla w|^{2} \rho\{[\frac{nh}{2}+\frac{y\cdot y_{0}}{2}-\frac{h|y|^{2}}{4}]^{2}+\frac{nh^{2}}{2}-\frac{|y_{0}|^{2}}{2}-\frac{h^{2}|y|^{2}}{2}\}dy\notag\\
&-\frac{1}{p+1}\int_{\mathbb{R}^{n}}|w|^{p+1}\rho\{[\frac{nh}{2}+\frac{y\cdot y_{0}}{2}-\frac{h|y|^{2}}{4}]^{2}+\frac{nh^{2}}{2}-\frac{|y_{0}|^{2}}{2}-\frac{h^{2}|y|^{2}}{2}\}dy\notag\\
&+\frac{1}{2(p-1)}\int_{\mathbb{R}^{n}}w^{2}\rho\{[\frac{nh}{2}+\frac{y\cdot y_{0}}{2}-\frac{h|y|^{2}}{4}]^{2}+\frac{nh^{2}}{2}-\frac{|y_{0}|^{2}}{2}-\frac{h^{2}|y|^{2}}{2}\}dy\notag.
\end{align}
Notice that
\begin{align}
&-\frac{p+1}{p-1}h\int_{\mathbb{R}^{n}}|\nabla w|^{2}\rho[\frac{nh}{2}+\frac{y\cdot y_{0}}{2}-\frac{h|y|^{2}}{4}]dy\notag\\
&+\frac{2}{p-1}h\int_{\mathbb{R}^{n}}|w|^{p+1}\rho[\frac{nh}{2}+\frac{y\cdot y_{0}}{2}-\frac{h|y|^{2}}{4}]dy\notag\\
&-\frac{2}{(p-1)^{2}}h\int_{\mathbb{R}^{n}}w^{2}\rho[\frac{nh}{2}+\frac{y\cdot y_{0}}{2}-\frac{h|y|^{2}}{4}]dy\notag\\
=&-2h(\frac{p}{p-1}+\frac{1}{p-1})\frac{1}{2}\int_{\mathbb{R}^{n}}|\nabla w|^{2}\rho[\frac{nh}{2}+\frac{y\cdot y_{0}}{2}-\frac{h|y|^{2}}{4}]dy\\
&+2h(\frac{p}{p-1}+\frac{1}{p-1})\frac{1}{p+1}\int_{\mathbb{R}^{n}}|w|^{p+1}\rho[\frac{nh}{2}+\frac{y\cdot y_{0}}{2}-\frac{h|y|^{2}}{4}]dy\notag\\
&-2h(\frac{p}{p-1}+\frac{1}{p-1})\frac{1}{2(p-1)}\int_{\mathbb{R}^{n}}w^{2}\rho[\frac{nh}{2}+\frac{y\cdot y_{0}}{2}-\frac{h|y|^{2}}{4}]dy \notag\\
&+\frac{h}{p-1}\int_{\mathbb{R}^{n}}w^{2}\rho[\frac{nh}{2}+\frac{y\cdot y_{0}}{2}-\frac{h|y|^{2}}{4}]dy \notag.
\end{align}
Thus
\begin{align}\label{Newsecondvariationformula}
&\frac{\partial^{2}}{\partial s^2} (F_{x(s), t(s)}(w+s\phi))|_{s=0}\notag\\
=&-\frac{1}{(p-1)^{2}}h^{2}\int_{\mathbb{R}^{n}}w^{2}\rho dy+\frac{2}{p-1}h\int_{\mathbb{R}^{n}}w\phi\rho dy\notag\\
&-\frac{2ph}{p-1}\frac{1}{2}\int_{\mathbb{R}^{n}}|\nabla w|^{2}\rho[\frac{nh}{2}-\frac{h|y|^{2}}{4}]dy\notag\\
&+\frac{2ph}{p-1}\frac{1}{p+1}\int_{\mathbb{R}^{n}}|w|^{p+1}\rho[\frac{nh}{2}-\frac{h|y|^{2}}{4}]dy\notag\\
&-\frac{2ph}{p-1}\frac{1}{2(p-1)}\int_{\mathbb{R}^{n}}w^{2}\rho[\frac{nh}{2}-\frac{h|y|^{2}}{4}]dy \notag\\
&+\frac{h}{p-1}\int_{\mathbb{R}^{n}}w^{2}\rho[\frac{nh}{2}+\frac{y\cdot y_{0}}{2}-\frac{h|y|^{2}}{4}]dy\notag\\
&+\int_{\mathbb{R}^{n}}|\nabla\phi|^{2}\rho dy-p\int_{\mathbb{R}^{n}}|w|^{p-1}\phi^{2}\rho dy+\frac{1}{p-1}\int_{\mathbb{R}^{n}}\phi^{2}\rho dy\\
&+2\int_{\mathbb{R}^{n}}(\nabla w\cdot\nabla\phi)\rho[\frac{y\cdot y_{0}}{2}-\frac{h|y|^{2}}{4}]dy\notag\\
&-2\int_{\mathbb{R}^{n}}|w|^{p-1}w\phi \rho[\frac{y\cdot y_{0}}{2}-\frac{h|y|^{2}}{4}]dy\notag\\
&+\frac{2}{p-1}\int_{\mathbb{R}^{n}}w\phi\rho[\frac{y\cdot y_{0}}{2}-\frac{h|y|^{2}}{4}]dy\notag\\
&+\frac{1}{2}\int_{\mathbb{R}^{n}}|\nabla w|^{2} \rho\{[\frac{nh}{2}+\frac{y\cdot y_{0}}{2}-\frac{h|y|^{2}}{4}]^{2}+\frac{nh^{2}}{2}-\frac{|y_{0}|^{2}}{2}-\frac{h^{2}|y|^{2}}{2}\}dy\notag\\
&-\frac{1}{p+1}\int_{\mathbb{R}^{n}}|w|^{p+1}\rho\{[\frac{nh}{2}+\frac{y\cdot y_{0}}{2}-\frac{h|y|^{2}}{4}]^{2}+\frac{nh^{2}}{2}-\frac{|y_{0}|^{2}}{2}-\frac{h^{2}|y|^{2}}{2}\}dy\notag\\
&+\frac{1}{2(p-1)}\int_{\mathbb{R}^{n}}w^{2}\rho\{[\frac{nh}{2}+\frac{y\cdot y_{0}}{2}-\frac{h|y|^{2}}{4}]^{2}+\frac{nh^{2}}{2}-\frac{|y_{0}|^{2}}{2}-\frac{h^{2}|y|^{2}}{2}\}dy\notag.
\end{align}
Multiplying   both sides of \eqref{SC1-2} by $(y\cdot y_{0})\phi\rho$ and  integrating by parts, we have
\begin{align}\label{mulyy0rho}
0=&\int_{\mathbb{R}^{n}}[\frac{1}{\rho}{\rm div}(\rho\nabla w)-\frac{1}{p-1}w+|w|^{p-1}w](y\cdot y_{0})\phi\rho dy\notag\\
=&-\int_{\mathbb{R}^{n}}\rho\nabla w\nabla((y\cdot y_{0})\phi)dy+\int_{\mathbb{R}^{n}}|w|^{p-1}w(y\cdot y_{0})\phi\rho dy\notag\\
&-\frac{1}{p-1}\int_{\mathbb{R}^{n}}w(y\cdot y_{0})\phi\rho dy\\
=&-\int_{\mathbb{R}^{n}}(\nabla w\cdot y_{0})\phi\rho dy-\int_{\mathbb{R}^{n}}(\nabla w\cdot\nabla\phi)(y\cdot y_{0})\rho dy\notag\\
&+\int_{\mathbb{R}^{n}}|w|^{p-1}w(y\cdot y_{0})\phi\rho dy-\frac{1}{p-1}\int_{\mathbb{R}^{n}}w(y\cdot y_{0})\phi\rho dy\notag.
\end{align}
Multiplying   both sides of \eqref{SC1-2} by $|y|^{2}\phi\rho$ and   integrating by parts, we have
\begin{align}\label{ysquarephirho}
0=&\int_{\mathbb{R}^{n}}[\frac{1}{\rho}{\rm div}(\rho\nabla w)-\frac{1}{p-1}w+|w|^{p-1}w]|y|^{2}\phi\rho dy\notag\\
=&-\int_{\mathbb{R}^{n}}\rho\nabla w\nabla(|y|^{2}\phi)dy+\int_{\mathbb{R}^{n}}|w|^{p-1}w|y|^{2}\phi\rho dy\notag\\
&-\frac{1}{p-1}\int_{\mathbb{R}^{n}}w|y|^{2}\phi\rho dy\\
=&-2\int_{\mathbb{R}^{n}}(\nabla w\cdot y)\phi\rho dy-\int_{\mathbb{R}^{n}}(\nabla w\cdot\nabla\phi)|y|^{2}\rho dy\notag\\
&+\int_{\mathbb{R}^{n}}|w|^{p-1}w|y|^{2}\phi\rho dy-\frac{1}{p-1}\int_{\mathbb{R}^{n}}w|y|^{2}\phi\rho dy\notag.
\end{align}
Substituting \eqref{mulyy0rho} and \eqref{ysquarephirho} into \eqref{Newsecondvariationformula}, we get
\begin{align}\label{Thefirstsimplication}
&\frac{\partial^{2}}{\partial s^2} (F_{x(s), t(s)}(w+s\phi))|_{s=0}\notag\\
&=\int_{\mathbb{R}^{n}}|\nabla\phi|^{2}\rho dy-p\int_{\mathbb{R}^{n}}|w|^{p-1}\phi^{2}\rho dy+\frac{1}{p-1}\int_{\mathbb{R}^{n}}\phi^{2}\rho dy\notag\\
&+h\int_{\mathbb{R}^{n}}(\frac{2}{p-1}w+\nabla w\cdot y)\phi\rho dy-\int_{\mathbb{R}^{n}}(\nabla w\cdot y_{0})\phi\rho dy\notag\\
&-\frac{2ph}{p-1}\frac{1}{2}\int_{\mathbb{R}^{n}}|\nabla w|^{2}\rho[\frac{nh}{2}-\frac{h|y|^{2}}{4}]dy\notag\\
&+\frac{2ph}{p-1}\frac{1}{p+1}\int_{\mathbb{R}^{n}}|w|^{p+1}\rho[\frac{nh}{2}-\frac{h|y|^{2}}{4}]dy\notag\\
&-\frac{2ph}{p-1}\frac{1}{2(p-1)}\int_{\mathbb{R}^{n}}w^{2}\rho[\frac{nh}{2}-\frac{h|y|^{2}}{4}]dy \\
&+\frac{h}{p-1}\int_{\mathbb{R}^{n}}w^{2}\rho[\frac{nh}{2}+\frac{y\cdot y_{0}}{2}-\frac{h|y|^{2}}{4}]dy-\frac{h^{2}}{(p-1)^{2}}\int_{\mathbb{R}^{n}}w^{2}\rho dy\notag\\
&+\frac{1}{2}\int_{\mathbb{R}^{n}}|\nabla w|^{2} \rho\{[\frac{nh}{2}+\frac{y\cdot y_{0}}{2}-\frac{h|y|^{2}}{4}]^{2}+\frac{nh^{2}}{2}-\frac{|y_{0}|^{2}}{2}-\frac{h^{2}|y|^{2}}{2}\}dy\notag\\
&-\frac{1}{p+1}\int_{\mathbb{R}^{n}}|w|^{p+1}\rho\{[\frac{nh}{2}+\frac{y\cdot y_{0}}{2}-\frac{h|y|^{2}}{4}]^{2}+\frac{nh^{2}}{2}-\frac{|y_{0}|^{2}}{2}-\frac{h^{2}|y|^{2}}{2}\}dy\notag\\
&+\frac{1}{2(p-1)}\int_{\mathbb{R}^{n}}w^{2}\rho\{[\frac{nh}{2}+\frac{y\cdot y_{0}}{2}-\frac{h|y|^{2}}{4}]^{2}+\frac{nh^{2}}{2}-\frac{|y_{0}|^{2}}{2}-\frac{h^{2}|y|^{2}}{2}\}dy\notag.
\end{align}
To continue the proof, we   need several more identities. Multiplying   both sides of \eqref{SC1-2} by $(\frac{n}{2}-\frac{|y|^{2}}{4})(\nabla w\cdot y_{0})\rho$ and   integrating by parts, we   get 
\begin{align}\label{24-5}
0=&\int_{\mathbb{R}^{n}}[\frac{1}{\rho}{\rm div}(\rho\nabla w)-\frac{1}{p-1}w+|w|^{p-1}w](\frac{n}{2}-\frac{|y|^{2}}{4})(\nabla w\cdot y_{0})\rho dy\notag\\
=&-\int_{\mathbb{R}^{n}}\rho\nabla w\cdot\nabla[(\frac{n}{2}-\frac{|y|^{2}}{4})(\nabla w\cdot y_{0})]dy\notag\\
&-\frac{1}{p-1}\int_{\mathbb{R}^{n}}w(\frac{n}{2}-\frac{|y|^{2}}{4})(\nabla w\cdot y_{0})\rho dy\notag\\
&+\int_{\mathbb{R}^{n}}|w|^{p-1}w(\frac{n}{2}-\frac{|y|^{2}}{4})(\nabla w\cdot y_{0})\rho dy\notag\\
=&-\frac{1}{4(p-1)}\int_{\mathbb{R}^{n}}w^{2}(y\cdot y_{0})\rho  dy+\frac{1}{2(p+1)}\int_{\mathbb{R}^{n}} |w|^{p+1}(y\cdot y_{0})\rho dy\\
&-\frac{1}{4(p-1)}\int_{\mathbb{R}^{n}}(\frac{n}{2}-\frac{|y|^{2}}{4})w^{2}(y\cdot y_{0})\rho dy\notag\\
&+\frac{1}{2(p+1)}\int_{\mathbb{R}^{n}}(\frac{n}{2}-\frac{|y|^{2}}{4})|w|^{p+1}(y\cdot y_{0})\rho dy\notag\\
&+\frac{1}{2}\int_{\mathbb{R}^{n}}(\nabla w\cdot y)(\nabla w\cdot y_{0}) \rho dy-\frac{1}{4}\int_{\mathbb{R}^{n}}|\nabla w|^{2}(y\cdot y_{0}) \rho dy
\notag\\
&-\frac{1}{4}\int_{\mathbb{R}^{n}}|\nabla w|^{2}(y\cdot y_{0})(\frac{n}{2}-\frac{|y|^{2}}{4})\rho dy\notag.
\end{align}
Multiplying   both sides of \eqref{SC1} by $(\nabla w\cdot y_{0})(y\cdot y_{0})\rho$
and   integrating by parts, we obtain
\begin{align}\label{24-5(1)}
0=&\int_{\mathbb{R}^{n}}[\frac{1}{\rho}{\rm div}(\rho\nabla w)-\frac{1}{p-1}w+|w|^{p-1}w](\nabla w\cdot y_{0})(y\cdot y_{0})\rho dy\notag\\
=&-\int_{\mathbb{R}^{n}}\rho\nabla w\cdot\nabla[(\nabla w\cdot y_{0})(y\cdot y_{0})]dy+\int_{\mathbb{R}^{n}}|w|^{p-1}w(\nabla w\cdot y_{0})(y\cdot y_{0})\rho dy\notag\\
&-\frac{1}{p-1}\int_{\mathbb{R}^{n}}w(\nabla w\cdot y_{0})(y\cdot y_{0})\rho dy\notag\\
=&-\frac{1}{4(p-1)}\int_{\mathbb{R}^{n}}( y\cdot y_{0})(y\cdot y_{0}) w^{2}\rho dy+\frac{1}{2(p-1)}\int_{\mathbb{R}^{n}}w^{2}|y_{0}|^{2}\rho dy\\
&+\frac{1}{2(p+1)}\int_{\mathbb{R}^{n}}(y\cdot y_{0})(y\cdot y_{0}) |w|^{p+1}\rho dy-\frac{1}{p+1}\int_{\mathbb{R}^{n}}|w|^{p+1}|y_{0}|^{2}\rho dy\notag\\
&-\int_{\mathbb{R}^{n}}(\nabla w\cdot y_{0})(\nabla w\cdot y_{0})\rho dy-\frac{1}{4}\int_{\mathbb{R}^{n}}|\nabla w|^{2}(y\cdot y_{0})(y\cdot y_{0})\rho dy\notag\\
&+\frac{1}{2}\int_{\mathbb{R}^{n}}|\nabla w|^{2}|y_{0}|^{2}\rho dy\notag.
\end{align}
Multiplying   both sides of \eqref{SC1} by $(\frac{n}{2}-\frac{|y|^{2}}{4})(\nabla w\cdot\nabla \rho)$ and   integrating by parts, we  get  
\begin{align}\label{24-5(2)}
0=&\int_{\mathbb{R}^{n}}[\frac{1}{\rho}{\rm div}(\rho\nabla w)-\frac{1}{p-1}w+|w|^{p-1}w](\frac{n}{2}-\frac{|y|^{2}}{4})(\nabla w\cdot\nabla \rho)dy\notag\\
=&-\int_{\mathbb{R}^{n}}\rho\nabla w\cdot\nabla[(\frac{n}{2}-\frac{|y|^{2}}{4})(\nabla \log\rho\cdot\nabla w)]dy\notag\\
&-\frac{1}{2(p-1)}\int_{\mathbb{R}^{n}}(\frac{n}{2}-\frac{|y|^{2}}{4})(\nabla w^{2}\cdot\nabla\rho)dy\notag\\
&+\frac{1}{p+1}\int_{\mathbb{R}^{n}}(\frac{n}{2}-\frac{|y|^{2}}{4})(\nabla |w|^{p+1}\cdot\nabla\rho)dy\notag\\
=&\frac{1}{8(p-1)}\int_{\mathbb{R}^{n}}w^{2}|y|^{2}\rho dy-\frac{1}{4(p+1)}\int_{\mathbb{R}^{n}}|w|^{p+1}|y|^{2}\rho dy\notag\\
&-\frac{1}{2(p-1)}\int_{\mathbb{R}^{n}}(\frac{n}{2}-\frac{|y|^{2}}{4})(\frac{n}{2}-\frac{|y|^{2}}{4})w^{2}\rho dy\\
&+\frac{1}{p+1}\int_{\mathbb{R}^{n}}(\frac{n}{2}-\frac{|y|^{2}}{4})(\frac{n}{2}-\frac{|y|^{2}}{4})|w|^{p+1}\rho dy\notag\\
&+\frac{n}{4}\int_{\mathbb{R}^{n}}|\nabla w|^{2}\rho dy-\frac{1}{4}\int_{\mathbb{R}^{n}}(\nabla w\cdot y)(\nabla w\cdot y)\rho dy\notag\\
&-\frac{1}{2}\int_{\mathbb{R}^{n}}(\frac{n}{2}-\frac{|y|^{2}}{4})(\frac{n}{2}-\frac{|y|^{2}}{4})|\nabla w|^{2}\rho dy\notag.
\end{align}
Combining \eqref{24-5}, \eqref{24-5(1)} and \eqref{24-5(2)}, we have
\begin{align}\label{24-5(3)}
&\frac{1}{2}\int_{\mathbb{R}^{n}}|\nabla w|^{2} \rho[\frac{nh}{2}+\frac{y\cdot y_{0}}{2}-\frac{h|y|^{2}}{4}]^{2}dy\notag\\
&-\frac{1}{p+1}\int_{\mathbb{R}^{n}}|w|^{p+1}\rho[\frac{nh}{2}+\frac{y\cdot y_{0}}{2}-\frac{h|y|^{2}}{4}]^{2}dy\notag\\
&+\frac{1}{2(p-1)}\int_{\mathbb{R}^{n}}w^{2}\rho[\frac{nh}{2}+\frac{y\cdot y_{0}}{2}-\frac{h|y|^{2}}{4}]^{2}dy\notag\\
=&\frac{h^{2}}{8(p-1)}\int_{\mathbb{R}^{n}}w^{2}|y|^{2}\rho dy-\frac{h^{2}}{4(p+1)}\int_{\mathbb{R}^{n}}|w|^{p+1}|y|^{2}\rho dy\notag\\
&+\frac{nh^{2}}{4}\int_{\mathbb{R}^{n}}|\nabla w|^{2}\rho dy-\frac{h^{2}}{4}\int_{\mathbb{R}^{n}}(\nabla w\cdot y)(\nabla w\cdot y)\rho dy\notag\\
&+\frac{1}{4(p-1)}\int_{\mathbb{R}^{n}}w^{2}|y_{0}|^{2}\rho dy-\frac{1}{2(p+1)}\int_{\mathbb{R}^{n}}|w|^{p+1}|y_{0}|^{2}\rho dy\\
&-\frac{1}{2}\int_{\mathbb{R}^{n}}(\nabla w\cdot y_{0})(\nabla w\cdot y_{0})\rho dy+\frac{1}{4}\int_{\mathbb{R}^{n}}|\nabla w|^{2}|y_{0}|^{2}\rho dy\notag\\
&-\frac{h}{2(p-1)}\int_{\mathbb{R}^{n}}w^{2}(y\cdot y_{0})\rho  dy+\frac{h}{p+1}\int_{\mathbb{R}^{n}} |w|^{p+1}(y\cdot y_{0})\rho dy\notag\\
&+h\int_{\mathbb{R}^{n}}(\nabla w\cdot y)(\nabla w\cdot y_{0}) \rho dy-\frac{h}{2}\int_{\mathbb{R}^{n}}|\nabla w|^{2}(y\cdot y_{0})\rho dy\notag\\
=&-\frac{1}{2}\int_{\mathbb{R}^{n}}(\nabla w\cdot y_{0})(\nabla w\cdot y_{0})\rho dy-\frac{h^{2}}{4}\int_{\mathbb{R}^{n}}(\nabla w\cdot y)(\nabla w\cdot y)\rho dy\notag\\
&+h\int_{\mathbb{R}^{n}}(\nabla w\cdot y)(\nabla w\cdot y_{0}) \rho dy+\frac{h^{2}}{2}\int_{\mathbb{R}^{n}}|\nabla w|^{2}[\frac{n}{2}-\frac{|y|^{2}}{4}]\rho dy\notag\\
&+\frac{1}{2}\int_{\mathbb{R}^{n}}|\nabla w|^{2}[\frac{|y_{0}|^{2}}{2}+\frac{|y|^{2}}{4}]\rho dy+\frac{1}{2(p-1)}\int_{\mathbb{R}^{n}}w^{2}[\frac{|y_{0}|^{2}}{2}+\frac{|y|^{2}}{4}]\rho dy\notag\\
&-\frac{1}{p+1}\int_{\mathbb{R}^{n}}|w|^{p+1}[\frac{|y_{0}|^{2}}{2}+\frac{|y|^{2}}{4}]\rho dy\notag,
\end{align}
where we have applied \eqref{Testfunction}. 

Then
\begin{align}
&\frac{1}{2}\int_{\mathbb{R}^{n}}|\nabla w|^{2} \rho\{[\frac{nh}{2}+\frac{y\cdot y_{0}}{2}-\frac{h|y|^{2}}{4}]^{2}+\frac{nh^{2}}{2}-\frac{|y_{0}|^{2}}{2}-\frac{h^{2}|y|^{2}}{2}\}dy\notag\\
&+\frac{1}{2(p-1)}\int_{\mathbb{R}^{n}}w^{2}\rho\{[\frac{nh}{2}+\frac{y\cdot y_{0}}{2}-\frac{h|y|^{2}}{4}]^{2}+\frac{nh^{2}}{2}-\frac{|y_{0}|^{2}}{2}-\frac{h^{2}|y|^{2}}{2}\}dy\notag\\
&-\frac{1}{p+1}\int_{\mathbb{R}^{n}}|w|^{p+1}\rho\{[\frac{nh}{2}+\frac{y\cdot y_{0}}{2}-\frac{h|y|^{2}}{4}]^{2}+\frac{nh^{2}}{2}-\frac{|y_{0}|^{2}}{2}-\frac{h^{2}|y|^{2}}{2}\}dy\notag\\
=&h^{2}\int_{\mathbb{R}^{n}}|\nabla w|^{2}[\frac{n}{2}-\frac{|y|^{2}}{4}]\rho dy-\frac{h^{2}}{p+1}\int_{\mathbb{R}^{n}}|w|^{p+1}[\frac{n}{2}-\frac{|y|^{2}}{4}]\rho dy\notag\\
&+\frac{h^{2}}{2(p-1)}\int_{\mathbb{R}^{n}}w^{2}[\frac{n}{2}-\frac{|y|^{2}}{4}]\rho dy-\frac{h^{2}}{4}\int_{\mathbb{R}^{n}}(\nabla w\cdot y)(\nabla w\cdot y)\rho dy\notag\\
&-\frac{1}{2}\int_{\mathbb{R}^{n}}(\nabla w\cdot y_{0})(\nabla w\cdot y_{0})\rho dy+h\int_{\mathbb{R}^{n}}(\nabla w\cdot y)(\nabla w\cdot y_{0})\rho dy\notag.
\end{align}
Plugging this into \eqref{Thefirstsimplication}, we have
\begin{align}\label{Thesecondsimplication}
&\frac{\partial^{2}}{\partial s^2} (F_{x(s), t(s)}(w+s\phi))|_{s=0}\notag\\
&=\int_{\mathbb{R}^{n}}|\nabla\phi|^{2}\rho dy-p\int_{\mathbb{R}^{n}}|w|^{p-1}\phi^{2}\rho dy+\frac{1}{p-1}\int_{\mathbb{R}^{n}}\phi^{2}\rho dy\notag\\
&+h\int_{\mathbb{R}^{n}}(\frac{2}{p-1}w+\nabla w\cdot y)\phi\rho dy-\int_{\mathbb{R}^{n}}(\nabla w\cdot y_{0})\phi\rho dy\notag\\
&-\frac{(p+1)h}{p-1}\frac{1}{2}\int_{\mathbb{R}^{n}}|\nabla w|^{2}\rho[\frac{nh}{2}-\frac{h|y|^{2}}{4}]dy\notag\\
&+\frac{(p+1)h}{p-1}\frac{1}{p+1}\int_{\mathbb{R}^{n}}|w|^{p+1}\rho[\frac{nh}{2}-\frac{h|y|^{2}}{4}]dy\notag\\
&-\frac{(p+1)h}{p-1}\frac{1}{2(p-1)}\int_{\mathbb{R}^{n}}w^{2}\rho[\frac{nh}{2}-\frac{h|y|^{2}}{4}]dy \\
&+\frac{h}{p-1}\int_{\mathbb{R}^{n}}w^{2}\rho[\frac{nh}{2}+\frac{y\cdot y_{0}}{2}-\frac{h|y|^{2}}{4}]dy-\frac{h^{2}}{(p-1)^{2}}\int_{\mathbb{R}^{n}}w^{2}\rho dy\notag\\
&+\frac{h^{2}}{2}\int_{\mathbb{R}^{n}}|\nabla w|^{2}[\frac{n}{2}-\frac{|y|^{2}}{4}] \rho dy-\frac{h^{2}}{4}\int_{\mathbb{R}^{n}}(\nabla w\cdot y)(\nabla w\cdot y)\rho dy\notag\\
&-\frac{1}{2}\int_{\mathbb{R}^{n}}(\nabla w\cdot y_{0})(\nabla w\cdot y_{0})\rho dy+h\int_{\mathbb{R}^{n}}(\nabla w\cdot y)(\nabla w\cdot y_{0})\rho dy\notag.
\end{align}
Multiplying   both sides of \eqref{1} by $(\frac{n}{2}-\frac{|y|^{2}}{4})w\rho$ and   integrating by parts, we   get  
\begin{align}
0=&\int_{\mathbb{R}^{n}}[\frac{1}{\rho}{\rm div}(\rho\nabla w)-\frac{1}{p-1}w+|w|^{p-1}w](\frac{n}{2}-\frac{|y|^{2}}{4})w\rho dy\notag\\
=&-\int_{\mathbb{R}^{n}}\rho\nabla w\nabla [(\frac{n}{2}-\frac{|y|^{2}}{4})w]-\frac{1}{p-1}\int_{\mathbb{R}^{n}}w^{2}(\frac{n}{2}-\frac{|y|^{2}}{4})\rho dy\notag\\
&+\int_{\mathbb{R}^{n}}|w|^{p+1}(\frac{n}{2}-\frac{|y|^{2}}{4})\rho dy\notag\\
=&-\int_{\mathbb{R}^{n}}|\nabla w|^{2}(\frac{n}{2}-\frac{|y|^{2}}{4})\rho dy+\frac{1}{2}\int_{\mathbb{R}^{n}}w(\nabla w\cdot y)\rho dy\notag\\
&-\frac{1}{p-1}\int_{\mathbb{R}^{n}}w^{2}(\frac{n}{2}-\frac{|y|^{2}}{4})\rho dy+\int_{\mathbb{R}^{n}}|w|^{p+1}(\frac{n}{2}-\frac{|y|^{2}}{4})\rho dy\notag\\
=&-\int_{\mathbb{R}^{n}}|\nabla w|^{2}(\frac{n}{2}-\frac{|y|^{2}}{4})\rho+\frac{1}{4}\int_{\mathbb{R}^{n}}(\nabla w^{2}\cdot y)\rho dy\notag\\
&+\frac{1}{p-1}\int_{\mathbb{R}^{n}}w^{2}(\frac{n}{2}-\frac{|y|^{2}}{4})\rho dy+\int_{\mathbb{R}^{n}}|w|^{p+1}(\frac{n}{2}-\frac{|y|^{2}}{4})\rho dy\notag\\
=&-\int_{\mathbb{R}^{n}}|\nabla w|^{2}(\frac{n}{2}-\frac{|y|^{2}}{4})\rho dy-\frac{1}{2}\int_{\mathbb{R}^{n}}w^{2}(\frac{n}{2}-\frac{|y|^{2}}{4})\rho dy\notag\\
&-\frac{1}{p-1}\int_{\mathbb{R}^{n}}w^{2}(\frac{n}{2}-\frac{|y|^{2}}{4})\rho dy+\int_{\mathbb{R}^{n}}|w|^{p+1}(\frac{n}{2}-\frac{|y|^{2}}{4})\rho dy\notag.
\end{align}
Therefore
\begin{align}\label{25-1}
0=&-\frac{1}{p+1}\int_{\mathbb{R}^{n}}|\nabla w|^{2}(\frac{n}{2}-\frac{|y|^{2}}{4})\rho dy-\frac{1}{2(p-1)}\int_{\mathbb{R}^{n}}w^{2}(\frac{n}{2}-\frac{|y|^{2}}{4})\rho dy\\
&+\frac{1}{p+1}\int_{\mathbb{R}^{n}}|w|^{p+1}(\frac{n}{2}-\frac{|y|^{2}}{4})\rho dy\notag.
\end{align}
By \eqref{10} and \eqref{25-1}, we have
\begin{align}\label{25-2}
0=\frac{1}{p-1}\int_{\mathbb{R}^{n}}|\nabla w|^{2}\rho dy-\frac{1}{p+1}\int_{\mathbb{R}^{n}}|\nabla w|^{2}(\frac{n}{2}-\frac{|y|^{2}}{4})\rho dy.
\end{align}
Substituting \eqref{25-2} and \eqref{10} into \eqref{Thesecondsimplication} gives
\begin{align}\label{Thethirdsimplication}
&\frac{\partial^{2}}{\partial s^2} (F_{x(s), t(s)}(w+s\phi))|_{s=0}\notag\\
&=\int_{\mathbb{R}^{n}}|\nabla\phi|^{2}\rho dy-p\int_{\mathbb{R}^{n}}|w|^{p-1}\phi^{2}\rho dy+\frac{1}{p-1}\int_{\mathbb{R}^{n}}\phi^{2}\rho dy\notag\\
&+h\int_{\mathbb{R}^{n}}(\frac{2}{p-1}w+\nabla w\cdot y)\phi\rho dy-\int_{\mathbb{R}^{n}}(\nabla w\cdot y_{0})\phi\rho dy\\
&+\frac{h}{p-1}\int_{\mathbb{R}^{n}}w^{2}\rho[\frac{nh}{2}+\frac{y\cdot y_{0}}{2}-\frac{h|y|^{2}}{4}]dy-\frac{h^{2}}{(p-1)^{2}}\int_{\mathbb{R}^{n}}w^{2}\rho dy\notag\\
&-\frac{h^{2}}{4}\int_{\mathbb{R}^{n}}(\nabla w\cdot y)(\nabla w\cdot y)\rho dy-\frac{1}{2}\int_{\mathbb{R}^{n}}(\nabla w\cdot y_{0})(\nabla w\cdot y_{0})\rho dy\notag\\
&+h\int_{\mathbb{R}^{n}}(\nabla w\cdot y)(\nabla w\cdot y_{0})\rho dy\notag.
\end{align}
Let $L$ be the linear operator defined by
\[L\psi=\Delta\psi-\frac{y}{2}\cdot\nabla\psi-\frac{1}{p-1}\psi+p|w|^{p-1}\psi.\]
We know from Lemma \ref{lemeig} that $2/(p-1)w+y\cdot\nabla w$ is an eigenfunction of $L$ associated
to the eigenvalue $-1$ and $w_{i}, i=1, 2,\cdots, n$ are eigenfunctions of $L$ associated to the
eigenvalue $-1/2$. Since $L$ is self-adjoint,  for any $y_{0}\in\mathbb{R}^{n}$,
\begin{equation}\label{25-x2}
\int_{\mathbb{R}^{n}}\left(\frac{2}{p-1}w+y\cdot\nabla w\right)(\nabla w\cdot y_{0})\rho dy=0.
\end{equation}
By \eqref{25-x2}, we have
\begin{align}\label{25-x3}
&\int_{\mathbb{R}^{n}}(\nabla w\cdot y)(\nabla w\cdot y_{0})\rho dy\notag\\
=&-\frac{2}{p-1}\int_{\mathbb{R}^{n}}w(\nabla w\cdot y_{0})\rho dy\\
=&-\frac{1}{p-1}\int_{\mathbb{R}^{n}}(\nabla w^{2}\cdot y_{0})\rho dy\notag\\
=&-\frac{1}{2(p-1)}\int_{\mathbb{R}^{n}}w^{2}\rho (y\cdot y_{0})dy\notag.
\end{align}
Finally, we have
\begin{align}\label{25-x5}
&\frac{h^{2}}{p-1}\int_{\mathbb{R}^{n}}w^{2}\rho[\frac{n}{2}-\frac{|y|^{2}}{4}]dy\notag\\
=&-\frac{h^{2}}{p-1}\int_{\mathbb{R}^{n}}w^{2}\Delta\rho\notag\\
=&\frac{h^{2}}{p-1}\int_{\mathbb{R}^{n}}\nabla w^{2}\cdot\nabla \rho dy\\
=&-\frac{h^{2}}{p-1}\int_{\mathbb{R}^{n}}w (\nabla w\cdot y)\rho dy\notag.
\end{align}
Substituting \eqref{25-x3} and \eqref{25-x5} into \eqref{Thethirdsimplication} gives
\begin{align}\label{secvari}
&\frac{\partial^{2}}{\partial s^2} (F_{x(s), t(s)}(w+s\phi))|_{s=0}\notag\\
&=\int_{\mathbb{R}^{n}}|\nabla\phi|^{2}\rho dy-p\int_{\mathbb{R}^{n}}|w|^{p-1}\phi^{2}\rho dy+\frac{1}{p-1}\int_{\mathbb{R}^{n}}\phi^{2}\rho dy\notag\\
&+h\int_{\mathbb{R}^{n}}(\frac{2}{p-1}w+\nabla w\cdot y)\phi\rho dy-\int_{\mathbb{R}^{n}}(\nabla w\cdot y_{0})\phi\rho dy\\
&-\frac{h^{2}}{(p-1)^{2}}\int_{\mathbb{R}^{n}}w^{2}\rho dy-\frac{h^{2}}{p-1}\int_{\mathbb{R}^{n}}w(\nabla w\cdot y)\rho dy\notag\\
&-\frac{h^{2}}{4}\int_{\mathbb{R}^{n}}(\nabla w\cdot y)(\nabla w\cdot y)\rho dy-\frac{1}{2}\int_{\mathbb{R}^{n}}(\nabla w\cdot y_{0})(\nabla w\cdot y_{0})\rho dy\notag.
\end{align}
Since
\begin{align}
&-\frac{h^{2}}{(p-1)^{2}}\int_{\mathbb{R}^{n}}w^{2}\rho dy-\frac{h^{2}}{p-1}\int_{\mathbb{R}^{n}}w(\nabla w\cdot y)\rho dy\notag-\frac{h^{2}}{4}\int_{\mathbb{R}^{n}}(\nabla w\cdot y)(\nabla w\cdot y)\rho dy\notag\\
=&-h^{2}\int_{\mathbb{R}^{n}}(\frac{1}{p-1}w+\frac{y}{2}\cdot\nabla w)^{2}\rho dy\notag,
\end{align}
we get  \eqref{secondvariation1} with the help of   \eqref{secvari}.
\end{proof}

\section*{Appendix B}
In this appendix, our main objective is to prove \eqref{comparewithsingularsolution}.
\begin{proof}[Proof of \eqref{comparewithsingularsolution}]
By the definition of $\rho$, we have
\begin{align}\label{Energysef2}
&\int_{\mathbb{R}^{n}}|y|^{-\frac{2(p+1)}{p-1}}\rho dy\notag\\
=&(4\pi)^{-\frac{n}{2}}\int_{\mathbb{R}^{n}}|y|^{-\frac{2(p+1)}{p-1}}e^{-\frac{|y|^{2}}{4}}dy\\
=&(4\pi)^{-\frac{n}{2}}\omega_{n-1}\int_{0}^{+\infty}r^{n-1-\frac{2(p+1)}{p-1}}e^{-\frac{r^{2}}{4}}dr\notag,
\end{align}
where $\omega_{n-1}$ is the area of the unit sphere $\mathbb{S}^{n-1}$ in $\mathbb{R}^{n}$. Recall that we have assumed $p>(n+2)/(n-2)$, so the above integral is well defined. Let $r=2\sqrt{s}$. Then $dr=s^{-1/2}ds$ and
\begin{align}
E(w)=&c(n, p)\int_{0}^{+\infty}r^{n-1-\frac{2(p+1)}{p-1}}e^{-\frac{r^{2}}{4}}dr\notag\\
=&c(n, p)\int_{0}^{+\infty}s^{-\frac{1}{2}}(2\sqrt{s})^{n-1-\frac{2(p+1)}{p-1}}e^{-s}ds\notag\\
=&2^{n-3-\frac{4}{p-1}}c(n, p)\int_{0}^{\infty}s^{\frac{n-4-\frac{4}{p-1}}{2}}e^{-s}ds\notag\\
=&2^{n-3-\frac{4}{p-1}}c(n, p)\Gamma(\frac{n-2-\frac{4}{p-1}}{2})\notag,
\end{align}
where
\[c(n, p)=\left(\frac{1}{2}-\frac{1}{p+1}\right)\left[\frac{2}{p-1}\left(n-2-\frac{2}{p-1}\right)\right]^{\frac{p+1}{p-1}}(4\pi)^{-\frac{n}{2}}\omega_{n-1}\]
and
\[\Gamma(\tau)=\int_{0}^{+\infty}s^{\tau-1}e^{-s}ds\]
is the  $\Gamma$-function. Recall that the area of the unit sphere is
\[\omega_{n-1}=\frac{2\pi^{\frac{n}{2}}}{\Gamma(\frac{n}{2})}.\]
Then
\begin{equation}\label{Energysef5}
\begin{aligned}
E(w)=&2^{n-3-\frac{4}{p-1}}c(n, p)\Gamma\left(\frac{n-2-\frac{4}{p-1}}{2}\right)\\
=&2^{-2-\frac{4}{p-1}}\left(\frac{1}{2}-\frac{1}{p+1}\right)\left[\frac{2}{p-1}\left(n-2-\frac{2}{p-1}\right)\right]^{\frac{p+1}{p-1}}\frac{\Gamma\left(\frac{n-2-\frac{4}{p-1}}{2}\right)}{\Gamma\left(\frac{n}{2}\right)}.
\end{aligned}
\end{equation}
In order that Lemma \ref{twodiemsionenergy} holds, we need only to show
\[2^{-2-\frac{4}{p-1}}\left[\frac{2}{p-1}\left(n-2-\frac{2}{p-1}\right)\right]^{\frac{p+1}{p-1}}\frac{\Gamma\left(\frac{n-2-\frac{4}{p-1}}{2}\right)}{\Gamma\left(\frac{n}{2}\right)}>(\frac{1}{p-1})^{\frac{p+1}{p-1}}.\]
This is equivalent to
\begin{equation}\label{appendix2newobjective}
\left(\frac{n-2}{2}-\frac{1}{p-1}\right)^{\frac{p+1}{p-1}}\frac{\Gamma\left(\frac{n-2-\frac{4}{p-1}}{2}\right)}{\Gamma\left(\frac{n}{2}\right)}>1.
\end{equation}
To obtain \eqref{appendix2newobjective}, we set $\alpha=2/(p-1)$. Using $p>(n+2)/(n-2)$ again, we have $\alpha\in (0, (n-2)/2)$. 
Let
\[x=\frac{n}{2},\quad f(x)=\frac{\Gamma(x-1-\alpha)}{\Gamma(x)}\left(x-1-\frac{\alpha}{2}\right)^{1+\alpha}\]
and
\[\phi(x)=\log f(x)=\log\Gamma(x-1-\alpha)-\log\Gamma(x)+(1+\alpha)\log\left(x-1-\frac{\alpha}{2}\right).\]
Then
\begin{equation}\label{phifirstorderderivative}
\phi^\prime(x)=\frac{\Gamma^\prime(x-1-\alpha)}{\Gamma(x-1-\alpha)}-\frac{\Gamma^\prime(x)}{\Gamma(x)}+\frac{1+\alpha}{x-1-\frac{\alpha}{2}}
\end{equation}
and
\[
\begin{aligned}
\phi^{\prime\prime}(x)=&\left(\frac{\Gamma^\prime(x-1-\alpha)}{\Gamma(x-1-\alpha)}\right)^\prime-\left(\frac{\Gamma^\prime(x)}{\Gamma(x)}\right)^\prime-\frac{1+\alpha}{\left(x-1-\frac{\alpha}{2}\right)^{2}}\\
=&\sum^{\infty}_{l=0}(x-1-\alpha+l)^{-2}-\sum^{\infty}_{l=0}(x+l)^{-2}-\frac{1+\alpha}{\left(x-1-\frac{\alpha}{2}\right)^{2}}
\end{aligned}\]
Notice that for any $\tau>1$,
\[\left(\tau-1-\frac{\alpha}{2}\right)\left(\tau+\frac{\alpha}{2}\right)\leq\tau(\tau-1),\]
so
\[
\begin{aligned}
&\frac{1}{1+\alpha}\left[\frac{1}{\left(\tau-1-\frac{\alpha}{2}\right)^{2}}-\frac{1}{\left(\tau+\frac{\alpha}{2}\right)^{2}}\right]-\left[\frac{1}{(\tau-1)^{2}}-\frac{1}{\tau^{2}}\right]\\
=&\frac{1}{1+\alpha}\frac{(1+\alpha)(2\tau-1)}{\left(\tau-1-\frac{\alpha}{2}\right)^{2}\left(\tau+\frac{\alpha}{2}\right)^{2}}-\frac{2\tau-1}{\tau^{2}(\tau-1)^{2}}\\
=&(2\tau-1)\left[\frac{1}{\left(\tau-1-\frac{\alpha}{2})^{2}(\tau+\frac{\alpha}{2}\right)^{2}}-\frac{1}{\tau^{2}(\tau-1)^{2}}\right]\\
\geq & 0.
\end{aligned}\]
Take $\tau=x+l-\alpha/2$ in the above inequality, which gives
\[
\phi^{\prime\prime}(x)\geq(1+\alpha)\sum_{l=0}^{\infty}\left[\frac{1}{(x-1+l-\alpha/2)^{2}}-\frac{1}{(x+l-\alpha/2)^{2}}\right]-\frac{1+\alpha}{(x-1-\frac{\alpha}{2})^{2}}=0.\]
Thus   $\phi$ is convex. By the mean value theorem, we have
\[
(1+\alpha)\left[\frac{1}{x-1-\frac{\alpha}{2}}-\sum_{l=0}^{\infty}(x+l)^{-2}\right]\leq\phi^\prime(x)\leq(1+\alpha)\left[\frac{1}{x-1-\frac{\alpha}{2}}-\sum_{l=0}^{\infty}(x-1-\alpha+l)^{-2}\right].\]
Hence
$\lim_{x\rightarrow\infty}\phi'(x)=0$; and for any $\alpha\in (0, (n-2)/2)$, $\phi$ is decreasing in $[3/2, +\infty)$. Moreover, we see from \eqref{phifirstorderderivative}  that $\phi^\prime(x)=O(|x|^{-2})$ as $x\rightarrow\infty$. Thus there exists a constant $c_{0}$ such that $\lim_{x\rightarrow\infty}\phi(x)=c_{0}$.  By the Stirling's formula (see \cite{Magnus}), we have for $m=1, 2, \cdots$,
\[\log\Gamma(x)=\left(x-\frac{1}{2}\right)\log x-x+\frac{1}{2}\log(2\pi)+\sum_{l=1}^{m}\frac{B_{2l}}{2l(2l-1)}x^{-2l+1}+O(x^{-2m-1}),\]
where $B_{2l}$ are the Bernoulli numbers.
Therefore, as $x\rightarrow\infty$,
\[
\begin{aligned}
\phi(x)=&\log\Gamma(x-1-\alpha)-\log\Gamma(x)+(1+\alpha)\log\left(x-1-\frac{\alpha}{2}\right)\\
=&(x-1-\alpha-\frac{1}{2})\log(x-1-\alpha)-(x-1-\alpha)\\
&-(x-\frac{1}{2})\log x+x+(1+\alpha)\log\left(x-1-\frac{\alpha}{2}\right)+O(|x|^{-1})\\
=&(x-\frac{1}{2})\log\left(1-\frac{1+\alpha}{x}\right)+(1+\alpha)\left[\log\left(1+\frac{\alpha}{2(x-1-\alpha)}\right)+1\right]+O(|x|^{-1}).
\end{aligned}\]
Because
\[\lim_{x\to\infty}\left(x-\frac{1}{2}\right)\log\left(1-\frac{1+\alpha}{x}\right)=-(1+\alpha),\]
we get
\[\phi(x)=O(|x|^{-1}),\quad\text{as}~x\to+\infty.\]
This then implies that for any $\alpha\in(0, (n-2)/2)$, if $x>1$, then $\phi(x)>0$.
\end{proof}

\bibliographystyle{plain}
\bibliography{Fujita}

\end{document}